\documentclass[10pt,a4paper]{article}
\usepackage{amsmath,amssymb,amsthm,graphicx,hyperref}
\usepackage{color}
\newtheorem{theorem}{Theorem}[section]
\newtheorem{lemma}[theorem]{Lemma}

\theoremstyle{definition}

\newtheorem{remark}[theorem]{Remark}
\newtheorem{example}[theorem]{Example}
\DeclareMathOperator{\sgn}{sgn}
\DeclareMathOperator{\Skellam}{Skellam}
\newcommand{\I}{\mathrm{i}}

\numberwithin{equation}{section}

\begin{document}
\title{\sc Initial-boundary value problems for linear\\ lattice differential equations}

\makeatletter
\let\@fnsymbol\@arabic
\makeatother
\author{Andreas Chatziafratis\thanks{Waseda Research Institute for Science and Engineering, Waseda University, Tokyo, Japan \& Department of Mathematics, National and Kapodistrian University of Athens, Athens, Greece. E-mail: {\tt chatziafrati@math.uoa.gr}, ORCID  0000-0002-3075-0924}, Spyridon Kamvissis\thanks{Mathematics Building, University of Crete, GR--700 13 Voutes, Greece
\& Institute of Applied and Computational Mathematics, FORTH,
GR--711 10 Voutes, Greece. E-mail: {\tt spyros@tem.uoc.gr}, ORCID 0000-0002-3952-3429}, and Antonín Slavík\thanks{Charles University, Faculty of Mathematics and Physics, Prague, Czech Republic. E-mail: {\tt slavik@karlin.mff.cuni.cz}, ORCID 0000-0003-3941-7375}}
\date{}
\maketitle
\allowdisplaybreaks

\begin{abstract}
We study initial-boundary value problems for linear
lattice differential equations on the half-line.
We first give a systematic account of fundamental solutions
for equations with general finite-range interactions.
We then develop complementary approaches based on reflection,
the Laplace transform, and the unified transform method,
obtaining contour integral and Bessel-function representations
of solutions. For the semidiscrete diffusion equation, we justify
these representations for continuous boundary data and initial
data of polynomial growth, establish uniqueness in the
corresponding growth class, and derive spatial estimates
and long-time properties, including asymptotic periodicity.
We also treat nonhomogeneous equations, Robin boundary
conditions, nonsymmetric coefficients, and a class of
time-dependent coefficients. Finally, higher-order examples
illustrate how the unified transform approach extends
beyond second-order equations.


\end{abstract}

\section{Introduction}

Lattice differential equations describe systems whose spatial
variable is discrete while time remains continuous.
They arise both as models of interacting discrete systems
and as spatial discretizations of partial differential
equations. On the whole lattice, translation invariance
makes fundamental solutions and convolution formulas
natural tools for studying initial-value problems.
On a half-line, the presence of a boundary introduces
additional questions: how to incorporate prescribed
boundary values into an explicit solution formula, how
to justify that formula under weak assumptions on the
data, and how the boundary influences the long-time
behavior of the solution.

This paper presents a systematic study of initial-boundary
value problems for linear lattice differential equations
on a half-line. Our main model is the semidiscrete diffusion
equation
\begin{equation}\label{main-problem-1}
\frac{\partial u}{\partial t}(x,t)
=
a\bigl[u(x+1,t)-2u(x,t)+u(x-1,t)\bigr],
\qquad x\in\mathbb N,\quad t\geq0,
\end{equation}
where $a>0$, subject to
\begin{equation}\label{main-problem-2}
u(0,t)=g(t),\qquad t\geq0,
\qquad
u(x,0)=u_x,\qquad x\in\mathbb N.
\end{equation}
Here $g\colon[0,\infty)\to\mathbb R$ is continuous.
We initially consider bounded initial sequences and
subsequently allow initial data of polynomial growth.
We also study forcing terms and Robin-type boundary
conditions. Our treatment of second-order equations
extends to the nonsymmetric, nonhomogeneous problem
\[
\frac{\partial u}{\partial t}(x,t)
=
\alpha(t)u(x-1,t)+\beta(t)u(x,t)
+\gamma(t)u(x+1,t)+f(x,t),
\qquad x\in\mathbb N,\quad t\geq0,
\]
with prescribed initial and Dirichlet boundary data.
Here the coefficients are continuous and may be
complex-valued, and we assume
$\gamma(t)=\lambda\alpha(t)$ for a fixed nonzero constant
$\lambda$. This proportionality preserves the 
symmetry used in the unified transform method.

A simple physical interpretation of \eqref{main-problem-1}
is provided by an infinite array of identical tanks
$T_x$, $x\in\mathbb N_0$, connected by pipes; see
Figure~\ref{fig1}.
Each tank contains a fixed volume of perfectly mixed
brine, and neighboring tanks exchange liquid at equal
constant rates. If $u(x,t)$ is the amount of salt in
tank $x$ at time~$t$, then its evolution is governed
by \eqref{main-problem-1}, with $a$ determined by the
exchange rate and the tank volume.
The conditions \eqref{main-problem-2} prescribe the
initial salt distribution and describe an external
device that maintains the amount of salt in tank $0$
at the prescribed value $g(t)$.

\begin{figure}[!h]
    \centering
    \includegraphics[width=10cm]{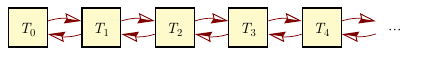}
    \caption{An infinite array of tanks connected by pipes;
    the amount of salt in tank $x$ at time $t$ is $u(x,t)$.}
    \label{fig1}
\end{figure}

An important objective of this paper is to determine how
the initial data, boundary data, and forcing influence
the long-time behavior of solutions.
For the semidiscrete diffusion equation with bounded initial
data, we show that convergence of the boundary data to
a constant implies convergence of the solution to the
same constant at every fixed lattice point.
This conclusion remains valid in the presence of a
forcing term that tends to zero at each lattice point
and admits a summable spatial majorant.
More generally, we bound the limit inferior and limit
superior of the solution by those of the boundary data.

We also establish asymptotic periodicity for periodic
boundary data and forcing: under suitable assumptions,
\[
u(x,t+T)-u(x,t)\longrightarrow0
\qquad\text{as }t\to\infty
\]
for every fixed $x\in\mathbb N$, where $T$ is the common
period of the prescribed data.
We prove this property for both symmetric and
nonsymmetric coefficients, identifying conditions on
the coefficients and data that ensure it.
In both cases, stronger spatial moment
assumptions yield explicit leading terms and remainder
estimates. 

Previous work on linear discrete-space evolution equations
includes explicit solution formulas, existence and
uniqueness results, and asymptotic analysis; see, for
example, \cite{FSS,SS2,SS,AS17,Slavik2020}.
For the half-line diffusion problem, local existence
for initial data in $c_0$ was investigated in
\cite{Oberta}.
Here we obtain explicit representations of global
classical solutions for bounded initial data and,
in the symmetric Dirichlet problem, for initial data
of polynomial growth. We prove uniqueness in the
corresponding growth classes and establish estimates
that describe the spatial behavior of the solutions.

An independent contribution is a systematic treatment
of fundamental solutions for the whole-line equation
\[
\frac{\partial u}{\partial t}(x,t)
=
\sum_{i=-m}^{m}a_i u(x-i,t),
\qquad x\in\mathbb Z,\quad t\geq0.
\]
Using generating functions and Fourier transforms,
we obtain representations for symmetric, antisymmetric,
and general coefficients.
The resulting Bessel-function formulas extend the
familiar nearest-neighbor kernels to equations involving
arbitrarily many spatial shifts.
We establish absolute summability estimates and use
the kernels to derive a Duhamel formula for nonhomogeneous
equations with forcing that is continuous at each
spatial point and uniformly bounded in space on compact
time intervals. These results provide the whole-line
foundation for the subsequent boundary-value analysis.

A central feature of the paper is the use of several
complementary methods for the same class of problems.
For symmetric diffusion, we derive the half-line solution
by reflection, by the Laplace transform, and by the
unified transform method, also known as the Fokas method.
This comparison connects the contour representations
arising from transform calculations with explicit
Bessel-kernel formulas.
We then rigorously justify the resulting solution formulas
for a broader class of data, allowing initial data of
polynomial growth and forcing terms that are bounded
in space, uniformly on compact time intervals.
The formulas also yield preservation of spatial
summability and decay, convergence results for
asymptotically constant boundary data, and asymptotic
periodicity under suitable periodicity and summability
assumptions.

The unified transform method was introduced and developed by Fokas
\cite{F1,F2,F3,F4,F5,F6,F7} as an approach to initial-boundary value
problems for partial differential equations.
It was initially developed for completely integrable
nonlinear equations such as the Korteweg--de Vries
and nonlinear Schr\"odinger equations.
Such equations can be formulated through the evolution
of a linear differential operator $L(t)$ satisfying
the Lax equation
\[
\frac{{\mathrm d}L}{{\mathrm d}t}=BL-LB,
\]
where $B$ is an auxiliary differential operator,
typically formally skew-adjoint.
For example, the inverse-scattering treatment of the
initial-value problem for the Korteweg--de Vries equation
centers on the associated Schr\"odinger operator $L$,
with $B$ determining the time evolution of the
scattering data.
For an initial-boundary value problem, the unified
transform uses spectral information associated with
both parts of the Lax pair.
Their interdependence makes the method a nontrivial
extension of the inverse-scattering approach.

Although initially proposed for nonlinear problems,
the method was soon recognized as useful for linear
equations as well.
Classical tools such as the Laplace and sine transforms
are especially effective for certain low-order equations
and simple geometries, but their application can become
cumbersome for higher-order equations or more general
boundary conditions.
The unified transform provides a systematic alternative
and has been applied to equations of arbitrary order,
to elliptic problems \cite{A}, and to problems with
variable coefficients and different domain geometries.
Its linear formulation has also clarified the role
of Lax pairs as nonlinear analogues of divergence
formulations for separable linear equations \cite{FS}.
Besides explicit solution representations, the method
provides a basis for effective numerical evaluation. For a survey of the unified transform method
for linear partial differential equations, including
problems on the half-line and finite interval and its
relationship with classical solution methods, we refer
to~\cite{FS}. Significant extensions and applications of the UTM have been developed in \cite{BL,FL}. For very recent advances around the UTM, we refer the reader to \cite{BCCK, CACF24, CGK24, CK25,CKS23, CLW27, CMKFA25,CO24,COT24}.

Semidiscrete versions of the unified transform method
have been developed in several earlier works.
Biondini and Hwang \cite{BH08} studied the discrete
linear Schr\"odinger equation and its integrable nonlinear
counterpart on the half-line, including Robin-type
boundary conditions for the linear equation.
Biondini and Wang \cite{BW10} developed a broader
framework for linear differential--difference equations,
with extensions to systems, forced problems, and fully
discrete equations. For problems on finite lattices,
see also Moon and Hwang \cite{MH19}.
Cisneros and Deconinck \cite{CD21} developed the
semidiscrete unified transform method with numerical
approximation of linear evolution partial differential
equations in view. Their approach yields exact
representations of the semidiscrete solutions and
addresses continuum limits, derivative boundary
conditions, and additional grid values introduced by
the spatial discretization.

Our treatment complements these developments by combining
reflection, Laplace-transform, and unified-transform
methods with a detailed analysis of the resulting
solution formulas. In particular, the half-line
transform analysis in \cite{BW10} assumes absolutely
summable initial data and continuous boundary data.
For the semidiscrete diffusion equation, we justify the
Bessel representation for initial data of polynomial
growth and continuous boundary data, and establish
uniqueness in the corresponding growth class.
We also treat forcing terms that are continuous in time
at each lattice point and bounded in space, uniformly
on compact time intervals. Beyond the construction and justification of the solution
formulas, we use them to establish the long-time results
described above, including the effect of decaying forcing,
asymptotic periodicity, and quantitative asymptotic
expansions under spatial moment assumptions. Together with our systematic
account of fundamental solutions for general finite-range
lattice equations, these results connect the explicit
representations with the existence, uniqueness, and
qualitative properties of solutions.
The paper is organized as follows.
Section~\ref{sec-2} develops the fundamental solutions, and
Section~\ref{sec-3} treats nonhomogeneous equations on the whole $\mathbb Z$.
Sections~\ref{sec-4} and~\ref{sec-5} establish symmetry properties,
uniqueness, and a superposition principle for
initial-boundary value problems.
Section~\ref{sec-6} studies symmetric diffusion by the three
methods described above and investigates boundary
conditions and asymptotic behavior.
Section~\ref{sec-7} treats nonsymmetric and time-dependent
coefficients.
Section~\ref{sec-8} develops a third-order unified-transform
example and outlines the corresponding algebraic
procedure for fourth-order and more general equations.
The higher-order discussion is intended to illustrate
the method; it does not constitute a general
well-posedness theory.

For convenience, Table~\ref{tab-kernel-notation} summarizes
the notation for the fundamental solutions and kernels
used throughout the paper. Their definitions are given
in the corresponding sections.

\begin{table}[ht]
\setlength{\belowcaptionskip}{8pt}
\centering
\caption{Notation for fundamental solutions and kernels.}
\label{tab-kernel-notation}
\begin{tabular}{ll}
\hline
Notation & Meaning\\
\hline
$G(x,t)$
& Whole-line fundamental solution for a general lattice equation\\
$K_x^{(a)}(t)$
& Whole-line diffusion kernel with diffusion coefficient $a>0$\\
$K_x(t)$
& Whole-line diffusion kernel with diffusion coefficient $1$\\
$P_{x,y}^{(a)}(t)$
& Dirichlet half-line diffusion kernel with diffusion coefficient $a>0$\\
$P_{x,y}(t)$
& Dirichlet half-line diffusion kernel with diffusion coefficient $1$\\
$Q_{x,y}(t)$
& Dirichlet kernel for constant nonsymmetric coefficients\\
$E_\ell(s,t)$
& Whole-line kernel for time-dependent coefficients\\
$Q_{x,y}(s,t)$
& Dirichlet kernel for time-dependent coefficients\\
$R^\kappa_{x,y}(t)$
& Robin kernel with boundary parameter $\kappa$\\
\hline
\end{tabular}
\end{table}
Explicit formulas for these kernels involve Bessel functions
and modified Bessel functions. We assume some basic familiarity
with these functions; their definitions and the properties used
below can be found in \cite{BCLO}.

\section{Fundamental solutions}\label{sec-2}

Consider the homogeneous equation
\begin{equation}\label{general-RHS}
	\frac{\partial u}{\partial t}(x,t)=\sum_{i=-m}^m a_iu(x-i,t),\quad x\in\mathbb Z,\quad t\ge 0,
\end{equation}
with $a_{-m},\dots,a_m\in\mathbb R$.
In this section, we focus on finding the fundamental solution, i.e., a  solution $G\colon\mathbb Z\times[0,\infty)\to\mathbb R$
for which the initial condition is
\[
G(x,0)=\begin{cases}
	1&\mbox{ if }x=0,\\
	0&\mbox{ if }x\ne 0.
\end{cases}
\]
In order to guarantee uniqueness, we require $G$ to be bounded on $\mathbb Z\times[0,T]$ for every $T>0$.

\subsection{The generating function approach}

We use the method of generating functions: Let $F(z,t)=\sum_{x\in\mathbb Z}G(x,t)z^x$. Then \eqref{general-RHS} yields
\[
\frac{\partial F}{\partial t}(z,t)=\sum_{i=-m}^m a_iz^{i} F(z,t),\quad F(z,0)=1,
\]
and therefore
\[
F(z,t)=\exp\left(\sum_{i=-m}^m a_iz^{i}t\right).
\]
An interesting consequence is that 
\[
\sum_{x\in\mathbb Z}G(x,t)=F(1,t)=\exp\left(\sum_{i=-m}^m a_i
t\right).
\]
In particular, $\lim_{x\to\pm\infty} G(x,t)=0$, and spatial sums are preserved (do not depend on $t$) if and only if $\sum_{i=-m}^m a_i=0$.

To get $G(x,t)$, we have to expand $F(z,t)$ in a Laurent series, and find the coefficient of $z^x$.

\subsection{Symmetric case}\label{section-symmetric-FS}

We begin with the symmetric case where $a_i=a_{-i}$ for all $i\in\{1,\dots,m\}$. 
Recall that
\begin{equation}\label{modified-bessel-expansion}
\exp\left(w\left(z+\frac{1}{z}\right)\right)=\sum_{x\in\mathbb Z}I_x(2w)z^x,
\end{equation}
where $I_x$ is the modified Bessel function of order $x$. Hence, we get
\begin{gather*}
F(z,t)={\mathrm e}^{a_0t}\exp\left(\sum_{i=1}^m a_i\left(\frac{1}{z^i}+z^i\right)t\right)
={\mathrm e}^{a_0t}\prod_{i=1}^m\exp\left(a_i\left(\frac{1}{z^i}+z^i\right)t\right)\
={\mathrm e}^{a_0t}\prod_{i=1}^m \sum_{x_i\in\mathbb Z}I_{x_i}(2a_it)z^{ix_i}.
\end{gather*}
We recover the fundamental solution $G$ as the coefficient of $z^x$ in $F(z,t)$, arriving at the following result.

\begin{theorem}\label{FS-even}
The fundamental solution of \eqref{general-RHS} where $a_i=a_{-i}$, $i\in\{1,\dots,m\}$, is
\[
G(x,t)={\mathrm e}^{a_0t}\sum_{\substack{x_1,\dots,x_m\in\mathbb Z,\\ 1x_1+2x_2+\cdots+mx_m=x}}I_{x_1}(2a_1t)I_{x_2}(2a_2t)\cdots I_{x_m}(2a_mt).
\]
\end{theorem} 

We will postpone the proof of convergence of the series until Theorem \ref{FS-general}.

For $m=1$, we get the formula
\begin{equation}\label{fs-sym-1}
G(x,t)={\mathrm e}^{a_0t}I_x(2a_1t),
\end{equation}
which appears in a variety of applications \cite{CJKS}. The derivation based on generating functions was already presented in \cite[p.~490]{SS2}. The formulas for $m\ge 2$ seem to be new; for instance, the choice $m=2$ leads to
\begin{equation}\label{fs-sym-2}
G(x,t)={\mathrm e}^{a_0t}\sum_{\substack{x_1,x_2\in\mathbb Z,\\ x_1+2x_2=x}}I_{x_1}(2a_1t)I_{x_2}(2a_2t)={\mathrm e}^{a_0t}\sum_{k\in\mathbb Z} I_k(2a_2t)I_{x-2k}(2a_1t)
\end{equation}
(to get the second equality, we observe that if $x_2=k$, then $x_1=x-2x_2=x-2k$).

\begin{example}
The derivative of a function $f$ of an even order $2m$ can be approximated using the central difference of order $2m$ given by
\[
\Delta^{2m}f(x)=\sum_{i=-m}^m (-1)^{m+i}\binom{2m}{m+i}f(x-i).
\]
The expression can be obtained using the binomial theorem from
\[
\Delta^{2m}=(E^{1/2}-E^{-1/2})^{2m},
\]
where $E$ is the forward shift operator defined by $(Ef)(x)=f(x+1)$, so $((E^{1/2}-E^{-1/2})f)(x)=f(x+\tfrac12)-f(x-\tfrac12)$.

We can now consider the equation
\begin{equation}\label{order-2m}
	\frac{\partial u}{\partial t}(x,t)=\alpha\Delta^{2m}u(x,t),\quad x\in\mathbb Z,\quad t\ge 0,
\end{equation}
where $\alpha\in\mathbb R$, and the difference on the right-hand side is with respect to $x$.
Equation \eqref{order-2m} is a special case of \eqref{general-RHS} with $a_i=\alpha(-1)^{m+i}\binom{2m}{m+i}$,  and one can easily check that $a_{-i}=a_i$. Therefore, its fundamental solution is 
\[
G(x,t)={\mathrm e}^{\alpha(-1)^m\binom{2m}{m}t}\sum_{\substack{x_1,\dots,x_m\in\mathbb Z,\\ 1x_1+2x_2+\cdots+mx_m=x}}\prod_{i=1}^m I_{x_i}\left(2\alpha(-1)^{m+i}\binom{2m}{m+i}t\right).
\]
Using the fact that $I_n(-t)=(-1)^nI_n(t)$ for all $n\in\mathbb Z$, the previous formula can be rewritten in the form
\[
G(x,t)=(-1)^{x}{\mathrm e}^{\alpha(-1)^m\binom{2m}{m}t}\sum_{\substack{x_1,\dots,x_m\in\mathbb Z,\\ 1x_1+2x_2+\cdots+mx_m=x}}(-1)^{m(x_1+\cdots+x_m)}\prod_{i=1}^m I_{x_i}\left(2\alpha\binom{2m}{m+i}t\right).
\]
For $m=1$, we get the discrete diffusion kernel \eqref{fs-sym-1}.
The formulas for $m\ge 2$ are new; for instance, the choice $m=2$ yields
\begin{align*}
G(x,t)&={\mathrm e}^{6\alpha t}\sum_{\substack{x_1,x_2\in\mathbb Z,\\ x_1+2x_2=x}}I_{x_1}(-8\alpha t)I_{x_2}(2\alpha t)\\
&={\mathrm e}^{6\alpha t}\sum_{k\in\mathbb Z}I_k(2\alpha t)I_{x-2k}(-8\alpha t)=(-1)^x{\mathrm e}^{6\alpha t}\sum_{k\in\mathbb Z}I_k(2\alpha t)I_{x-2k}(8\alpha t).
\end{align*}
\end{example}

\subsection{Antisymmetric case}

If the coefficients $a_{-m},\ldots,a_m$ in \eqref{general-RHS} satisfy $a_i=-a_{-i}$ for all $i\in\{0,\dots,m\}$ (in particular, $a_0=0$), we can proceed in similar way. Instead of \eqref{modified-bessel-expansion}, we use the expansion \begin{equation*}
	\exp\left(w\left(z-\frac{1}{z}\right)\right)=\sum_{x\in\mathbb Z}J_x(2w)z^x,
\end{equation*}
where $J_x$ is the Bessel function of order $x$. The same calculation as before leads to the following result.

\begin{theorem}\label{FS-odd}
	The fundamental solution of \eqref{general-RHS} where $a_i=-a_{-i}$ for all $i\in\{0,\dots,m\}$, is
	\[
	G(x,t)=\sum_{\substack{x_1,\dots,x_m\in\mathbb Z,\\ 1x_1+2x_2+\cdots+mx_m=x}}J_{x_1}(2a_1t)J_{x_2}(2a_2t)\cdots J_{x_m}(2a_mt).
	\]
\end{theorem} 

Again, we will postpone the proof of convergence of the series until Theorem \ref{FS-general}. For $m=1$, we get the formula
\begin{equation}\label{fs-antisym-1}
G(x,t)=J_x(2a_1t),
\end{equation}
which follows from \cite[p.~490]{SS2}. The formulas for $m\ge 2$ seem to be new; for instance, the choice $m=2$ leads to
\begin{equation}\label{fs-antisym-2}
G(x,t)=\sum_{\substack{x_1,x_2\in\mathbb Z,\\ x_1+2x_2=x}}J_{x_1}(2a_1t)J_{x_2}(2a_2t)=\sum_{k\in\mathbb Z}  J_k(2a_2t)J_{x-2k}(2a_1t)
\end{equation}
(the second equality is obtained by letting $x_2=k$ and $x_1=x-2x_2=x-2k$).

\begin{example}
	The derivative of a function $f$ of an odd order $2m-1$ can be approximated using the central difference of order $2m-1$ given by
	\[
	\Delta^{2m-1}f(x)=\frac{1}{2}\sum_{i=-m}^{m}(-1)^{m-i}\left(\binom{2m-1}{i+m}-\binom{2m-1}{i+m-1}\right)f(x-i).
	\]
	The expression can be obtained using the binomial theorem from
	\[
	\Delta^{2m-1}=\frac{1}{2}(E^{1/2}+E^{-1/2})(E^{1/2}-E^{-1/2})^{2m-1},
	\]
	where $E$ is again the forward shift operator; the averaging operator $\frac{1}{2}(E^{1/2}+E^{-1/2})$ is employed to avoid non-integer powers of $E$, making it possible to deal with functions defined only on the integer grid. 
	
	We can now consider the equation
	\begin{equation}\label{order-2m+1}
		\frac{\partial u}{\partial t}(x,t)=\alpha\Delta^{2m-1}u(x,t),\quad x\in\mathbb Z,\quad t\ge 0,
	\end{equation}
	where $\alpha\in\mathbb R$, and the difference on the right-hand side is with respect to $x$.
	Equation \eqref{order-2m+1} is a special case of \eqref{general-RHS} with 
	\[
	a_i=\frac{\alpha}{2}(-1)^{m-i}\left(\binom{2m-1}{i+m}-\binom{2m-1}{i+m-1}\right),\]
	and one can easily check that $a_0=0$ and $a_{-i}=-a_i$ for $i\in\{1,\dots,m\}$. Therefore, the fundamental solution is 
	\begin{align*}
	G(x,t)&=\sum_{\substack{x_1,\dots,x_m\in\mathbb Z,\\ 1x_1+2x_2+\cdots+mx_m=x}}\prod_{i=1}^m J_{x_i}\left(\alpha(-1)^{m-i}\left(\binom{2m-1}{i+m}-\binom{2m-1}{i+m-1}\right)t\right)\\
	&=(-1)^{x}\sum_{\substack{x_1,\dots,x_m\in\mathbb Z,\\ 1x_1+2x_2+\cdots+mx_m=x}}(-1)^{m(x_1+\cdots+x_m)}\prod_{i=1}^m J_{x_i}\left(\alpha\left(\binom{2m-1}{i+m}-\binom{2m-1}{i+m-1}\right)t\right).
	\end{align*}
\end{example}

For $m=1$, we get the equation
\[
\frac{\partial u}{\partial t}(x,t)=\alpha\Delta u(x,t),\quad x\in\mathbb Z,\quad t\ge 0,
\]
where
\[
\Delta u(x,t)=\frac{1}{2}\left(u(x+1,t)-u(x-1,t)\right).
\]
The equation is a semidiscrete version of the transport equation with centered spatial difference, whose fundamental solution is
\begin{equation*}
G(x,t)=J_x(-\alpha t)=(-1)^xJ_x(\alpha t).
\end{equation*}
This is a special case of a result from \cite[p.~489]{SS2} (see also \eqref{asymmetric-fundamental-solution}). 
One can compare this formula with the fundamental solution of the semidiscrete transport equation with backward spatial difference, which was obtained in \cite[Lemma 4.1]{SV}.

The results for $m\ge 2$ are new. For instance, the choice  $m=2$ leads to the equation
\[
\frac{\partial u}{\partial t}(x,t)=\alpha\Delta^3 u(x,t),\quad x\in\mathbb Z,\quad t\ge 0,
\]
where 
\[
\Delta^3 u(x,t)=\frac{1}{2}\left(u(x+2,t)-2u(x+1,t)+2u(x-1,t)-u(x-2,t)\right)
\]
is a finite difference approximation of the third-order derivative. The fundamental solution is
\[
	G(x,t)=(-1)^x\sum_{\substack{x_1,x_2\in\mathbb Z,\\ x_1+2x_2=x}}J_{x_1}(-2\alpha t)J_{x_2}(-\alpha t)=\sum_{k\in\mathbb Z}(-1)^kJ_{x-2k}(2\alpha t)J_k(\alpha t)=\sum_{l\in\mathbb Z}J_{x+2l}(2\alpha t)J_l(\alpha t)
\]
(to get the last equality, let $l=-k$ and use $J_{-l}(z)=(-1)^lJ_l(z)$).

\subsection{General case}\label{sub-general-case}

General coefficients $a_{-m},\dots,a_m\in\mathbb R$ are handled in a similar way.
Let \mbox{$a_{-m},\dots,a_{-1},a_1,\dots,a_m\neq0$}.
For each $i\in\{1,\dots,m\}$, choose $\eta_i\in\mathbb C$
such that
\[
\eta_i^2=a_{-i}a_i.
\]
Then
\[
\eta_i\left(
\frac{\eta_i}{a_i z^i}
+\frac{a_i z^i}{\eta_i}
\right)
=
\frac{a_{-i}}{z^i}+a_i z^i.
\]
Using \eqref{modified-bessel-expansion}, we obtain
\begin{align}
F(z,t)
&=
{\mathrm e}^{a_0t}
\prod_{i=1}^m
{\mathrm e}^{(a_{-i}/z^i+a_i z^i)t}
=
{\mathrm e}^{a_0t}
\prod_{i=1}^m
\exp\left[
\eta_i t\left(
\frac{\eta_i}{a_i z^i}
+\frac{a_i z^i}{\eta_i}
\right)
\right]
\label{GF-calculation-1}\\
&=
{\mathrm e}^{a_0t}
\prod_{i=1}^m
\sum_{x_i\in\mathbb Z}
I_{x_i}(2\eta_i t)
\left(\frac{a_i}{\eta_i}\right)^{x_i}z^{ix_i}.
\label{GF-calculation-2}
\end{align}
These expressions are independent of the choices of signs
of the $\eta_i$: replacing $\eta_i$ by $-\eta_i$ leaves
\[
I_{x_i}(2\eta_i t)
\left(\frac{a_i}{\eta_i}\right)^{x_i}
\]
unchanged, since $I_n(-z)=(-1)^nI_n(z)$ for integer $n$.

If $a_i$ and $a_{-i}$ have opposite signs, then $\eta_i$
is imaginary. The modified Bessel function can then be
expressed using a Bessel function with real argument through
the identity $I_n(\I r)=\I^nJ_n(r)$.

Extracting the coefficient of $z^x$, we get the following result.

\begin{theorem}\label{FS-general}
	The fundamental solution of \eqref{general-RHS} with coefficients satisfying $a_{-m},\dots,a_{-1},a_1,\dots,a_m\ne 0$ is
\[
G(x,t)
=
{\mathrm e}^{a_0t}
\sum_{\substack{x_1,\dots,x_m\in\mathbb Z\\
x_1+2x_2+\cdots+mx_m=x}}
\prod_{i=1}^m
I_{x_i}(2\eta_i t)
\left(\frac{a_i}{\eta_i}\right)^{x_i},
\]
where $\eta_i^2=a_{-i}a_i$ for $i=1,\dots,m$.
The infinite series is absolutely convergent, and
$G(\cdot,t)\in\ell^1(\mathbb Z)$ for every $t\ge 0$, with
\[
\sum_{x\in\mathbb Z}|G(x,t)|
\le \exp\left(\sum_{i=-m}^m |a_i|t\right).
\]
Consequently, $G$ is bounded on $\mathbb Z\times[0,T]$ for every $T>0$.
\end{theorem}

\begin{proof}
	The formula for $G$ follows from the earlier calculations.
	The estimate $|I_x(z)|\le I_x(|z|)$ implies
	\[
	|G(x,t)|\le {\mathrm e}^{a_0t}\sum_{\substack{x_1,\dots,x_m\in\mathbb Z,\\ 1x_1+2x_2+\cdots+mx_m=x}}I_{x_1}(2t\sqrt{|a_{-1}a_1|})\cdots I_{x_m}(2t\sqrt{|a_{-m}a_m|})\left(\sqrt{\frac{|a_{1}|}{|a_{-1}|}}\right)^{x_1}\cdots \left(\sqrt{\frac{|a_{m}|}{|a_{-m}|}}\right)^{x_m}.
	\]
	Therefore, similar calculations as in \eqref{GF-calculation-1}--\eqref{GF-calculation-2} with $z=1$ reveal that
	\[
	\sum_{x\in\mathbb Z}|G(x,t)|
	\le {\mathrm e}^{a_0t}\prod_{i=1}^m \sum_{x_i\in\mathbb Z}I_{x_i}(2t\sqrt{|a_{-i}a_i|})\left(\sqrt{\frac{|a_{i}|}{|a_{-i}|}}\right)^{x_i}
	={\mathrm e}^{a_0t}\prod_{i=1}^m{\mathrm e}^{(|a_{-i}|+|a_i|)t}
	\le \exp\left(\sum_{i=-m}^m |a_i|t\right).
	\]
	For every $T>0$, we get the estimate
	\[
	|G(x,t)|\le \exp\left(\sum_{i=-m}^m |a_i|T\right),\quad x\in\mathbb Z,\quad t\in[0,T].\tag*{\qedhere}
	\]
\end{proof}

In the symmetric case $a_{-i}=a_i$, choosing $\eta_i=a_i$
reduces the formula to Theorem~\ref{FS-even}.
In the antisymmetric case $a_{-i}=-a_i$ and $a_0=0$,
choose $\eta_i=\I a_i$.
Then
\[
I_{x_i}(2\I a_i t)
\left(\frac1\I\right)^{x_i}
=
J_{x_i}(2a_i t),
\]
so the formula reduces to Theorem~\ref{FS-odd}.

For $m=1$, we get the formula
\begin{equation}\label{asymmetric-fundamental-solution}
G(x,t)
=
{\mathrm e}^{a_0t}I_x(2\eta_1 t)
\left(\frac{a_1}{\eta_1}\right)^x,
\qquad \eta_1^2=a_{-1}a_1.
\end{equation}
which was obtained in \cite[p.~489]{SS2}. The formulas for $m\ge 2$ are new.

\subsection{The Fourier transform approach}

Sometimes it might be more convenient to express the fundamental solutions in terms of integrals. This can be achieved using the discrete Fourier transform. Recall that the fundamental solution $G$
should satisfy
\begin{equation*}
	\frac{\partial G}{\partial t}(x,t)=\sum_{j=-m}^m a_jG(x-j,t),\quad x\in\mathbb Z,\quad t\ge 0,
\end{equation*}
and
\[
G(x,0)=\begin{cases}
	1&\mbox{ if }x=0,\\
	0&\mbox{ if }x\ne 0.
\end{cases}
\]
The Fourier transform of $G$ with respect to $x$ is 
\[
\widehat G(\xi,t)=\sum_{x\in\mathbb Z}G(x,t){\mathrm e}^{-\I x\xi},\quad \xi\in[-\pi,\pi],\quad t\ge 0.
\]
The initial condition implies
\[
\widehat G(\xi,0)=1,\quad \xi\in[-\pi,\pi].
\]
Observe that
\[
\sum_{x\in\mathbb Z}G(x-j,t){\mathrm e}^{-\I x\xi}={\mathrm e}^{-\I j\xi}\sum_{x\in\mathbb Z}G(x-j,t){\mathrm e}^{-\I (x-j)\xi}={\mathrm e}^{-\I j\xi}\widehat G(\xi,t),\quad j\in\{-m,\dots,m\}.
\]
Hence, taking the Fourier transform of the differential equation, we get
\[
\frac{\partial \widehat G}{\partial t}(\xi,t)=\widehat G(\xi,t)\sum_{j=-m}^m a_j{\mathrm e}^{-\I j\xi},\quad \xi\in[-\pi,\pi],\quad t\ge 0.
\]
This means that
\[
\widehat G(\xi,t)=\exp\left(\sum_{j=-m}^m a_j{\mathrm e}^{-\I j\xi}t\right),\quad \xi\in[-\pi,\pi],\quad t\ge 0.
\]
Applying the inverse Fourier transform, we obtain
\[
G(x,t)=\frac{1}{2\pi}\int_{-\pi}^{\pi}\widehat G(\xi,t){\mathrm e}^{\I x\xi}\,{\mathrm d}\xi
=\frac{1}{2\pi}\int_{-\pi}^{\pi}{\mathrm e}^{\I x\xi}\exp\left(\sum_{j=-m}^m a_j{\mathrm e}^{-\I j\xi}t\right)\,{\mathrm d}\xi.
\]
Note that the term corresponding to $j=0$ does not depend on $\xi$, so we can also write
\begin{equation}\label{fourier-fundamental-fs}
G(x,t)=\frac{{\mathrm e}^{a_0t}}{2\pi}\int_{-\pi}^{\pi}{\mathrm e}^{\I x\xi}\exp\left(\sum_{j\in\{-m,\dots,m\}\setminus\{0\}} a_j{\mathrm e}^{-\I j\xi}t\right)\,{\mathrm d}\xi.
\end{equation}

In the symmetric case when $a_i=a_{-i}$ for all $i$,  formula \eqref{fourier-fundamental-fs} simplifies to
\[
G(x,t)=\frac{{\mathrm e}^{a_0t}}{2\pi}\int_{-\pi}^{\pi}{\mathrm e}^{\I x\xi}\exp\left(2\sum_{j=1}^m a_j t \cos(j\xi)\right)\,{\mathrm d}\xi.
\]
For example, if $m=1$, we get
\[
G(x,t)=\frac{{\mathrm e}^{a_0t}}{2\pi}\int_{-\pi}^{\pi}{\mathrm e}^{\I x\xi}{\mathrm e}^{2 a_1 t \cos(\xi)}\,{\mathrm d}\xi={\mathrm e}^{a_0t}I_{-x}(2a_1t)={\mathrm e}^{a_0t}I_{x}(2a_1t),
\]
which agrees with \eqref{fs-sym-1}; we have used the integral representation $I_x(z)=\frac{1}{2\pi}\int_{-\pi}^\pi {\mathrm e}^{z\cos\xi-\I x\xi}\,{\mathrm d}\xi$.

For $m=2$, we get
\[
G(x,t)=\frac{{\mathrm e}^{a_0t}}{2\pi}\int_{-\pi}^{\pi}{\mathrm e}^{\I x\xi}{\mathrm e}^{2 a_1 t \cos(\xi)}{\mathrm e}^{2 a_2 t \cos(2\xi)}\,{\mathrm d}\xi,
\]
which is a different way of expressing \eqref{fs-sym-2}. It is possible to derive \eqref{fs-sym-2} from the previous formula by using the expansion ${\mathrm e}^{z\cos\xi}=\sum_{k\in\mathbb Z}I_k(z){\mathrm e}^{\I k\xi}$, multiplying the series, and integrating term by term.

In the antisymmetric case when $a_i=-a_{-i}$ for all $i$ (and $a_0=0$),  formula \eqref{fourier-fundamental-fs} becomes
\[
G(x,t)=\frac{1}{2\pi}\int_{-\pi}^{\pi}{\mathrm e}^{\I x\xi}\exp\left(-2\I\sum_{j=1}^m a_j t \sin(j\xi)\right)\,{\mathrm d}\xi.
\]
For example, if $m=1$, we get
\[
G(x,t)=\frac{1}{2\pi}\int_{-\pi}^{\pi}{\mathrm e}^{\I x\xi}{\mathrm e}^{-2\I a_1 t \sin(\xi)}\,{\mathrm d}\xi=J_{x}(2a_1t),
\]
which agrees with \eqref{fs-antisym-1}; we have used the integral representation $J_x(z)=\frac{1}{2\pi}\int_{-\pi}^\pi {\mathrm e}^{-\I z\sin\xi+\I x\xi}\,{\mathrm d}\xi$.

For $m=2$, we get
\[
G(x,t)=\frac{1}{2\pi}\int_{-\pi}^{\pi}{\mathrm e}^{\I x\xi}{\mathrm e}^{-2\I a_1 t \sin(\xi)}{\mathrm e}^{-2\I a_2 t \sin(2\xi)}\,{\mathrm d}\xi,
\]
which is a different way of expressing \eqref{fs-antisym-2}. It is possible to derive \eqref{fs-antisym-2} from the previous formula by using the expansion ${\mathrm e}^{\I z\sin\xi}=\sum_{k\in\mathbb Z}J_k(z){\mathrm e}^{\I k\xi}$, multiplying the series, and integrating term by term.

\section{Nonhomogeneous equations} \label{sec-3}

We now use the fundamental solutions obtained above to
construct solutions with prescribed initial data and forcing.
The main tools are the superposition principle and
Duhamel's formula.

Consider again the equation
\begin{equation}\label{homogeneous}
	\frac{\partial u}{\partial t}(x,t)=\sum_{i=-m}^m a_iu(x-i,t),\quad x\in\mathbb Z,\quad t\ge 0.
\end{equation}

We begin by recalling the superposition principle from \cite[Theorem 2.5]{AS17}.

\begin{theorem}\label{superposition}
	Let $u_k\colon\mathbb Z\times[0,T]\to\mathbb R$, $k\in\mathbb N$, be a sequence of bounded solutions of \eqref{homogeneous}. Assume there exists a $\beta>0$ such that $\sum_{k=1}^\infty |u_k(x,0)|\leq \beta$ for every $x\in\mathbb Z$.                                                                                                                          Then, for every bounded sequence $\{c_k\}_{k=1}^\infty$, the function $u(x,t)=\sum_{k=1}^\infty c_k u_k(x,t)$ is a solution of \eqref{homogeneous} on $\mathbb Z\times[0,T]$. Moreover, $u$ is bounded on $\mathbb Z\times[0,T]$, with
	\[
	|u(x,t)|\le \beta\left(\sup_{k\in\mathbb N} |c_k|\right)\exp\left(\sum_{i=-m}^m|a_i|t\right),\quad x\in\mathbb Z,\quad t\in[0,T].
	\]
\end{theorem}

The fact that $u$ is bounded and the upper bound are not explicitly mentioned in \cite[Theorem 2.5]{AS17}, but they follow from the proof there.

As a corollary of the superposition principle, given a bounded bi-infinite sequence $\{u_x\}_{x\in\mathbb Z}$, the solution of the homogeneous equation \eqref{homogeneous}
subject to the initial condition
\[
u(x,0)=u_x,\quad x\in\mathbb Z,
\]
is given by the formula
\begin{equation*}
	u(x,t)=\sum_{k\in\mathbb Z} u_k G(x-k,t),\quad x\in\mathbb Z,\quad t\ge 0,
\end{equation*}
where $G$ is the fundamental solution (note that boundedness of $G$ on $\mathbb Z\times[0,T]$ for every $T>0$ is guaranteed by Theorem \ref{FS-general}).

Let us now consider the nonhomogeneous equation 
\begin{equation}\label{nonhomogeneous}
	\frac{\partial v}{\partial t}(x,t)=\sum_{i=-m}^m a_iv(x-i,t)+f(x,t),\quad x\in\mathbb Z,\quad t\ge 0,
\end{equation}
where $f$ is continuous. 

According to Duhamel's principle, we expect that the solution corresponding to zero initial values at $t=0$ will be given by the formula of the next theorem.

\begin{theorem}\label{duhamel} Let $G\colon\mathbb Z\times[0,\infty)\to\mathbb R$ be the fundamental solution of \eqref{homogeneous}. Suppose that $f\colon\mathbb Z\times[0,\infty)\to\mathbb R$ satisfies the following conditions:
	\begin{itemize}
		\item For every $x\in\mathbb Z$, the function $t\mapsto f(x,t)$ is continuous on $[0,\infty)$.
		\item For every $T>0$, the function $f$ is bounded on $\mathbb Z\times[0,T]$.
	\end{itemize}
Then the function
\[
v(x,t)=\int_0^t \sum_{k\in\mathbb Z}f(k,s)G(x-k,t-s)\,{\mathrm d}s,\quad x\in\mathbb Z,\quad t\ge 0,
\]
is the unique solution of \eqref{nonhomogeneous} which is bounded on $\mathbb Z\times[0,T]$ for every $T>0$, and satisfies $v(x,0)=0$ for all $x\in\mathbb Z$.
\end{theorem}

\begin{proof}
Fix $T>0$ and put
\[
C:=\sum_{i=-m}^m|a_i|,
\qquad
M_T:=\sup_{\substack{k\in\mathbb Z\\0\leq s\leq T}}
|f(k,s)|.
\]
We first obtain a summable bound for $G$ that is uniform
in time.
Since $G$ satisfies \eqref{homogeneous} and
$G(j,0)=\delta_{j,0}$, we have
\[
G(j,r)
=
\delta_{j,0}
+
\int_0^r
\sum_{i=-m}^m a_iG(j-i,s)\,{\mathrm d}s.
\]
Consequently,
\[
\sup_{0\leq r\leq T}|G(j,r)|
\leq
\delta_{j,0}
+
\int_0^T
\sum_{i=-m}^m|a_i|\,|G(j-i,s)|\,{\mathrm d}s.
\]
For every $N\in\mathbb N$, Theorem~\ref{FS-general} gives
\begin{align*}
\sum_{j=-N}^N\sup_{0\leq r\leq T}|G(j,r)|
&\leq
1+
\int_0^T
\sum_{i=-m}^m|a_i|
\sum_{j=-N}^N|G(j-i,s)|\,{\mathrm d}s\leq
1+C\int_0^T{\mathrm e}^{Cs}\,{\mathrm d}s
=
{\mathrm e}^{CT}.
\end{align*}
Letting $N\to\infty$, we obtain
\[
\sum_{j\in\mathbb Z}
\sup_{0\leq r\leq T}|G(j,r)|
\leq{\mathrm e}^{CT}.
\]

For each fixed $x\in\mathbb Z$, the estimate
\[
|f(k,s)G(x-k,t-s)|
\leq
M_T\sup_{0\leq r\leq T}|G(x-k,r)|
\]
and the Weierstrass M-test imply that
\[
\sum_{k\in\mathbb Z}f(k,s)G(x-k,t-s)
\]
converges absolutely and uniformly on
$0\leq s\leq t\leq T$.
Each summand is continuous in $(s,t)$, so the sum is
continuous and the integral defining $v(x,t)$ exists.
Moreover,
\[
\frac{\partial G}{\partial r}(j,r)
=
\sum_{i=-m}^m a_iG(j-i,r),
\]
and therefore
\begin{align*}
\sum_{k\in\mathbb Z}
\sup_{0\leq r\leq T}
\left|\frac{\partial G}{\partial r}(x-k,r)\right|
&\leq
\sum_{i=-m}^m|a_i|
\sum_{k\in\mathbb Z}
\sup_{0\leq r\leq T}|G(x-k-i,r)|\leq C{\mathrm e}^{CT}.
\end{align*}
Thus the series
\[
\sum_{k\in\mathbb Z}
f(k,s)\frac{\partial G}{\partial r}(x-k,t-s)
\]
also converges uniformly for $0\leq s\leq t\leq T$.
Its sum is continuous and equals the partial derivative
with respect to $t$ of the preceding series.
The Leibniz rule consequently yields
\begin{align*}
\frac{\partial v}{\partial t}(x,t)
&=
\sum_{k\in\mathbb Z}f(k,t)G(x-k,0)+
\int_0^t
\sum_{k\in\mathbb Z}f(k,s)
\sum_{i=-m}^m a_iG(x-k-i,t-s)\,{\mathrm d}s\\
&=
f(x,t)+\sum_{i=-m}^m a_iv(x-i,t).
\end{align*}
Here we used $G(x-k,0)=\delta_{x,k}$.
Also, $v(x,0)=0$ for every $x\in\mathbb Z$.

To establish boundedness, we again use
Theorem~\ref{FS-general}:
\begin{align*}
|v(x,t)|
&\leq
M_T\int_0^t
\sum_{k\in\mathbb Z}|G(x-k,t-s)|\,{\mathrm d}s\leq
M_T\int_0^t{\mathrm e}^{C(t-s)}\,{\mathrm d}s
\leq
M_TT{\mathrm e}^{CT},
\qquad x\in\mathbb Z,\quad 0\leq t\leq T.
\end{align*}
Since $T>0$ was arbitrary, $v$ is bounded on
$\mathbb Z\times[0,T]$ for every $T>0$.

Uniqueness of the solution follows from the fact that the difference of any two solutions is a solution of the homogeneous equation with zero initial conditions. The only solution of the latter equation which is bounded on bounded time intervals is the zero function.
\end{proof}

In the special case when $f$ does not depend on $x$, the
identity
\[
\sum_{j\in\mathbb Z}G(j,r)
=
\exp\left(\sum_{i=-m}^m a_ir\right)
\]
gives
\[
v(x,t)
=
\int_0^t
f(s)\exp\left(\sum_{i=-m}^m a_i(t-s)\right)
\,{\mathrm d}s.
\]
In particular, if $\sum_{i=-m}^m a_i=0$, then
\[
v(x,t)=\int_0^t f(s)\,{\mathrm d}s.
\]

\section{Symmetric and antisymmetric solutions}\label{sec-4}

We first show that equations with symmetric coefficients
preserve the symmetry or antisymmetry of the initial data
and forcing. We then study a particular antisymmetric
solution of the semidiscrete diffusion equation, which will
be used in the reflection method and in the analysis
of long-time behavior on the half-line.

\begin{theorem}\label{symmetry}
Suppose that $v$ is a bounded function such that
\begin{equation}
	\frac{\partial v}{\partial t}(x,t)=\sum_{i=-m}^m a_iv(x-i,t)+f(x,t),\quad x\in\mathbb Z,\quad t\ge 0,
\end{equation}
where $a_i=a_{-i}$ for all $i\in\{1,\dots,m\}$.
\begin{enumerate}
\item If $f(-x,t)=-f(x,t)$ for all $x$, $t$ and
$v(-x,0)=-v(x,0)$ for all $x\in\mathbb N_0$, then $v(-x,t)=-v(x,t)$ for all $x$ and $t$.
In particular, $v(0,t)=0$ for all $t\ge 0$.
\item If $f(-x,t)=f(x,t)$ for all $x$, $t$ and
$v(-x,0)=v(x,0)$ for all $x\in\mathbb N_0$, then $v(-x,t)=v(x,t)$ for all $x$ and $t$.
\end{enumerate}
\end{theorem}

\begin{proof}
	To prove the first statement, let $w(x,t)=-v(-x,t)$. Then
	\begin{align*}
		\frac{\partial w}{\partial t}(x,t)&=-\frac{\partial v}{\partial t}(-x,t)=-\left(\sum_{i=-m}^m a_iv(-x-i,t)+f(-x,t)\right)\\
		&=-\left(\sum_{i=-m}^m a_iv(-x+i,t)+f(-x,t)\right)
		=\sum_{i=-m}^m a_iw(x-i,t)+f(x,t).
	\end{align*}
	Hence $z=v-w$ satisfies the homogeneous equation
	\begin{equation}
		\frac{\partial z}{\partial t}(x,t)=\sum_{i=-m}^m a_iz(x-i,t),\quad x\in\mathbb Z,\quad t\ge 0.
	\end{equation}
	Moreover, $z(x,0)=v(x,0)-w(x,0)=v(x,0)+v(-x,0)=0$ for all $x\in\mathbb Z$. By uniqueness of bounded solutions for initial-value problems \cite[Theorem 2.1]{FSS}, the function $z$ is identically zero, i.e., $v(x,t)=w(x,t)=-v(-x,t)$ for all $x$, $t$.

	The second statement is proved similarly using $w(x,t)=v(-x,t)$.
\end{proof}

We now consider the solution of the homogeneous
semidiscrete diffusion equation with initial data $-\sgn x$, which will also appear in the explicit  solution formula for the boundary-initial value problem.

For $a>0$, denote the whole-line fundamental solution by
\[
K_x^{(a)}(t):={\mathrm e}^{-2at}I_x(2at),
\qquad x\in\mathbb Z,\quad t\geq0.
\]
When $a=1$, we write $K_x(t):=K_x^{(1)}(t)$.

\begin{theorem}\label{phi-function}
    Let $a>0$.
	The unique bounded function $\phi$ satisfying
	\begin{align}
	\frac{\partial \phi}{\partial t}(x,t)&=a\Delta \phi(x,t),\quad x\in\mathbb Z,\quad t\ge 0,\label{phi-1}\\
	\phi(x,0)&=-\sgn x,\quad x\in\mathbb Z,\label{phi-2}
	\end{align}
    where
    $\Delta \phi(x,t)=\phi(x+1,t)-2\phi(x,t)+\phi(x-1,t)$,
	is given by
	\[
	\phi(x,t)=\sum_{k\in\mathbb N}
	\bigl(K_{x+k}^{(a)}(t)-K_{x-k}^{(a)}(t)\bigr)
	=\sum_{z=x+1}^\infty K_z^{(a)}(t)
	-\sum_{z=-\infty}^{x-1}K_z^{(a)}(t),
	\quad x\in\mathbb Z,\quad t\ge 0.
	\]
	The solution has the following properties:
	\begin{itemize}
		\item $\phi(0,t)=0$ for all $t\ge 0$.
		\item For every $t\ge 0$, the sequence $x\mapsto \phi(x,t)$ is nonincreasing and satisfies
		\begin{equation}\label{limits}
		\lim_{x\to\infty}\phi(x,t)=-1,\quad \lim_{x\to-\infty}\phi(x,t)=1.
		\end{equation}
		In particular, $-1\le\phi(x,t)\le 1$ for all $x\in\mathbb Z$, $t\ge 0$.
		\item $\Delta\phi(x,t)=K_{x-1}^{(a)}(t)-K_{x+1}^{(a)}(t)$ for all $x\in\mathbb Z$, $t\ge 0$.
		\item If $a>0$, then $t\mapsto\phi(x,t)$ is increasing for $x\ge 1$, and decreasing for $x\le -1$.
		\item $\lim_{t\to\infty}\phi(x,t)=0$ for every $x\in\mathbb Z$. The convergence is not uniform with respect to $x$.
        \item
For every $x\geq1$ and $t\geq0$,
\begin{equation}\label{phi-finite-sum}
\phi(x,t)
=
-K_0^{(a)}(t)-2\sum_{z=1}^{x-1}K_z^{(a)}(t)-K_x^{(a)}(t)
=
-\sum_{\ell=1-x}^{x}K_\ell^{(a)}(t).
\end{equation}
For negative $x$, the values are determined by
$\phi(-x,t)=-\phi(x,t)$.
	\end{itemize}
\end{theorem}

\begin{proof}
By the superposition principle (Theorem \ref{superposition}), the solution of the problem \eqref{phi-1}--\eqref{phi-2} is given by
\[
\phi(x,t)=\sum_{k\in\mathbb Z}(-\sgn k)K_{x-k}^{(a)}(t),
\]
which coincides with the formula in the statement of the theorem.

The relation $\phi(0,t)=0$ for all $t\ge 0$ follows from Theorem \ref{symmetry}.
Next, observe that
\begin{align*}
	\phi(x+1,t)-\phi(x,t)
	&=\sum_{z=x+2}^\infty K_z^{(a)}(t)
	-\sum_{z=-\infty}^{x}K_z^{(a)}(t)
	-\left(
	\sum_{z=x+1}^\infty K_z^{(a)}(t)
	-\sum_{z=-\infty}^{x-1}K_z^{(a)}(t)
	\right)\\
	&=-K_{x+1}^{(a)}(t)-K_x^{(a)}(t)\le 0,
\end{align*}
which implies that $x\mapsto \phi(x,t)$ is nonincreasing. The relations \eqref{limits} follow from the definition of $\phi$ and the fact that $\sum_{z\in\mathbb Z}K_z^{(a)}(t)=1$ for every $t\ge 0$.

Next, we calculate
\begin{gather*}
	\Delta \phi(x,t)
	=\phi(x+1,t)-\phi(x,t)-(\phi(x,t)-\phi(x-1,t))\\
	=-K_{x+1}^{(a)}(t)-K_x^{(a)}(t)
	-\bigl(-K_x^{(a)}(t)-K_{x-1}^{(a)}(t)\bigr)
	=K_{x-1}^{(a)}(t)-K_{x+1}^{(a)}(t).
\end{gather*}
Consequently,
\begin{equation*}
	\frac{\partial \phi}{\partial t}(x,t)
	=a\Delta \phi(x,t)
	=a\bigl(K_{x-1}^{(a)}(t)-K_{x+1}^{(a)}(t)\bigr)
	=a{\mathrm e}^{-2at}(I_{x-1}(2at)-I_{x+1}(2at)).
\end{equation*}
Recalling that $x\mapsto I_x(\tau)$ is decreasing on $\mathbb N_0$ for $\tau>0$ and $I_{-x}(\tau)=I_x(\tau)$, we see that if $t>0$, then $\frac{\partial \phi}{\partial t}(x,t)$ is zero for $x=0$, positive for $x\ge 1$, and negative for $x\le-1$.

Results on the asymptotic behavior of solutions to the homogeneous diffusion equation imply (see \cite[Corollary 2.5]{Slavik2020}) that $
\lim_{t\to\infty} \phi(x,t)=0$ for all $x\in\mathbb Z$.
The convergence is not uniform with respect to $x$, since for every $t\ge 0$, we have $\sup_{x\in\mathbb Z}\phi(x,t)=1$ and $\inf_{x\in\mathbb Z}\phi(x,t)=-1$.

Recall that $K_{-z}^{(a)}(t)=K_z^{(a)}(t)$ and $\sum_{z\in\mathbb Z}K_z^{(a)}(t)=1$ for all $t\ge 0$, $z\in\mathbb Z$. It follows that
\[
\sum_{z=-\infty}^{-1}K_z^{(a)}(t)
=\sum_{z=1}^\infty K_z^{(a)}(t)
=\frac{1-K_0^{(a)}(t)}{2}.
\]
Since $x\mapsto \phi(x,t)$ is odd, we can focus on $x\ge 1$. Then
\begin{align*}
	\sum_{z=x+1}^\infty K_z^{(a)}(t)
	&=\frac{1-K_0^{(a)}(t)}{2}
	-\sum_{z=1}^x K_z^{(a)}(t),\\
	\sum_{z=-\infty}^{x-1}K_z^{(a)}(t)
	&=\frac{1-K_0^{(a)}(t)}{2}
	+\sum_{z=0}^{x-1}K_z^{(a)}(t),
\end{align*}
and therefore
\begin{align*}
\phi(x,t)
&=\frac{1-K_0^{(a)}(t)}{2}
-\sum_{z=1}^x K_z^{(a)}(t)
-\frac{1-K_0^{(a)}(t)}{2}
-\sum_{z=0}^{x-1}K_z^{(a)}(t)\\
&=-K_0^{(a)}(t)-2\sum_{z=1}^{x-1}K_z^{(a)}(t)-K_x^{(a)}(t)
=
-\sum_{\ell=1-x}^{x}K_\ell^{(a)}(t),
\quad x\ge 1,
\end{align*}
where the last equality follows immediately from $K_{-\ell}^{(a)}(t)=K_\ell^{(a)}(t)$.
\end{proof}

\begin{figure}[!h]
	\centering
	\includegraphics[width=12cm]{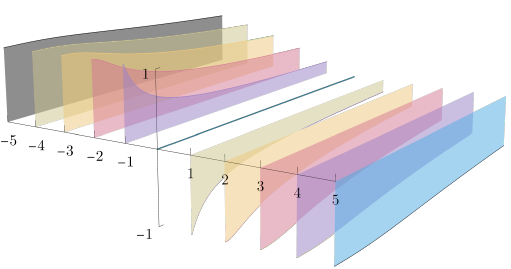}
	\caption{Graph of the solution $\phi$ with $a=1$}
	\label{fig2}
\end{figure}

The graph of $\phi$ is shown in Figure \ref{fig2}.
The finite-sum representation \eqref{phi-finite-sum} is useful
for computing $\phi$, since it avoids infinite sums.

Here is yet another way of expressing $\phi$: The Skellam distribution \cite{Skellam}, denoted by $\Skellam(\mu_1,\mu_2)$, is a discrete probability distribution that corresponds to the difference of two independent random variables, which are
Poisson-distributed with expected values $\mu_1$ and $\mu_2$, respectively. Its  probability mass function is
\[
P(X=x)={\mathrm e}^{-(\mu_1+\mu_2)}\left(\frac{\mu_1}{\mu_2}\right)^{x/2}I_x(2\sqrt{\mu_1\mu_2}).
\]
In the symmetric case when $X\sim\Skellam(at,at)$, we get
\[
P(X=x)={\mathrm e}^{-2at}I_x(2at)=K_x^{(a)}(t).
\]
Hence, we can express $\phi$ using the tail probabilities:
\[
\phi(x,t)=\sum_{z=x+1}^\infty K_z^{(a)}(t)
-\sum_{z=-\infty}^{x-1}K_z^{(a)}(t)
=P(X\ge x+1)-P(X\le x-1).
\]

\section{Basic results for initial-boundary value problems}\label{sec-5}

Before constructing solutions on the half-line, we establish
two general tools: uniqueness among solutions of polynomial
growth and a superposition principle for initial-boundary
value problems.

The proof of the following uniqueness result for initial-boundary value problems is inspired by the proof of a similar result for initial-value problems from \cite[Theorem 3.5]{SS}. In contrast to that result, the initial data need not be bounded since we allow polynomial growth.

\begin{theorem}\label{uniquenessForward}
	Let $m\in\mathbb N$. Consider the initial-boundary value problem
	\begin{align*}
		\frac{\partial u}{\partial t}(x,t)&=\sum_{i=-m}^m a_iu(x-i,t)+f(x,t),\quad x\in\{m,m+1,m+2,\dots\},\quad t\in[0,T],\\
		u(x,t)&=g_x(t),\quad t\in[0,T],\quad x\in\{0,\dots,m-1\},\\
		u(x,0)&=u_x,\quad x\in\{m,m+1,m+2,\dots\},
	\end{align*}
	where $g_0,\dots,g_{m-1}\colon[0,T]\to\mathbb R$ and $f\colon\mathbb N_0\times [0,T]\to\mathbb R$ are continuous and $\{u_x\}_{x\ge m}$  is such that $\{u_x\cdot x^{-M}\}$ is bounded for a certain $M\ge 0$.
	
	The problem has at most one solution
$u\colon\mathbb N_0\times[0,T]\to\mathbb R$ such that
$(x,t)\mapsto u(x,t)x^{-M}$ is bounded on
$\{m,m+1,m+2,\dots\}\times[0,T]$.
\end{theorem}

\begin{proof}
	Suppose, for a contradiction, that there are two distinct solutions
$u_1,u_2\colon\mathbb N_0\times[0,T]\to\mathbb R$ such that
$(x,t)\mapsto u_i(x,t)x^{-M}$ is bounded on
$\{m,m+1,m+2,\dots\}\times[0,T]$ for $i=1,2$.
	It is clear that $u_1(x,t)=g_x(t)=u_2(x,t)$ for all $x\in\{0,\dots,m-1\}$ and $t\in[0,T]$, so $u_1$ and $u_2$ must differ for a certain $t\ge0$ and $x\in\{m,m+1,m+2,\dots\}$.
	Let
	\[
	t_0=\inf\{s\in[0,T]:u_1(x,s)\neq u_2(x,s)\mbox{ for some }x\in\{m,m+1,m+2,\dots\}\}.
	\]
	By continuity, it is clear that $u_1(x,t_0)=u_2(x,t_0)$ for every $x\in\{m,m+1,m+2,\dots\}$, and therefore $0\le t_0<T$.
	For every $r\in[t_0,T]$ and $x\in\{m,m+1,m+2,\dots\}$, we have
	\begin{equation*}
		u_1(x,r)-u_2(x,r)=\int_{t_0}^r \left(\frac{\partial u_1}{\partial t}(x,t)-\frac{\partial u_2}{\partial t}(x,t)\right) {\mathrm d}t
		=\int_{t_0}^r \left(\sum_{i=-m}^m a_iu_1(x-i,t)-\sum_{i=-m}^m a_iu_2(x-i,t)\right){\mathrm d}t,
	\end{equation*}
	and therefore
		\begin{equation*}
		|u_1(x,r)-u_2(x,r)|\le 
		\sum_{i=-m}^m |a_i|\int_{t_0}^r |u_1(x-i,t)-u_2(x-i,t)|\,{\mathrm d}t.
	\end{equation*}
	Multiplying by ${\mathrm e}^{-x}$ and observing that $(x,t)\mapsto |u_1(x,t)-u_2(x,t)|{\mathrm e}^{-x}$ is bounded, we get
		\begin{align*}
	|u_1(x,r)-u_2(x,r)|{\mathrm e}^{-x}&\le 
	\sum_{i=-m}^m \left(|a_i|{\mathrm e}^{-i}\int_{t_0}^r |u_1(x-i,t)-u_2(x-i,t)|{\mathrm e}^{-(x-i)}\,{\mathrm d}t\right)\\
&\le \left({\mathrm e}^m\sum_{i=-m}^m|a_i|\right)(r-t_0)\sup_{x\in\mathbb N_0,\,t\in[t_0,r]}|u_1(x,t)-u_2(x,t)|{\mathrm e}^{-x}.
	\end{align*}
	Now choose $t_1\in(t_0,T]$ such that
$\left({\mathrm e}^m\sum_{i=-m}^m|a_i|\right)(t_1-t_0)
\leq\frac12$.
Then
\[
|u_1(x,r)-u_2(x,r)|{\mathrm e}^{-x}
\leq
\frac12
\sup_{\substack{x\in\mathbb N_0\\t\in[t_0,t_1]}}
|u_1(x,t)-u_2(x,t)|{\mathrm e}^{-x},
\quad r\in[t_0,t_1],\quad x\geq m.
\]
Since the two solutions agree at the boundary points,
taking the supremum gives
\[
\sup_{\substack{x\in\mathbb N_0\\t\in[t_0,t_1]}}
|u_1(x,t)-u_2(x,t)|{\mathrm e}^{-x}
\leq
\frac12
\sup_{\substack{x\in\mathbb N_0\\t\in[t_0,t_1]}}
|u_1(x,t)-u_2(x,t)|{\mathrm e}^{-x}.
\]
This is a contradiction, since the supremum is finite
and is positive by the definition of $t_0$.
\end{proof}

The next result is a version of the superposition principle for solutions defined on $\mathbb N_0\times[0,T]$ rather than $\mathbb Z\times[0,T]$. The proof is essentially the same as the proof of Theorem \ref{superposition} (which can be found in \cite[Theorem 2.5]{AS17}), but we include it for reader's convenience.

\begin{theorem}\label{superposition-2}
Let $u_k\colon\mathbb N_0\times[0,T]\to\mathbb R$,
$k\in\mathbb N$, be a sequence of bounded functions satisfying
\begin{equation}\label{uk-equation}
\frac{\partial u_k}{\partial t}(x,t)
=
\sum_{i=-m}^{m}a_i u_k(x-i,t),
\quad x\in\{m,m+1,m+2,\dots\},\quad t\in[0,T].
\end{equation}
Assume that:
\begin{itemize}
\item
There exists a $\beta\geq0$ such that
$\sum_{k=1}^{\infty}|u_k(x,0)|\leq\beta$ for every
$x\in\{m,m+1,m+2,\dots\}$.
\item
There exists a $\gamma\geq0$ such that
$\sum_{k=1}^{\infty}|u_k(x,t)|\leq\gamma$ for all
$x\in\{0,\dots,m-1\}$ and $t\in[0,T]$.
\end{itemize}
Assume also that $t\mapsto u_k(x,t)$ is continuous on $[0,T]$
for every $k\in\mathbb N$ and every $x\in\{0,\dots,m-1\}$.

Then, for every bounded sequence $\{c_k\}_{k=1}^{\infty}$
such that
\[
t\longmapsto\sum_{k=1}^{\infty}c_k u_k(x,t)
\]
is continuous on $[0,T]$ for every $x\in\{0,\dots,m-1\}$,
the function $u\colon\mathbb N_0\times[0,T]\to\mathbb R$
given by
\[
u(x,t)=\sum_{k=1}^{\infty}c_k u_k(x,t)
\]
satisfies
\[
\frac{\partial u}{\partial t}(x,t)
=
\sum_{i=-m}^{m}a_i u(x-i,t),
\quad x\in\{m,m+1,m+2,\dots\},\quad t\in[0,T].
\]
Moreover, $u$ is bounded on $\mathbb N_0\times[0,T]$, with
\begin{equation}\label{super-estimate}
|u(x,t)|
\leq
(\beta+\gamma)
\left(\sup_{k\in\mathbb N}|c_k|\right)
\exp\left(\sum_{i=-m}^{m}|a_i|t\right),
\quad x\in\mathbb N_0,\quad t\in[0,T].
\end{equation}
\end{theorem}

\begin{proof}
	Let $M=\sup_{k\in\mathbb N} |c_k|$, and let $\{d_k\}_{k=1}^\infty$ be an arbitrary sequence of numbers such that $|d_k|\leq M$. Consider the sums
	\[
	u^{(N)}(x,t)=\sum_{k=1}^N d_k u_k(x,t),\quad x\in\mathbb N_0,\quad t\in[0,T],\quad N\in\mathbb N.
	\]
	Since each $u_k$ satisfies \eqref{uk-equation}, we have
	\[
	u^{(N)}(x,t)=u^{(N)}(x,0)+\int_{0}^t\sum_{i=-m}^m a_i u^{(N)}(x-i,s)\,{\mathrm d}s,\quad x\ge m,\quad t\in[0,T],\quad N\in\mathbb N,
	\]
	which leads to the estimate
	\[
	\sup_{x\ge m}|u^{(N)}(x,t)|\leq\beta M+\left(\sum_{i=-m}^m|a_i|\right)\int_{0}^t \sup_{x\ge 0}|u^{(N)}(x,s)|\,{\mathrm d s},\quad t\in[0,T], \quad N\in\mathbb N.
	\]
	Since
	\[
	|u^{(N)}(x,t)|\leq \gamma M,\quad x\in\{0,\dots,m-1\},\quad t\in[0,T],\quad N\in\mathbb N,
	\]
	we get 
	\[
	\sup_{x\ge 0}|u^{(N)}(x,t)|\leq\beta M+\gamma M+\left(\sum_{i=-m}^m|a_i|\right)\int_{0}^t \sup_{x\ge 0}|u^{(N)}(x,s)|\,{\mathrm d s},\quad t\in[0,T],\quad N\in\mathbb N.
	\]	
	By Gronwall's inequality, we have
	\[
	\sup_{x\ge 0}|u^{(N)}(x,t)|\leq (\beta+\gamma) M \exp\left(\sum_{i=-m}^m |a_i|t\right),\quad t\in[0,T],\quad N\in\mathbb N,
	\]
	and therefore 
	\[
	|u^{(N)}(x,t)|\leq(\beta+\gamma) M \exp\left(\sum_{i=-m}^m |a_i|t\right),\quad x\in\mathbb N_0,\quad t\in[0,T],\quad N\in\mathbb N.
	\]
	
	For an arbitrary fixed pair $(x,t)$, we can choose $d_k=|c_k|\sgn u_k(x,t)$, $k\in\mathbb N$. Then, the previous inequality reduces to
	\[
	\sum_{k=1}^N |c_k|\cdot|u_k(x,t)|\leq (\beta+\gamma) M \exp\left(\sum_{i=-m}^m |a_i|t\right),\quad N\in\mathbb N.
	\]
	In particular, the infinite series in the definition of $u$ is absolutely convergent, $u$ is bounded, and the estimate \eqref{super-estimate} holds.
	
	Next, we observe that for all $x\ge m$, $N\in\mathbb N$, $t\in[0,T]$, we have
	\begin{gather*}
	\left|\sum_{k=1}^N d_k \frac{\partial u_k}{\partial t}(x,t)\right|=\left|\sum_{k=1}^N d_k\sum_{i=-m}^m a_i u_k(x-i,t)\right|\\
	\le \sum_{i=-m}^m |a_i|\sum_{k=1}^N|d_k u_k(x-i,t)|
	\le \left(\sum_{i=-m}^m |a_i|\right)(\beta+\gamma) M \exp\left(\sum_{i=-m}^m |a_i|t\right)
	\end{gather*}
	In particular, for $d_k=|c_k|\sgn \frac{\partial u_k}{\partial t}(x,t)$, $k\in\mathbb N$, the left-hand side becomes $\sum_{k=1}^N \left| c_k \frac{\partial u_k}{\partial t}(x,t)\right|$, so we see that the series $\sum_{k=1}^\infty c_k \frac{\partial u_k}{\partial t}(x,t)$ is absolutely convergent. Term by term integration (which is possible by Lebesgue's dominated convergence theorem) yields
	\begin{equation}\label{integral-eqn}
	\sum_{k=1}^\infty c_k u_k(x,0)+\int_{0}^t\left(\sum_{k=1}^\infty c_k \frac{\partial u_k}{\partial t}(x,s)\right)\,{\mathrm d} s
	=\sum_{k=1}^\infty c_k u_k(x,0)+\sum_{k=1}^\infty c_k (u_k(x,t)-u_k(x,0))=u(x,t).
	\end{equation}
	Since the integral on the left-hand side is a continuous function of its upper limit, we have proved continuity of $u(x,\cdot)$ for every $x\geq m$.
For $x\in\{0,\dots,m-1\}$, continuity holds by assumption.

    The relation
	\[
	\sum_{k=1}^\infty c_k \frac{\partial u_k}{\partial t}(x,t)=\sum_{k=1}^\infty c_k\sum_{i=-m}^m a_iu_k(x-i,t)=\sum_{i=-m}^m a_i u(x-i,t),\quad x\ge m,\quad t\in[0,T],
	\]
	implies that the left-hand side  is continuous with respect  to $t$. Hence, we can differentiate the equality~\eqref{integral-eqn} 
	with respect to $t$ and obtain
	\begin{equation*}
		\frac{\partial u}{\partial t}(x,t)=\sum_{k=1}^\infty c_k \frac{\partial u_k}{\partial t}(x,t)=\sum_{i=-m}^m a_iu(x-i,t),\quad x\ge m,\quad t\in[0,T].
		\tag*{\qedhere}
	\end{equation*}
\end{proof}

\section{Symmetric diffusion}\label{sec-6}

We now study the initial-boundary value problem for the
semidiscrete diffusion equation on the half-line.
We derive solution formulas by reflection, the Laplace
transform, and the unified transform method, and then
investigate their spatial and long-time properties.
We also consider Robin-type boundary conditions.

\subsection{The reflection method}\label{sec51}

We will solve the equation
\begin{equation}\label{half-line}
\frac{\partial u}{\partial t}(x,t)
=
\Delta u(x,t),
\qquad x\in\mathbb N,\quad t\geq0,
\end{equation}
where
$\Delta u(x,t)=u(x+1,t)-2u(x,t)+u(x-1,t)$,
subject to the conditions
\begin{equation}\label{initial-boundary-conditions}
u(0,t)=g(t),\qquad t\geq0,
\qquad
u(x,0)=u_x,\qquad x\in\mathbb N,
\end{equation}
where $g\colon[0,\infty)\to\mathbb R$ is continuous and
$\{u_x\}_{x\in\mathbb N}$ is bounded.
We will use the whole-line heat kernel
\[
K_x(t)={\mathrm e}^{-2t}I_x(2t).
\]

For simplicity, we take the diffusion coefficient to be $1$; the general case is easily handled by a change of variables, and will be discussed in Theorem \ref{theo-nonhomogeneous}.

We will temporarily suppose that $g$ is continuously differentiable.
Performing the substitution $v(x,t)=u(x,t)-g(t)$, we see that
$u$ satisfies \eqref{half-line} if and only if $v$ satisfies
\[
\frac{\partial v}{\partial t}(x,t)
=
\Delta u(x,t)-g'(t)
=
\Delta v(x,t)-g'(t),
\qquad x\in\mathbb N,\quad t\geq0,
\]
and the initial-boundary conditions become
\[
v(0,t)=0,\qquad t\geq0,
\qquad
v(x,0)=u_x-g(0),\qquad x\in\mathbb N.
\]
To handle the homogeneous Dirichlet condition at $x=0$, we will
use the reflection method, which is well known in classical PDE theory
(see e.g.~\cite[Section 8.2]{Epstein} or
\cite[Section 3.1]{Strauss}).
We will construct $v$ as the sum of two solutions:
\[
v(x,t)=v_h(x,t)+v_g(x,t),
\]
where $v_h$ satisfies
\[
\frac{\partial v_h}{\partial t}(x,t)
=
\Delta v_h(x,t),
\qquad x\in\mathbb Z,\quad t\geq0,
\]
with
\[
v_h(x,0)=
\begin{cases}
\sgn x\cdot(u_{|x|}-g(0)),&x\neq0,\\
0,&x=0,
\end{cases}
\]
and $v_g$ satisfies
\[
\frac{\partial v_g}{\partial t}(x,t)
=
\Delta v_g(x,t)-\sgn x\cdot g'(t),
\qquad x\in\mathbb Z,\quad t\geq0,
\]
with
\[
v_g(x,0)=0,\qquad x\in\mathbb Z.
\]
The initial data for both functions, as well as the forcing
for $v_g$, are odd in $x$.
Hence, by Theorem~\ref{symmetry},
$v_h(0,t)=v_g(0,t)=0$ for all $t\geq0$.

By the results of the previous sections, $v_h$ and $v_g$ are given by
\begin{align*}
v_h(x,t)
&=
\sum_{\substack{k\in\mathbb Z\\k\neq0}}
\sgn k\cdot(u_{|k|}-g(0))K_{x-k}(t)\\
&=
\sum_{k\in\mathbb N}(u_k-g(0))K_{x-k}(t)
+
\sum_{k\in\mathbb N}(-u_k+g(0))K_{x+k}(t)\\
&=
\sum_{k\in\mathbb N}(u_k-g(0))
\bigl[K_{x-k}(t)-K_{x+k}(t)\bigr],\\
v_g(x,t)
&=
\int_0^t
\sum_{k\in\mathbb Z}
\sgn k\cdot(-g'(s))K_{x-k}(t-s)\,{\mathrm d}s\\
&=
\int_0^t
\sum_{k\in\mathbb N}
g'(s)\bigl[K_{x+k}(t-s)-K_{x-k}(t-s)\bigr]
\,{\mathrm d}s=
\int_0^t g'(s)\phi(x,t-s)\,{\mathrm d}s,
\end{align*}
where
\begin{align*}
\phi(x,\tau)
&=
\sum_{k\in\mathbb N}
\bigl[K_{x+k}(\tau)-K_{x-k}(\tau)\bigr]=
\sum_{z=x+1}^{\infty}K_z(\tau)
-
\sum_{z=-\infty}^{x-1}K_z(\tau),
\qquad x\in\mathbb Z,\quad \tau\geq0,
\end{align*}
is the function introduced in Theorem~\ref{phi-function},
here with diffusion coefficient equal to $1$.

Since we want to avoid the derivative of $g$ in the final result
and assume only continuity, we integrate by parts:
\begin{align*}
v_g(x,t)
&=
\int_0^t g'(s)\phi(x,t-s)\,{\mathrm d}s=
\left[g(s)\phi(x,t-s)\right]_{s=0}^t
-
\int_0^t
g(s)\frac{\partial}{\partial s}\phi(x,t-s)\,{\mathrm d}s\\
&=
-g(t)\sgn x-g(0)\phi(x,t)
+
\int_0^t g(s)\Delta\phi(x,t-s)\,{\mathrm d}s.
\end{align*}
Finally, note that
\[
-g(0)\phi(x,t)
=
\sum_{k\in\mathbb N}g(0)
\bigl[K_{x-k}(t)-K_{x+k}(t)\bigr].
\]
Recalling that $u(x,t)=g(t)+v_h(x,t)+v_g(x,t)$ and using
the expression for $\Delta\phi$ from
Theorem~\ref{phi-function}, we arrive at the following result.

For convenience, introduce the Dirichlet kernel
\[
P_{x,y}(t):=K_{x-y}(t)-K_{x+y}(t),
\qquad x\in\mathbb N_0,\quad y\in\mathbb N,\quad t\geq0.
\]
By symmetry of the whole-line kernel, $P_{0,y}(t)=0$.

\begin{theorem}\label{explicit-solution}
Let $g\colon[0,\infty)\to\mathbb R$ be continuous and
$\{u_x\}_{x\in\mathbb N}$ be bounded.
Then the function
$u\colon\mathbb N_0\times[0,\infty)\to\mathbb R$ given by
\begin{align*}
u(x,t)
&=
g(t)(1-\sgn x)
+
\sum_{k\in\mathbb N}u_kP_{x,k}(t)+
\int_0^t g(s)P_{x,1}(t-s)\,{\mathrm d}s,
\end{align*}
is the unique solution of the initial-boundary value problem
\eqref{half-line}--\eqref{initial-boundary-conditions}
that is bounded on $\mathbb N_0\times[0,T]$ for every $T>0$.
\end{theorem}

Uniqueness follows from Theorem~\ref{uniquenessForward}.
We omit the existence proof, since a more general result
will be proved in Theorem~\ref{theo-nonhomogeneous}.

\begin{example}\label{nonuniform-example}
Consider the initial-boundary value problem
\eqref{half-line}--\eqref{initial-boundary-conditions}
with $g\equiv1$ and $u_x=0$ for $x\in\mathbb N$.
Using the previous theorem, we obtain, for $x\in\mathbb N$,
\begin{align*}
u(x,t)
&=
\int_0^t\Delta\phi(x,t-s)\,{\mathrm d}s
=
-\int_0^t
\frac{\partial}{\partial s}\phi(x,t-s)\,{\mathrm d}s=
-\left[\phi(x,t-s)\right]_{s=0}^t
=
\phi(x,t)-\phi(x,0)
=
\phi(x,t)+1.
\end{align*}
In view of Theorem~\ref{phi-function}, this result is not
surprising.
Note that $\lim_{t\to\infty}u(x,t)=1$ for every
$x\in\mathbb N_0$, but the convergence is not uniform
with respect to $x$.
The reader might wish to compare this solution with the
well-known solution of the corresponding PDE, which can
be expressed using the error function
(see e.g.~\cite[p.~57]{Strauss}).
\end{example}

\subsection{The Laplace transform method}\label{sec52}

Here we provide an alternative derivation of the explicit solution
from Theorem~\ref{explicit-solution}, based on the Laplace transform.
By the superposition principle, the function
\[
u_a(x,t)
=
\sum_{k\in\mathbb N}u_kP_{x,k}(t)
\]
is a solution of \eqref{half-line} satisfying
$u_a(0,t)=0$ for all $t\geq0$ and $u_a(x,0)=u_x$ for every
$x\in\mathbb N$.
The remaining task is to find the solution of
\[
\frac{\partial u}{\partial t}(x,t)
=
\Delta u(x,t),
\qquad x\in\mathbb N,\quad t\geq0,
\]
subject to
\[
u(0,t)=g(t),\qquad t\geq0,
\qquad
u(x,0)=0,\qquad x\in\mathbb N,
\]
where $g\colon[0,\infty)\to\mathbb R$ is continuous.

We proceed formally, assuming that the Laplace transforms and
the operations below are justified.
Consider
\[
G(s)=\int_0^\infty g(t){\mathrm e}^{-st}\,{\mathrm d}t,
\qquad
U(x,s)=\int_0^\infty u(x,t){\mathrm e}^{-st}\,{\mathrm d}t.
\]

Transforming the boundary condition gives $U(0,s)=G(s)$.
Transforming the differential equation gives
\[
sU(x,s)-u(x,0)
=
U(x+1,s)-2U(x,s)+U(x-1,s),
\qquad x\in\mathbb N.
\]
Since the initial data are zero, this reduces to
\[
sU(x,s)
=
U(x+1,s)-2U(x,s)+U(x-1,s),
\qquad x\in\mathbb N,
\]
or equivalently
\[
U(x+1,s)-(2+s)U(x,s)+U(x-1,s)=0,
\qquad x\in\mathbb N.
\]
For a fixed real $s>0$ in the domain of convergence,
this is a recurrence relation with characteristic equation
\[
r^2-(2+s)r+1=0.
\]
Its roots are
\[
r_{1,2}(s)
=
\frac{2+s\mp\sqrt{(2+s)^2-4}}{2}
=
1+\frac{s}{2}\mp\frac12\sqrt{s(s+4)}.
\]
Thus the general solution of the recurrence is
\[
U(x,s)=\alpha(s)r_1(s)^x+\beta(s)r_2(s)^x.
\]
The roots satisfy $r_1(s)r_2(s)=1$ and $r_2(s)>1$, so
$0<r_1(s)<1$.
Seeking a transform bounded in $x$, we take $\beta(s)=0$.
The boundary condition then gives
\[
U(x,s)=G(s)r_1(s)^x.
\]

By the convolution theorem for the Laplace transform,
\[
u(x,t)=\int_0^t P_x(t-s)g(s)\,{\mathrm d}s,
\qquad x\in\mathbb N,
\]
where $P_x$ is determined by
\[
\mathcal L\{P_x\}(s)=r_1(s)^x.
\]
It remains to identify this inverse transform.

Using the Laplace transform of the modified Bessel function
and the shift property, we obtain
\[
\mathcal L\{{\mathrm e}^{-2t}I_x(2t)\}(s)
=
\frac{2^x}
{\sqrt{s(s+4)}\bigl(\sqrt{s(s+4)}+2+s\bigr)^x}
=
\frac{r_1(s)^x}{\sqrt{s(s+4)}},
\qquad x\in\mathbb N_0.
\]
Recall that $r_1(s)$ satisfies
\[
(2+s)r_1(s)=1+r_1(s)^2.
\]
Hence
\begin{align*}
(s+4)r_1(s)
&=
1+r_1(s)^2+2r_1(s)
=
(r_1(s)+1)^2,\\
sr_1(s)
&=
1+r_1(s)^2-2r_1(s)
=
(r_1(s)-1)^2.
\end{align*}
Multiplying these identities and taking square roots, we get
\[
r_1(s)\sqrt{s(s+4)}
=
|r_1(s)^2-1|
=
1-r_1(s)^2,
\]
where we used $0<r_1(s)<1$.
Multiplication by $r_1(s)^{x-1}$ therefore gives
\[
r_1(s)^x\sqrt{s(s+4)}
=
r_1(s)^{x-1}-r_1(s)^{x+1},
\qquad x\in\mathbb N.
\]
Consequently,
\begin{align*}
&\mathcal L\{
{\mathrm e}^{-2t}[I_{x-1}(2t)-I_{x+1}(2t)]
\}(s)=
\frac{r_1(s)^{x-1}-r_1(s)^{x+1}}{\sqrt{s(s+4)}}
=
r_1(s)^x.
\end{align*}
We conclude that
\[
P_x(t)
=
{\mathrm e}^{-2t}
\bigl[I_{x-1}(2t)-I_{x+1}(2t)\bigr]
=
P_{x,1}(t),
\qquad x\in\mathbb N,
\]
which agrees with the result obtained using the reflection method.

\subsection{A semi-discrete version of the unified transform method}

\subsubsection{Homogeneous problem -- a formal derivation of the solution}

We present another approach for solving the semidiscrete homogeneous equation
\begin{equation}\label{1.1-again}
\frac{\partial u}{\partial t}(x,t)
=u(x+1,t)-2u(x,t)+u(x-1,t),
\quad x\in\mathbb N,\quad t\ge 0,
\end{equation}
subject to the boundary and initial value conditions
\begin{equation}\label{1.2-again}
u(0,t)=g(t),\quad t\ge 0,\quad
u(x,0)=u_x,\quad x\in\mathbb N.
\end{equation}
We start by assuming that $g\colon[0,\infty)\to\mathbb R$ is continuous
and $\{u_x\}_{x\in\mathbb N}$ belongs to $\ell^1$; the latter assumption
will be weakened later.

For the formal derivation, suppose that
$u(\cdot,t)\in\ell^1$ for every $t\geq0$ and that the
termwise differentiations below are justified.
Define the discrete transform by
\[
\hat u(z,t):=\sum_{x\in\mathbb N}u(x,t)z^x,
\qquad |z|\leq1.
\]
For each fixed $t$, this function is analytic in the open
unit disc and continuous on the closed unit disc.

Conversely, for any function $U(z,t)$ analytic in the open unit disc and continuous in
the closed unit disc,
we define the inverse discrete transform by the formula
\begin{align*}
\tilde U(x,t)
=
\frac{1}{2\pi{\mathrm i}}
\int_{|z|=1}U(z,t)z^{-x-1}\,{\mathrm d}z,
\quad x\in\mathbb N,
\end{align*}
where the unit circle has counterclockwise orientation. Clearly, we can replace the unit circle by any piecewise ${\mathcal C}^1$ closed
curve contained in the closed unit disc whose index with respect to the
origin is $1$.

By the residue theorem, $\tilde U(x,t)$ is the coefficient of $z^x$
in the Taylor expansion of $U(z,t)$ at the origin. Consequently, the
inverse discrete transform of the discrete transform of an
$\ell^1$-sequence is the sequence itself.

Suppose that $u\colon\mathbb N_0\times[0,\infty)\to\mathbb R$ satisfies \eqref{1.1-again}, and define the shifted functions $u^s(x,t)=u(x+1,t)$ and $u^{-s}(x,t)=u(x-1,t)$.
Clearly,
\begin{align*}
\widehat{u^s}(z,t) &= z^{-1}\hat u(z,t)-u(1,t),\\
\widehat{u^{-s}}(z,t) &= z\hat u(z,t)+zu(0,t).
\end{align*}
Applying the discrete transform formally to \eqref{1.1-again},
we obtain
\begin{equation*}
\frac{\partial\hat u}{\partial t}(z,t)
=
(z+z^{-1}-2)\hat u(z,t)+zu(0,t)-u(1,t).
\end{equation*}
Suppose now that $u$ satisfies \eqref{1.2-again}, and denote
$U(z,t)=\hat u(z,t)$. Then
\begin{equation*}
U(z,0)=\sum_{x\in\mathbb N}u_xz^x,\qquad |z|\leq1.
\end{equation*}
Letting $u(1,t)-u(0,t)=g_1(t)$, the previous differential equation becomes
\begin{equation*}
\frac{\partial U}{\partial t}(z,t)
=
(z+z^{-1}-2)U(z,t)+(z-1)g(t)-g_1(t).
\end{equation*}
For each fixed $z$ with $0<|z|\le 1$, this is a linear ordinary differential equation,
whose solution is
\begin{equation}\label{Fokas1}
U(z,t)
=
\exp\bigl((z+z^{-1}-2)t\bigr)U(z,0)
+(z-1)\tilde g(z,t)-\tilde g_1(z,t),
\end{equation}
where
\begin{align*}
U(z,0)&=\sum_{x\in\mathbb N}u_xz^x,\\
\tilde g(z,t)
&=
\int_0^t
g(s)\exp\bigl[(z+z^{-1}-2)(t-s)\bigr]\,{\mathrm d}s,\\
\tilde g_1(z,t)
&=
\int_0^t
g_1(s)\exp\bigl[(z+z^{-1}-2)(t-s)\bigr]\,{\mathrm d}s.
\end{align*}

Note that $\tilde g(z,t)$ and $\tilde g_1(z,t)$ are defined in the whole complex plane except at $0$. Clearly,
\begin{equation*}
\tilde g(z,t)=\tilde g(z^{-1},t),
\qquad
\tilde g_1(z,t)=\tilde g_1(z^{-1},t),
\qquad z\neq0.
\end{equation*}

The solution $u(x,t)$ is recovered from the inverse transform
\begin{align}\label{solution-recovered}
u(x,t)
=
\frac{1}{2\pi{\mathrm i}}
\int_{|z|=1}U(z,t)z^{-x-1}\,{\mathrm d}z,
\qquad x\in\mathbb N.
\end{align}

This is the discrete analogue of the integral formula that appears in the original method for PDEs. Similarly to that setting, we have to deal with the problem that the function $g_1$ is not given.
As in the PDE case, we use the symmetry of the phase function
\begin{equation*}
w(z)=z+z^{-1}-2,
\qquad
w(z^{-1})=w(z).
\end{equation*}

Substituting \eqref{Fokas1} into \eqref{solution-recovered}, we get
\begin{equation}\label{fokas1}
u(x,t)
=
\frac{1}{2\pi{\mathrm i}}
\int_{|z|=1}
\left[
\exp\bigl((z+z^{-1}-2)t\bigr)U(z,0)
+(z-1)\tilde g(z,t)-\tilde g_1(z,t)
\right]
z^{-x-1}\,{\mathrm d}z.
\end{equation}
For $|z|=1$, both $z$ and $z^{-1}$ lie on the unit circle.
Replacing $z$ by $z^{-1}$ in \eqref{Fokas1} and using the
symmetry of $w$, we obtain
\begin{equation}\label{fokas2}
U(z^{-1},t)
=
\exp\bigl[(z+z^{-1}-2)t\bigr]U(z^{-1},0)
+(z^{-1}-1)\tilde g(z,t)-\tilde g_1(z,t).
\end{equation}
Adding \eqref{Fokas1} and \eqref{fokas2} and solving for $\tilde g_1(z,t)$, we obtain
\begin{equation}\label{fokas3}
\begin{aligned}
\tilde g_1(z,t)
={}&
\frac12\exp\bigl[(z+z^{-1}-2)t\bigr]
\bigl(U(z,0)+U(z^{-1},0)\bigr)\\
&-\frac12\bigl(U(z,t)+U(z^{-1},t)\bigr)
+\frac12(z+z^{-1}-2)\tilde g(z,t).
\end{aligned}
\end{equation}

Substituting \eqref{fokas3} back into \eqref{Fokas1} gives
\begin{equation*}
\begin{aligned}
U(z,t)
={}&
\frac12\exp\bigl[(z+z^{-1}-2)t\bigr]
\bigl(U(z,0)-U(z^{-1},0)\bigr)\\
&+\frac12\bigl(U(z,t)+U(z^{-1},t)\bigr)
+\frac12(z-z^{-1})\tilde g(z,t).
\end{aligned}
\end{equation*}
Hence
\begin{equation}\label{fokas4}
U(z,t)
=
\exp\bigl[(z+z^{-1}-2)t\bigr]
\bigl(U(z,0)-U(z^{-1},0)\bigr)
+U(z^{-1},t)
+(z-z^{-1})\tilde g(z,t).
\end{equation}

For $x\geq1$, the term involving $U(z^{-1},t)$ has zero contribution to the inverse transform, since
\begin{equation*}
U(z^{-1},t)=\sum_{y\in\mathbb N}u(y,t)z^{-y}
\end{equation*}
contains only negative powers of $z$, and therefore
$U(z^{-1},t)z^{-x-1}$ has no $z^{-1}$ term.

Since
\begin{equation*}
U(z,0)-U(z^{-1},0)
=
\sum_{y\in\mathbb N}u_y\left(z^y-z^{-y}\right),
\end{equation*}
putting everything together gives
\begin{equation}\label{fokas5}
\begin{aligned}
u(x,t)
={}&
\frac{1}{2\pi{\mathrm i}}
\int_{|z|=1}
\left[
\sum_{y\in\mathbb N}
u_y\left(z^y-\frac{1}{z^y}\right)
\right]
\exp\left[\left(z+\frac{1}{z}-2\right)t\right]
\frac{{\mathrm d}z}{z^{x+1}}
\\
&+
\frac{1}{2\pi{\mathrm i}}
\int_{|z|=1}
\left(z-\frac{1}{z}\right)
\left\{
\int_0^t
g(s)
\exp\left[\left(z+\frac{1}{z}-2\right)(t-s)\right]
\,{\mathrm d}s
\right\}
\frac{{\mathrm d}z}{z^{x+1}},
\qquad x\geq1.
\end{aligned}
\end{equation}

This is a Fokas-type integral formula involving only the initial and
Dirichlet data.

By further manipulating the contour integrals, we can express the solution in
terms of the modified Bessel functions. This will allow
us to relax significantly the assumption on the data \(\{u_{x}\}_{x\in\mathbb N}\) to
 \begin{equation} \label{fokas6} 
u_{x} = O(x^{M}),\quad \mbox{ as }x \rightarrow \infty,
\end{equation}
for some nonnegative number \(M\).
We first finish the formal derivation, which will be rigorously justified in the next subsection under the relaxed assumption \eqref{fokas6}.

For \(z\in\mathbb C\setminus\{0\}\), we have the Laurent expansion
\begin{gather*}
\exp\left[\left(z+\frac{1}{z}\right)t\right]
={\mathrm e}^{tz}{\mathrm e}^{t/z} =\left(\sum_{n=0}^{\infty}\frac{t^n}{n!}z^n\right)
  \left(\sum_{m=0}^{\infty}\frac{t^m}{m!}\frac{1}{z^m}\right) =\sum_{n,m\in\mathbb N_0}
  \frac{t^{n+m}}{n!m!}z^{n-m} \\
=\sum_{n\ge m}\frac{t^{n+m}}{n!m!}z^{n-m}
  +\sum_{n<m}\frac{t^{n+m}}{n!m!}z^{n-m} =\sum_{\ell=0}^{\infty}
  \left[\sum_{m=0}^{\infty}
  \frac{t^{2m+\ell}}{m!(m+\ell)!}\right]z^\ell
  +\sum_{\ell=1}^{\infty}
  \left[\sum_{n=0}^{\infty}
  \frac{t^{2n+\ell}}{n!(n+\ell)!}\right]\frac{1}{z^\ell} \\
=\sum_{\ell\in\mathbb Z}
  \left[\sum_{n=0}^{\infty}
  \frac{t^{2n+|\ell|}}{n!(n+|\ell|)!}\right]z^\ell =\sum_{\ell\in\mathbb Z}I_\ell(2t)z^\ell,
\end{gather*}
where
\[
I_\ell(2t)=
\sum_{n=0}^{\infty}
\frac{t^{2n+|\ell|}}{n!(n+|\ell|)!},
\qquad \ell\in\mathbb Z.
\]
$I_\ell$ is the modified Bessel function of the first kind.
Note that $I_{-\ell}=I_\ell$ for integer $\ell$, and
\[
I_{\ell}(0)=\begin{cases}1,&\text{if }\ell=0,\\0,&\text{if }\ell\neq0.\end{cases}
\]
Now, for fixed \(x \in\mathbb N\), we have
\begin{gather}
u_{a}(x,t):= \frac{1}{2\pi {\mathrm i}}\int_{|z| = 1}{\left[ \sum_{y \in\mathbb N}u_{y}\left( z^{y} - \frac{1}{z^{y}} \right) \right]\exp\left[ \left( z + \frac{1}{z} - 2 \right)t \right]\frac{\,{\mathrm d}z}{z^{x + 1}}}\label{ua-definition}\\
= \frac{{\mathrm e}^{- 2t}}{2\pi {\mathrm i}}\int_{|z| = 1}{\left( \sum_{y \in\mathbb N}{u_{y}z^{y}} \right)\exp\left[ \left( z + \frac{1}{z} \right)t \right]\frac{\,{\mathrm d}z}{z^{x + 1}}} - \frac{{\mathrm e}^{- 2t}}{2\pi {\mathrm i}}\int_{|z| = 1}{\left( \sum_{y \in\mathbb N}\frac{u_{y}}{z^{y}} \right)\exp\left[ \left( z + \frac{1}{z} \right)t \right]\frac{\,{\mathrm d}z}{z^{x + 1}}}\nonumber\\
=:\sigma_{a,1}(x,t) - \sigma_{a,2}(x,t).\nonumber
\end{gather}
Using the residue theorem, we get
\begin{align*}
\sigma_{a,1}(x,t) &= \frac{{\mathrm e}^{- 2t}}{2\pi {\mathrm i}}\int_{|z| = 1}{\left[ \sum_{y \in\mathbb N,\ \ell \in\mathbb Z,\ y + \ell = x}{u_{y}I_{\ell}(2t)} \right]\frac{\,{\mathrm d}z}{z}}= {\mathrm e}^{- 2t}\sum_{\substack{y\in\mathbb N,\ \ell\in\mathbb Z\\
\ell+y=x}}{u_{y}I_{\ell}(2t)} = {\mathrm e}^{- 2t}\sum_{y = 1}^{\infty}{u_{y}I_{x - y}(2t)},\\
\sigma_{a,2}(x,t)
&=
\frac{{\mathrm e}^{- 2t}}{2\pi {\mathrm i}}\int_{|z| = 1}{\left[ \sum_{y \in\mathbb N,\ \ell \in\mathbb Z,\ \ell-y = x}{u_{y}I_{\ell}(2t)} \right]\frac{\,{\mathrm d}z}{z}}=
{\mathrm e}^{-2t}
\sum_{\substack{y\in\mathbb N,\ \ell\in\mathbb Z\\
\ell-y=x}}
u_yI_\ell(2t)=
{\mathrm e}^{-2t}\sum_{y=1}^{\infty}u_yI_{x+y}(2t).
\end{align*}
Thus, for \(x \in\mathbb N\),
\begin{equation}\label{ua-expression}
u_{a}(x,t) ={\mathrm e}^{- 2t}\sum_{y = 1}^{\infty}{u_{y}[ I_{y - x}(2t) - I_{y + x}(2t)]}.
\end{equation}
On the other hand,
\begin{align}
u_{g}(x,t):&= \frac{1}{2\pi {\mathrm i}}\int_{|z| = 1}{\left( z - \frac{1}{z} \right)\left\{ \int_{0}^{t}{g(s)\exp\left[ \left( z + \frac{1}{z} - 2 \right)(t - s) \right]\,{\mathrm d}s} \right\}\frac{\,{\mathrm d}z}{z^{x + 1}}}\label{ug-definition}\\
&= \frac{1}{2\pi {\mathrm i}}\int_{|z| = 1}{z\left\{ \int_{0}^{t}{{\mathrm e}^{- 2(t - s)}g(s)\exp\left[ \left( z + \frac{1}{z} \right)(t - s) \right]\,{\mathrm d}s} \right\}\frac{\,{\mathrm d}z}{z^{x + 1}}}\nonumber\\
&- \frac{1}{2\pi {\mathrm i}}\int_{|z| = 1}{\frac{1}{z}\left\{ \int_{0}^{t}{{\mathrm e}^{- 2(t - s)}g(s)\exp\left[ \left( z + \frac{1}{z} \right)(t - s) \right]\,{\mathrm d}s} \right\}\frac{\,{\mathrm d}z}{z^{x + 1}}}=:\sigma_{g,1}(x,t) - \sigma_{g,2}(x,t).\nonumber
\end{align}
For fixed \(x \in\mathbb N\), we have
\begin{align*}
\sigma_{g,1}(x,t) &= \frac{1}{2\pi {\mathrm i}}\int_{|z| = 1}{z\left\{ \int_{0}^{t}{{\mathrm e}^{- 2(t - s)}g(s)\left[ \sum_{\ell \in\mathbb Z}{I_{\ell}\bigl(2(t - s)\bigr)z^{\ell}} \right]\,{\mathrm d}s} \right\}\frac{\,{\mathrm d}z}{z^{x + 1}}}\\
&= \frac{1}{2\pi {\mathrm i}}\sum_{\ell \in\mathbb Z}\left\{ \left( \int_{|z| = 1}\frac{z^{1 + \ell}}{z^{x + 1}} \,{\mathrm d}z \right)\int_{0}^{t}{{\mathrm e}^{- 2(t - s)}g(s)I_{\ell}\bigl(2(t - s)\bigr)\,{\mathrm d}s} \right\} = \int_{0}^{t}{{\mathrm e}^{- 2(t - s)}g(s)I_{x - 1}\bigl(2(t - s)\bigr)\,{\mathrm d}s}
\end{align*}
and 
\begin{align*}
\sigma_{g,2}(x,t) &= \frac{1}{2\pi {\mathrm i}}\int_{|z| = 1}{\frac{1}{z}\left\{ \int_{0}^{t}{{\mathrm e}^{- 2(t - s)}g(s)\left[ \sum_{\ell \in\mathbb Z}{I_{\ell}\bigl(2(t - s)\bigr)z^{\ell}} \right]\,{\mathrm d}s} \right\}\frac{\,{\mathrm d}z}{z^{x + 1}}}\\
&=
\frac{1}{2\pi{\mathrm i}}
\sum_{\ell\in\mathbb Z}
\left\{
\left(
\int_{|z|=1}z^{\ell-x-2}\,{\mathrm d}z
\right)
\int_0^t
{\mathrm e}^{-2(t-s)}g(s)I_\ell\bigl(2(t-s)\bigr)\,{\mathrm d}s
\right\}=
\int_0^t
{\mathrm e}^{-2(t-s)}g(s)I_{x+1}\bigl(2(t-s)\bigr)\,{\mathrm d}s.
\end{align*}
Therefore,
\begin{equation}\label{ug-expression}
u_{g}(x,t) = \int_{0}^{t}{{\mathrm e}^{- 2(t - s)}g(s)[ I_{x - 1}\bigl(2(t - s)\bigr) - I_{x + 1}\bigl(2(t - s)\bigr)]\,{\mathrm d}s}.
\end{equation}
In view of the above computations, we see that \eqref{fokas5}  becomes
\begin{equation}\label{fokas7} 
\begin{aligned}
u(x,t) &= {\mathrm e}^{- 2t}\sum_{y = 1}^{\infty}{u_{y}[ I_{y - x}(2t) - I_{y + x}(2t)]}\\ 
&+ \int_{0}^{t}{{\mathrm e}^{- 2(t - s)}g(s)[ I_{x - 1}\bigl(2(t - s)\bigr) - I_{x + 1}\bigl(2(t - s)\bigr)]\,{\mathrm d}s},
\quad x \in\mathbb N, \quad t \geq 0,
\end{aligned}
\end{equation}
which agrees with the results obtained in the previous sections.

\subsubsection{Homogeneous problem -- rigorous justification}

We first give a detailed verification under the assumption
$\sum_{x\in\mathbb N}|u_x|<\infty$.
Later, by studying the convergence of the series in \eqref{fokas7}, we
will show that its right-hand side remains well-defined and gives a
solution under the weaker assumption \eqref{fokas6}.

Let $u_a(x,t)$ and $u_g(x,t)$, for $x\in\mathbb N$, be the functions
defined by \eqref{ua-definition} and \eqref{ug-definition}, respectively.
We extend them to $x=0$ by setting
\[
u_a(0,t):=0,
\qquad
u_g(0,t):=g(t).
\]
Thus
\[
u(x,t)=u_a(x,t)+u_g(x,t),
\qquad x\in\mathbb N_0.
\]
The function $u_a$ represents the contribution from the initial data
and has homogeneous boundary data, whereas $u_g$ represents the
contribution from the boundary data and has zero initial data.

We first verify the initial condition. For $x\in\mathbb N$, the residue
theorem gives
\begin{align*}
u_a(x,0)
&=
\frac{1}{2\pi{\mathrm i}}
\int_{|z|=1}
\left[
\sum_{y\in\mathbb N}
u_y\left(z^y-z^{-y}\right)
\right]
\frac{{\mathrm d}z}{z^{x+1}}\\
&=
\frac{1}{2\pi{\mathrm i}}
\sum_{y\in\mathbb N}u_y
\int_{|z|=1}z^{y-x-1}\,{\mathrm d}z
-
\frac{1}{2\pi{\mathrm i}}
\sum_{y\in\mathbb N}u_y
\int_{|z|=1}z^{-y-x-1}\,{\mathrm d}z=u_x.
\end{align*}
Moreover, $u_g(x,0)=0$. Therefore,
\[
u(x,0)=u_x,\qquad x\in\mathbb N.
\]
By our definitions at $x=0$, we also have
\[
u(0,t)=u_a(0,t)+u_g(0,t)=g(t).
\]

It remains to show that both $u_a$ and $u_g$ satisfy the semidiscrete
diffusion equation with their respective boundary conditions. More
precisely, we shall prove that
\begin{align*}
\frac{\partial u_a}{\partial t}(x,t)
&=
u_a(x+1,t)-2u_a(x,t)+u_a(x-1,t),\\
\frac{\partial u_g}{\partial t}(x,t)
&=
u_g(x+1,t)-2u_g(x,t)+u_g(x-1,t)
\end{align*}
for $x\in\mathbb N$.
The desired equation for $u=u_a+u_g$ will then follow by addition.

Define
\[
A(z,t):=
\left[\hat u(z,0)-\hat u(z^{-1},0)\right]
\exp\left[\left(z+z^{-1}-2\right)t\right].
\]
Then, for $x\in\mathbb N$, 
\[
u_a(x,t)=\frac{1}{2\pi\mathrm{i}}\int_{|z|=1}A(z,t)\frac{\mathrm{d}z}{z^{x+1}}.
\]
Direct differentiation gives
\[
\frac{\partial A}{\partial t}(z,t)
=
(z+z^{-1}-2)A(z,t).
\]
For $|z|=1$, we have
\[
\left|\hat u(z,0)-\hat u(z^{-1},0)\right|
\leq 2\sum_{y=1}^{\infty}|u_y|,
\qquad
z+z^{-1}-2\in[-4,0].
\]
Hence
\[
\left|\frac{\partial A}{\partial t}(z,t)\right|
\leq 8\sum_{y=1}^{\infty}|u_y|,
\qquad |z|=1,\quad t\geq0.
\]
This bound justifies differentiation under the integral sign in the
formula for $u_a$, yielding
\[
\frac{\partial u_a}{\partial t}(x,t)
=
\frac{1}{2\pi{\mathrm i}}
\int_{|z|=1}
(z+z^{-1}-2)A(z,t)\frac{{\mathrm d}z}{z^{x+1}}.
\]
For $x\geq2$, we immediately obtain
\begin{align*}
\frac{\partial u_a}{\partial t}(x,t)
&=
\frac{1}{2\pi{\mathrm i}}
\int_{|z|=1}
A(z,t)
\left(
\frac{1}{z^{x+2}}
-\frac{2}{z^{x+1}}
+\frac{1}{z^x}
\right)
\,{\mathrm d}z=
u_a(x+1,t)-2u_a(x,t)+u_a(x-1,t).
\end{align*}
For $x=1$, we have
\begin{align*}
\frac{\partial u_a}{\partial t}(1,t)
&=
\frac{1}{2\pi{\mathrm i}}
\int_{|z|=1}
A(z,t)
\left(
\frac{1}{z^3}-\frac{2}{z^2}+\frac{1}{z}
\right)
\,{\mathrm d}z=
u_a(2,t)-2u_a(1,t)
+
\frac{1}{2\pi{\mathrm i}}
\int_{|z|=1}A(z,t)\frac{{\mathrm d}z}{z}.
\end{align*}
Since $u_a(0,t)=0$, it remains to show that the last integral vanishes.

By the definition of $A$,
\begin{align*}
\int_{|z|=1}A(z,t)\frac{{\mathrm d}z}{z}
&=
\int_{|z|=1}
\hat u(z,0)
{\mathrm e}^{(z+z^{-1}-2)t}
\frac{{\mathrm d}z}{z}-
\int_{|z|=1}
\hat u(z^{-1},0)
{\mathrm e}^{(z+z^{-1}-2)t}
\frac{{\mathrm d}z}{z}.
\end{align*}
In the second integral, make the substitution $w=z^{-1}$.
The minus sign arising from ${\mathrm d}z/z=-{\mathrm d}w/w$
is canceled by the reversal of the orientation of the unit circle.
The two integrals are therefore equal, so their difference is zero.

Thus $u_a$ satisfies the semidiscrete diffusion equation for every
$x\in\mathbb N$.

We now turn to $u_g$. Put
\[
H(z,t):=(z-z^{-1})\tilde g(z,t).
\]
Then
\[
u_g(x,t)
=
\frac{1}{2\pi{\mathrm i}}
\int_{|z|=1}H(z,t)\frac{{\mathrm d}z}{z^{x+1}},
\qquad x\in\mathbb N.
\]
Since $g$ is continuous, differentiation of
\[
\tilde g(z,t)
=
\int_0^t
g(s){\mathrm e}^{(z+z^{-1}-2)(t-s)}\,{\mathrm d}s
\]
gives
\[
\frac{\partial\tilde g}{\partial t}(z,t)
=
g(t)+(z+z^{-1}-2)\tilde g(z,t).
\]
Consequently,
\[
\frac{\partial H}{\partial t}(z,t)
=
(z-z^{-1})g(t)+(z+z^{-1}-2)H(z,t).
\]
Both $H$ and $\partial_t H$ are continuous on
$\{z:|z|=1\}\times[0,T]$ for every $T>0$. Thus differentiation under
the contour integral is justified, and
\begin{align*}
\frac{\partial u_g}{\partial t}(x,t)
&=
\frac{g(t)}{2\pi{\mathrm i}}
\int_{|z|=1}
(z-z^{-1})\frac{{\mathrm d}z}{z^{x+1}}+
\frac{1}{2\pi{\mathrm i}}
\int_{|z|=1}
(z+z^{-1}-2)H(z,t)\frac{{\mathrm d}z}{z^{x+1}}.
\end{align*}
For $x\geq2$, the first integral vanishes, since
\[
(z-z^{-1})z^{-x-1}=z^{-x}-z^{-x-2}
\]
has no $z^{-1}$ term. Expanding the second integrand, we obtain
\begin{align*}
\frac{\partial u_g}{\partial t}(x,t)
&=
\frac{1}{2\pi{\mathrm i}}
\int_{|z|=1}
H(z,t)
\left(
\frac{1}{z^{x+2}}
-\frac{2}{z^{x+1}}
+\frac{1}{z^x}
\right)
\,{\mathrm d}z=
u_g(x+1,t)-2u_g(x,t)+u_g(x-1,t).
\end{align*}
For $x=1$, the residue theorem gives
\[
\frac{1}{2\pi{\mathrm i}}
\int_{|z|=1}
(z-z^{-1})\frac{{\mathrm d}z}{z^2}
=
\frac{1}{2\pi{\mathrm i}}
\int_{|z|=1}(z^{-1}-z^{-3})\,{\mathrm d}z
=1.
\]
Hence
\begin{align*}
\frac{\partial u_g}{\partial t}(1,t)
&=
g(t)+
\frac{1}{2\pi{\mathrm i}}
\int_{|z|=1}
H(z,t)
\left(
\frac{1}{z^3}-\frac{2}{z^2}+\frac{1}{z}
\right)
\,{\mathrm d}z\\
&=
g(t)+u_g(2,t)-2u_g(1,t)
+
\frac{1}{2\pi{\mathrm i}}
\int_{|z|=1}H(z,t)\frac{{\mathrm d}z}{z}.
\end{align*}
Since $u_g(0,t)=g(t)$, it remains to show that the last integral
vanishes.

The symmetry $\tilde g(z^{-1},t)=\tilde g(z,t)$ implies
\[
H(z^{-1},t)=-H(z,t).
\]
Making the substitution $w=z^{-1}$, and accounting for the cancellation
of the minus sign in ${\mathrm d}z/z=-{\mathrm d}w/w$ with the reversal
of orientation, we obtain
\[
\int_{|z|=1}H(z,t)\frac{{\mathrm d}z}{z}
=
\int_{|w|=1}H(w^{-1},t)\frac{{\mathrm d}w}{w}
=
-\int_{|w|=1}H(w,t)\frac{{\mathrm d}w}{w}.
\]
The left and right-hand side differ only by sign, so they must both vanish.
Thus $u_g$ satisfies the semidiscrete diffusion equation for every
$x\in\mathbb N$.

It follows that $u=u_a+u_g$ satisfies \eqref{1.1-again} together with
the initial-boundary conditions \eqref{1.2-again}.

The following lemma will be crucial for the proof of
Theorem~\ref{theo-1}.

\begin{lemma}\label{double-sum-lemma}
For each $x\in\mathbb N$ and $M,t\in[0,\infty)$, we have
\begin{align*}
\sum_{y=1}^{\infty}\sum_{n=0}^{\infty}
y^{M}\frac{t^{2n+|y-x|}}{n!(n+|y-x|)!}
&\leq
x^{M}\left(
{\mathrm e}^{t+t^{2}}
+2^{M}{\mathrm e}^{t^{2}}
\sum_{m=1}^{\infty}\frac{m^{M}t^{m}}{m!}
\right),\\
\sum_{y=1}^{\infty}\sum_{n=0}^{\infty}
y^{M}\frac{t^{2n+y+x}}{n!(n+y+x)!}
&\leq
{\mathrm e}^{t^{2}}
\sum_{m=1}^{\infty}\frac{m^{M}t^{m}}{m!}.
\end{align*}
\end{lemma}

\begin{proof}
For the first inequality, we split the sum into the parts corresponding
to $y\leq x$ and $y>x$. Since $y^M\leq x^M$ when $1\leq y\leq x$,
the substitution $m=x-y$ gives
\begin{equation*}
\sum_{y=1}^{x}\sum_{n=0}^{\infty}
y^{M}\frac{t^{2n+|y-x|}}{n!(n+|y-x|)!}
\leq
x^{M}\sum_{m=0}^{x-1}\sum_{n=0}^{\infty}
\frac{t^{2n+m}}{n!(n+m)!}\leq
x^{M}\sum_{m=0}^{\infty}\sum_{n=0}^{\infty}
\frac{t^{2n+m}}{n!m!}=x^{M}{\mathrm e}^{t+t^{2}}.
\end{equation*}
For $y>x$, set $m=y-x$. Since $x,m\geq1$, we have
$x+m\leq2xm$, and therefore
\begin{align*}
\sum_{y=x+1}^{\infty}\sum_{n=0}^{\infty}
y^{M}\frac{t^{2n+|y-x|}}{n!(n+|y-x|)!}
&=
\sum_{m=1}^{\infty}\sum_{n=0}^{\infty}
(x+m)^{M}\frac{t^{2n+m}}{n!(n+m)!}\leq
2^{M}x^{M}
\sum_{m=1}^{\infty}\sum_{n=0}^{\infty}
m^{M}\frac{t^{2n+m}}{n!m!}\\
&=
2^{M}x^{M}{\mathrm e}^{t^{2}}
\sum_{m=1}^{\infty}\frac{m^{M}t^{m}}{m!}.
\end{align*}
Adding these two estimates proves the first inequality.

For the second inequality, set $m=y+x$. Then $y=m-x\leq m$ and
$(n+m)!\geq m!$, so
\begin{align*}
\sum_{y=1}^{\infty}\sum_{n=0}^{\infty}
y^{M}\frac{t^{2n+y+x}}{n!(n+y+x)!}
&=
\sum_{m=x+1}^{\infty}\sum_{n=0}^{\infty}
(m-x)^{M}\frac{t^{2n+m}}{n!(n+m)!}\\
&\leq
\sum_{m=x+1}^{\infty}\sum_{n=0}^{\infty}
m^{M}\frac{t^{2n+m}}{n!m!}\leq
{\mathrm e}^{t^{2}}
\sum_{m=1}^{\infty}\frac{m^{M}t^{m}}{m!}.\tag*{\qedhere}
\end{align*}
\end{proof}

We can now state the main result of this subsection.
Recall that the Dirichlet diffusion kernel is
\[
P_{x,y}(t)
=
K_{x-y}(t)-K_{x+y}(t)
=
{\mathrm e}^{-2t}
\bigl[I_{x-y}(2t)-I_{x+y}(2t)\bigr],
\qquad x,y\in\mathbb N.
\]

\begin{theorem}\label{theo-1}
Suppose that \(g\colon[0,\infty)\to\mathbb R\) is continuous and
\(u_x\in\mathbb R\), \(x\in\mathbb N\). Then the following hold:
\begin{enumerate}
\item
If \(u_x=O(x^M)\) as \(x\to\infty\), for some 
\(M\in[0,\infty)\), then the function $u\colon\mathbb N_0\times[0,\infty)\to\mathbb R$ given by
\[
u(x,t)
=\begin{cases}
\displaystyle
\sum_{y=1}^{\infty}u_yP_{x,y}(t)
+
\int_0^t g(s)P_{x,1}(t-s)\,{\mathrm d}s,
&
x\in\mathbb N,\\
g(t),&x=0,
\end{cases}
\]
is well-defined and satisfies the semidiscrete diffusion equation
\eqref{1.1-again} with the initial-boundary conditions
\eqref{1.2-again}.

Furthermore, for every $T>0$,
$u(x,t)=O(x^M)$ as $x\to\infty$, uniformly for $t\in[0,T]$.
The function $u$ is the unique solution satisfying this
growth condition on every compact time interval.

In particular, if
\(\sup_{x\in\mathbb N}|u_x|<\infty\), then
$\sup_{\substack{x\in\mathbb N\\0\leq t\leq T}}
|u(x,t)|<\infty$.

\item
If $\sum_{x\in\mathbb N}|u_x|<\infty$, then, for
$x\in\mathbb N$ and $t\geq0$, the solution can alternatively
be expressed by the integral formula
\begin{gather*}
u(x,t)
=
\frac{1}{2\pi{\mathrm i}}
\int_{|z|=1}
\left[
\sum_{y\in\mathbb N}
u_y\left(z^y-\frac{1}{z^y}\right)
\right]
\exp\left[
\left(z+\frac{1}{z}-2\right)t
\right]
\frac{{\mathrm d}z}{z^{x+1}}
\\
+
\frac{1}{2\pi{\mathrm i}}
\int_{|z|=1}
\left(z-\frac{1}{z}\right)
\left\{
\int_0^t
g(s)
\exp\left[
\left(z+\frac{1}{z}-2\right)(t-s)
\right]
\,{\mathrm d}s
\right\}
\frac{{\mathrm d}z}{z^{x+1}}.
\end{gather*}

\item
For \(p>0\) and \(T>0\), if
$\sum_{x\in\mathbb N}|u_x|^p<\infty$,
then
$\sup_{0\leq t\leq T}
\sum_{x\in\mathbb N}|u(x,t)|^p<\infty$.

\item
If \(\lim_{x\to\infty}u_x=0\), then, for \(T>0\),
$\lim_{x\to\infty}u(x,t)=0$
uniformly for \(t\in[0,T]\).

\item
For every $N\in\mathbb N$ and $T>0$, the part $u_g$ given by
\eqref{ug-definition} satisfies
$\sup_{0\leq t\leq T}
\sum_{x\in\mathbb N}x^N|u_g(x,t)|<\infty$.
\end{enumerate}
\end{theorem}

\begin{proof}
We first justify the equivalence of the contour and Bessel
representations when the initial data belong to $\ell^1$.
Under this assumption, the series involved in the calculations leading
to \eqref{ua-expression} and \eqref{ug-expression} converge absolutely
and uniformly on the unit circle, uniformly for time in compact
intervals. Together with the continuity of $g$, this justifies the
interchanges of sums and integrals in those calculations.
Consequently, the two representations agree, proving assertion~2.
By the preceding verification of the contour formula, the Bessel
formula therefore gives a solution whenever the initial data belong
to $\ell^1$, in particular when they have finite support.

We now prove assertion~1 under the weaker assumption
$u_y=O(y^M)$.
Recall that
\[
K_\ell(t)={\mathrm e}^{-2t}I_\ell(2t),
\qquad \ell\in\mathbb Z,\quad t\geq0.
\]
Write
\[
u(x,t)=u_a(x,t)+u_g(x,t),
\]
where
\begin{align*}
u_a(x,t)
&=
\sum_{y=1}^{\infty}
u_yP_{x,y}(t),\\
u_g(x,t)
&=
\int_0^t g(s)P_{x,1}(t-s)\,{\mathrm d}s.
\end{align*}
We shall repeatedly use the estimate
\begin{equation}\label{K-factorial-bound}
0\leq K_\ell(t)
\leq I_\ell(2t)
\leq
{\mathrm e}^{T^2}\frac{T^{|\ell|}}{|\ell|!},
\qquad 0\leq t\leq T.
\end{equation}
Indeed, for $\ell\geq0$,
\[
I_\ell(2t)
=
\sum_{n=0}^{\infty}
\frac{t^{2n+\ell}}{n!(n+\ell)!}
\leq
\frac{t^\ell}{\ell!}
\sum_{n=0}^{\infty}\frac{t^{2n}}{n!}
=
{\mathrm e}^{t^2}\frac{t^\ell}{\ell!}.
\]

Fix $x\in\mathbb N$ and $T>0$, and choose $C>0$ such that
$|u_y|\leq Cy^M$ for every $y\in\mathbb N$.
Since the series for $I_\ell(2t)$ has nonnegative coefficients,
$0\leq I_\ell(2t)\leq I_\ell(2T)$ for $0\leq t\leq T$. Thus
\[
\left|
{\mathrm e}^{-2t}u_y
\bigl[I_{y-x}(2t)-I_{y+x}(2t)\bigr]
\right|
\leq
Cy^M\bigl[I_{y-x}(2T)+I_{y+x}(2T)\bigr].
\]
By Lemma~\ref{double-sum-lemma}, the sum of these bounds over
$y\in\mathbb N$ is finite. The Weierstrass $M$-test therefore shows
that the series defining $u_a(x,t)$ converges absolutely and uniformly
for $t\in[0,T]$.

Moreover, the two estimates in Lemma~\ref{double-sum-lemma} give
\begin{align*}
|u_a(x,t)|
&\leq
C\sum_{y=1}^{\infty}y^M
\bigl[I_{y-x}(2T)+I_{y+x}(2T)\bigr]\\
&\leq
Cx^M\left(
{\mathrm e}^{T+T^2}
+2^M{\mathrm e}^{T^2}
\sum_{m=1}^{\infty}\frac{m^MT^m}{m!}
\right)
+C{\mathrm e}^{T^2}
\sum_{m=1}^{\infty}\frac{m^MT^m}{m!}\\
&\leq
Cx^M\left(
{\mathrm e}^{T+T^2}
+(2^M+1){\mathrm e}^{T^2}
\sum_{m=1}^{\infty}\frac{m^MT^m}{m!}
\right)
=:C_Tx^M,
\end{align*}
where the last inequality uses $x^M\geq1$.
The constant $C_T$ is finite and independent of $x$ and $t$.

Let
\[
G_T:=\max_{0\leq s\leq T}|g(s)|.
\]
Using the definition of $u_g$ and
\eqref{K-factorial-bound}, we obtain
\begin{equation}\label{ug-factorial-bound}
|u_g(x,t)|
\leq
TG_T{\mathrm e}^{T^2}
\left(
\frac{T^{x-1}}{(x-1)!}
+
\frac{T^{x+1}}{(x+1)!}
\right),
\qquad 0\leq t\leq T.
\end{equation}
Since $T^k/k!\leq{\mathrm e}^T$ for every $k\geq0$, it follows that
\[
|u_g(x,t)|\leq2TG_T{\mathrm e}^{T+T^2}.
\]
Combining this with the estimate for $u_a$ and using $x^M\geq1$, we get
\[
|u(x,t)|
\leq
\left(C_T+2TG_T{\mathrm e}^{T+T^2}\right)x^M.
\]
Thus $u(x,t)=O(x^M)$ as $x\to\infty$, uniformly for $t\in[0,T]$.

It remains to verify the equation. For $r\in\mathbb N$, define the
truncated initial data
\[
u_y^{(r)}
=
\begin{cases}
u_y,&y\leq r,\\
0,&y>r.
\end{cases}
\]
For these finitely supported initial data, the corresponding function
$u^{(r)}$ satisfies the equation and the initial-boundary conditions,
as established at the beginning of the proof.
For each fixed $j\in\mathbb N$, the uniform convergence of the series
defining $u_a(j,t)$ implies that
\[
u^{(r)}(j,t)\longrightarrow u(j,t)
\]
uniformly for $t\in[0,T]$, since the boundary-data contribution $u_g$
is unchanged. At $j=0$, we have
$u^{(r)}(0,t)=u(0,t)=g(t)$ for every $r$.
Consequently, we may let $r\to\infty$ in the integral form
\[
u^{(r)}(x,t)
=
u_x^{(r)}
+
\int_0^t
\bigl[
u^{(r)}(x+1,s)-2u^{(r)}(x,s)+u^{(r)}(x-1,s)
\bigr]\,{\mathrm d}s.
\]
This gives
\[
u(x,t)
=
u_x+
\int_0^t
\bigl[
u(x+1,s)-2u(x,s)+u(x-1,s)
\bigr]\,{\mathrm d}s.
\]
For each fixed $j\in\mathbb N_0$, the function $u(j,\cdot)$ is
continuous, being the uniform limit on $[0,T]$ of the continuous
functions $u^{(r)}(j,\cdot)$. Thus the integrand is continuous,
and differentiation with respect to $t$ shows that $u$ satisfies
\eqref{1.1-again}. Moreover,
$u(x,0)=u_x$, while $u(0,t)=g(t)$ by definition. Uniqueness in the stated growth class follows from
Theorem~\ref{uniquenessForward}, applied on $[0,T]$ with
$m=1$, $a_{-1}=a_1=1$, and $a_0=-2$. This proves assertion 1.

We next prove assertion~3. Fix $p>0$ and suppose that
\[
\sum_{y=1}^{\infty}|u_y|^p<\infty.
\]
In particular, the initial sequence is bounded, so assertion~1 applies
with $M=0$. Extend the initial sequence oddly to $\mathbb Z$ by setting
\[
v_0=0,\qquad v_y=u_y,\qquad v_{-y}=-u_y,
\qquad y\in\mathbb N.
\]
Then
\[
u_a(x,t)=\sum_{y\in\mathbb Z}K_{x-y}(t)v_y.
\]
Moreover, substituting $z=1$ into the generating function
${\mathrm e}^{t(z+z^{-1})}=\sum_{\ell\in\mathbb Z}I_\ell(2t)z^\ell$ gives
\[
\sum_{\ell\in\mathbb Z}K_\ell(t)=1.
\]
If $p\geq1$, Young's inequality $\|a*b\|_{\ell^p(\mathbb Z)}
\leq\|a\|_{\ell^1(\mathbb Z)}\|b\|_{\ell^p(\mathbb Z)}$  yields
\begin{align*}
\sum_{x=1}^{\infty}|u_a(x,t)|^p\leq
\sum_{x\in\mathbb Z}
\left|\sum_{y\in\mathbb Z}K_{x-y}(t)v_y\right|^p\leq
\left(\sum_{\ell\in\mathbb Z}K_\ell(t)\right)^p
\sum_{y\in\mathbb Z}|v_y|^p=
\sum_{y\in\mathbb Z}|v_y|^p
=
2\sum_{y=1}^{\infty}|u_y|^p.
\end{align*}
If $0<p<1$, the triangle inequality and the subadditivity of
$r\mapsto r^p$ give
\[
\left|\sum_j a_j\right|^p
\leq \sum_j|a_j|^p.
\]
For infinite sums, this follows by applying the finite-sum inequality
to partial sums and passing to the limit.
Consequently,
\begin{align*}
\sum_{x=1}^{\infty}|u_a(x,t)|^p
&\leq
\sum_{x\in\mathbb Z}
\left|\sum_{y\in\mathbb Z}K_{x-y}(t)v_y\right|^p\\
&\leq
\sum_{x\in\mathbb Z}\sum_{y\in\mathbb Z}
K_{x-y}(t)^p|v_y|^p=
\sum_{y\in\mathbb Z}|v_y|^p
\sum_{x\in\mathbb Z}K_{x-y}(t)^p=
\left(\sum_{\ell\in\mathbb Z}K_\ell(t)^p\right)
\left(\sum_{y\in\mathbb Z}|v_y|^p\right).
\end{align*}
By \eqref{K-factorial-bound},
\[
\sup_{0\leq t\leq T}
\sum_{\ell\in\mathbb Z}K_\ell(t)^p
\leq
{\mathrm e}^{pT^2}
\left(
1+2\sum_{\ell=1}^{\infty}
\frac{T^{p\ell}}{(\ell!)^p}
\right)
<\infty,
\]
where convergence follows from the ratio test.
Thus, for the given $p>0$,
\[
\sup_{0\leq t\leq T}
\sum_{x=1}^{\infty}|u_a(x,t)|^p<\infty.
\]

By \eqref{ug-factorial-bound}, $|u_g(x,t)|$ is bounded, uniformly for
$t\in[0,T]$, by a sequence belonging to $\ell^p$ for every $p>0$.
Indeed, the ratio test gives
\[
\sum_{k=0}^{\infty}
\left(\frac{T^k}{k!}\right)^p<\infty.
\]
Consequently,
\[
\sup_{0\leq t\leq T}
\sum_{x=1}^{\infty}|u_g(x,t)|^p<\infty.
\]
Combining this with the estimate for $u_a$ and using
\[
|a+b|^p\leq
\begin{cases}
|a|^p+|b|^p,&0<p\leq1,\\
2^{p-1}(|a|^p+|b|^p),&p>1,
\end{cases}
\]
we obtain
\[
\sup_{0\leq t\leq T}
\sum_{x=1}^{\infty}|u(x,t)|^p<\infty.
\]
This proves assertion 3.

For assertion 4, suppose that $u_y\to0$. Then its odd extension
$v$ belongs to $c_0(\mathbb Z)$. Given $\varepsilon>0$, choose a
finitely supported sequence $v^{(r)}$ such that
\[
\|v-v^{(r)}\|_\infty<\varepsilon.
\]
Since $\sum_{\ell\in\mathbb Z}K_\ell(t)=1$,
\[
\left\|
K(t)*(v-v^{(r)})
\right\|_\infty
\leq\varepsilon
\]
uniformly for $t\in[0,T]$. For the finitely supported sequence $v^{(r)}$,
\eqref{K-factorial-bound} gives
\[
\sup_{0\leq t\leq T}
\left|(K(t)*v^{(r)})_x\right|
\leq
{\mathrm e}^{T^2}
\sum_{y\in\operatorname{supp}v^{(r)}}
|v_y^{(r)}|
\frac{T^{|x-y|}}{|x-y|!}
\longrightarrow0
\]
as $x\to\infty$, since the sum is finite.
Together with the preceding approximation estimate, this yields
\[
\limsup_{x\to\infty}
\sup_{0\leq t\leq T}|u_a(x,t)|
\leq\varepsilon.
\]
Since $\varepsilon>0$ was arbitrary, the required uniform convergence
follows. The same conclusion for $u_g$ follows
immediately from \eqref{ug-factorial-bound}. This proves
assertion 4.

Finally, for every $N\in\mathbb N$, estimate
\eqref{ug-factorial-bound} gives
\begin{align*}
\sup_{0\leq t\leq T}
\sum_{x=1}^{\infty}x^N|u_g(x,t)|
&\leq
G_TT{\mathrm e}^{T^2}
\sum_{x=1}^{\infty}
x^N
\left(
\frac{T^{x-1}}{(x-1)!}
+
\frac{T^{x+1}}{(x+1)!}
\right)
<\infty.
\end{align*}
This proves assertion 5.
\end{proof}

\begin{remark}
We briefly compare our transform convention with those
used in \cite{BW10,CD21}.
Biondini and Wang \cite{BW10} define the spatial transform by
\[
\hat q(z,t)=\sum_{n=1}^{\infty}q_n(t)z^{-n}.
\]
With $q_n(t)=u(n,t)$, this gives
\[
\hat q(z,t)=\hat u(z^{-1},t).
\]
Thus the two conventions are related by replacing $z$ by its reciprocal. In particular,
the exterior of the unit disk in their formulation
corresponds to the interior of the unit disk in ours.

Cisneros and Deconinck \cite{CD21} use the transform
\[
\hat q(k,t)
=
h\sum_{n=1}^{\infty}
q_n(t){\mathrm e}^{-{\mathrm i}knh},
\]
where $h>0$ is the spatial mesh size, with inversion formula
\[
q_n(t)
=
\frac{1}{2\pi}
\int_{-\pi/h}^{\pi/h}
{\mathrm e}^{{\mathrm i}knh}\hat q(k,t)\,{\mathrm d}k.
\]
The substitution
\[
z={\mathrm e}^{-{\mathrm i}kh}
\]
gives $\hat q(k,t)=h\hat u(z,t)$ and
${\mathrm d}k=-{\mathrm d}z/({\mathrm i}hz)$.
As $k$ increases from $-\pi/h$ to $\pi/h$, the variable
$z$ traverses the unit circle clockwise.
Reversing this orientation therefore yields
\[
q_n(t)
=
\frac{1}{2\pi{\mathrm i}}
\int_{|z|=1}\hat u(z,t)\frac{{\mathrm d}z}{z^{n+1}},
\]
with counterclockwise orientation, as in our convention.

For the centered semidiscrete diffusion equation, their
dispersion relation is
\[
W(k)
=
\frac{2-{\mathrm e}^{{\mathrm i}kh}
-{\mathrm e}^{-{\mathrm i}kh}}{h^2}.
\]
Consequently,
\[
-W(k)=\frac{z+z^{-1}-2}{h^2},
\]
and the symmetry $k\mapsto-k$ becomes $z\mapsto z^{-1}$.
Their time transform of the boundary data satisfies
\[
{\mathrm e}^{-W(k)t}
\int_0^t {\mathrm e}^{W(k)s}g(s)\,{\mathrm d}s
=
\int_0^t
g(s){\mathrm e}^{-W(k)(t-s)}\,{\mathrm d}s.
\]
Taking $h=1$, this is precisely $\tilde g(z,t)$.
These substitutions convert their representation for
the diffusion equation with Dirichlet boundary data into
the contour formula in Theorem~\ref{theo-1}.

\end{remark}

\subsubsection{Nonhomogeneous problem}

We conclude this section by deriving the solution of
the nonhomogeneous problem
\begin{align*}
\frac{\partial u}{\partial t}(x,t)
&=
u(x+1,t)-2u(x,t)+u(x-1,t)+f(x,t),
\quad x\in\mathbb N,\quad t\geq0,\\
u(0,t)&=g(t),
\quad t\geq0,\\
u(x,0)&=u_x,
\quad x\in\mathbb N.
\end{align*}
For the formal derivation, assume that
$g\colon[0,\infty)\to\mathbb R$ is continuous,
$(u_x)_{x\in\mathbb N}\in\ell^1$, and
$f\colon\mathbb N\times[0,\infty)\to\mathbb R$ satisfies
\[
\sum_{x=1}^{\infty}|f(x,t)|<\infty,
\qquad t\geq0,
\]
with continuous dependence on $t$ in the sense that
\[
\sum_{x=1}^{\infty}|f(x,t)-f(x,t_0)|
\longrightarrow0
\qquad\text{as }t\to t_0.
\]
These assumptions justify the transform calculations below.
After deriving the Bessel representation, we shall show that it
remains valid for initial data of polynomial growth and forcing
that is continuous in time at each spatial index and bounded
on $\mathbb N\times[0,R]$ for every $R>0$.

Define
\[
\hat f(z,t):=\sum_{x\in\mathbb N}f(x,t)z^x,
\qquad |z|\leq1.
\]
The series converges absolutely and uniformly for $|z|\leq1$, locally
uniformly with respect to $t$. By the residue theorem,
\[
f(x,t)
=
\frac{1}{2\pi{\mathrm i}}
\int_{|z|=1}\hat f(z,t)z^{-x-1}\,{\mathrm d}z,
\qquad x\in\mathbb N.
\]

As in the homogeneous case, put
\[
U(z,t):=\hat u(z,t),
\qquad
g_1(t):=u(1,t)-g(t),
\]
and let
\[
w(z):=z+z^{-1}-2.
\]
Applying the discrete transform to the equation gives
\[
\frac{\partial U}{\partial t}(z,t)
=
w(z)U(z,t)+(z-1)g(t)-g_1(t)+\hat f(z,t).
\]
Consequently,
\begin{equation}\label{nonhomogeneous-U}
U(z,t)
=
{\mathrm e}^{w(z)t}U(z,0)
+(z-1)\tilde g(z,t)
-\tilde g_1(z,t)
+\tilde f(z,t),
\end{equation}
where
\begin{align*}
\tilde g(z,t)
&=
\int_0^t g(s){\mathrm e}^{w(z)(t-s)}\,{\mathrm d}s,\\
\tilde g_1(z,t)
&=
\int_0^t g_1(s){\mathrm e}^{w(z)(t-s)}\,{\mathrm d}s,\\
\tilde f(z,t)
&=
\int_0^t\hat f(z,s){\mathrm e}^{w(z)(t-s)}\,{\mathrm d}s.
\end{align*}

Since $w(z^{-1})=w(z)$ and
$\tilde g(z^{-1},t)=\tilde g(z,t)$, replacing $z$ by $z^{-1}$ in
\eqref{nonhomogeneous-U} gives
\[
U(z^{-1},t)
=
{\mathrm e}^{w(z)t}U(z^{-1},0)
+(z^{-1}-1)\tilde g(z,t)
-\tilde g_1(z,t)
+\tilde f(z^{-1},t).
\]
Adding the last two identities and solving for $\tilde g_1(z,t)$, we
obtain
\begin{align*}
\tilde g_1(z,t)
={}&
\frac12{\mathrm e}^{w(z)t}
\bigl[U(z,0)+U(z^{-1},0)\bigr]-\frac12\bigl[U(z,t)+U(z^{-1},t)\bigr]
+\frac12w(z)\tilde g(z,t)+\frac12\bigl[\tilde f(z,t)+\tilde f(z^{-1},t)\bigr].
\end{align*}
Substituting this expression back into \eqref{nonhomogeneous-U} yields
\begin{align*}
U(z,t)
={}&
{\mathrm e}^{w(z)t}
\bigl[U(z,0)-U(z^{-1},0)\bigr]
+U(z^{-1},t)+(z-z^{-1})\tilde g(z,t)
+\tilde f(z,t)-\tilde f(z^{-1},t).
\end{align*}

For $x\in\mathbb N$, the term involving $U(z^{-1},t)$ has zero
contribution to the inverse transform. We therefore obtain
\begin{align}
u(x,t)
={}&
\frac{1}{2\pi{\mathrm i}}
\int_{|z|=1}
\left[
\sum_{y\in\mathbb N}
u_y\left(z^y-z^{-y}\right)
\right]
{\mathrm e}^{(z+z^{-1}-2)t}
\frac{{\mathrm d}z}{z^{x+1}}
\notag\\
&+
\frac{1}{2\pi{\mathrm i}}
\int_{|z|=1}
(z-z^{-1})
\left[
\int_0^t
g(s){\mathrm e}^{(z+z^{-1}-2)(t-s)}\,{\mathrm d}s
\right]
\frac{{\mathrm d}z}{z^{x+1}}
\notag\\
&+
\frac{1}{2\pi{\mathrm i}}
\int_{|z|=1}
\left[
\int_0^t
\bigl(\hat f(z,s)-\hat f(z^{-1},s)\bigr)
{\mathrm e}^{(z+z^{-1}-2)(t-s)}\,{\mathrm d}s
\right]
\frac{{\mathrm d}z}{z^{x+1}}.
\label{nonhomogeneous-contour-formula}
\end{align}

It remains to rewrite the last contour integral in terms of the
modified Bessel functions. Since
\[
{\mathrm e}^{(z+z^{-1}-2)(t-s)}
=
{\mathrm e}^{-2(t-s)}
\sum_{\ell\in\mathbb Z}I_\ell\bigl(2(t-s)\bigr)z^\ell,
\]
termwise integration and the residue theorem give
\begin{align*}
\frac{1}{2\pi{\mathrm i}}
\int_{|z|=1}
\left[
\int_0^t
\hat f(z,s)
{\mathrm e}^{(z+z^{-1}-2)(t-s)}\,{\mathrm d}s
\right]
\frac{{\mathrm d}z}{z^{x+1}}&=
\int_0^t
{\mathrm e}^{-2(t-s)}
\sum_{y=1}^{\infty}
f(y,s)I_{x-y}\bigl(2(t-s)\bigr)\,{\mathrm d}s,\\
\frac{1}{2\pi{\mathrm i}}
\int_{|z|=1}
\left[
\int_0^t
\hat f(z^{-1},s)
{\mathrm e}^{(z+z^{-1}-2)(t-s)}\,{\mathrm d}s
\right]
\frac{{\mathrm d}z}{z^{x+1}}&=
\int_0^t
{\mathrm e}^{-2(t-s)}
\sum_{y=1}^{\infty}
f(y,s)I_{x+y}\bigl(2(t-s)\bigr)\,{\mathrm d}s.
\end{align*}
Using $I_{x-y}=I_{y-x}$, the contribution of the nonhomogeneous term is
therefore
\[
\int_0^t
{\mathrm e}^{-2(t-s)}
\sum_{y=1}^{\infty}
f(y,s)
\bigl[I_{y-x}\bigl(2(t-s)\bigr)-I_{y+x}\bigl(2(t-s)\bigr)\bigr]
\,{\mathrm d}s.
\]

Putting everything together, we arrive at the formula
\begin{align}
u(x,t)
={}&
{\mathrm e}^{-2t}
\sum_{y=1}^{\infty}
u_y\bigl[I_{y-x}(2t)-I_{y+x}(2t)\bigr]
\notag\\
&+
\int_0^t
{\mathrm e}^{-2(t-s)}g(s)
\bigl[I_{x-1}\bigl(2(t-s)\bigr)-I_{x+1}\bigl(2(t-s)\bigr)\bigr]
\,{\mathrm d}s
\notag\\
&+
\int_0^t
{\mathrm e}^{-2(t-s)}
\sum_{y=1}^{\infty}
f(y,s)
\bigl[I_{y-x}\bigl(2(t-s)\bigr)-I_{y+x}\bigl(2(t-s)\bigr)\bigr]
\,{\mathrm d}s,
\label{nonhomogeneous-solution}
\end{align}
for $x\in\mathbb N$ and $t\geq0$, together with
$u(0,t)=g(t)$.

We now justify the Bessel formula under weaker assumptions
and extend it to an arbitrary diffusion coefficient $a>0$.

For $a>0$, recall the whole-line heat kernel
\[
K_\ell^{(a)}(t)={\mathrm e}^{-2at}I_\ell(2at),
\qquad \ell\in\mathbb Z,\quad t\geq0,
\]
and define
\[
P_{x,y}^{(a)}(t)
:=
K_{x-y}^{(a)}(t)-K_{x+y}^{(a)}(t),
\qquad x\in\mathbb N_0,\quad y\in\mathbb N.
\]
Thus
\[
K_\ell^{(a)}(t)=K_\ell(at),
\qquad
P_{x,y}^{(a)}(t)=P_{x,y}(at).
\]
As before, we omit the superscript when $a=1$.

\begin{theorem}\label{theo-nonhomogeneous}
Let $a>0$, let $g\colon[0,\infty)\to\mathbb R$ be continuous,
and suppose that $(u_x)_{x\in\mathbb N}$ satisfies
\[
u_x=O(x^M)
\qquad\text{as }x\to\infty
\]
for some $M\geq0$.
Let $f\colon\mathbb N\times[0,\infty)\to\mathbb R$ be such that
$t\mapsto f(x,t)$ is continuous for every $x\in\mathbb N$ and
\[
\sup_{\substack{x\in\mathbb N\\0\leq t\leq R}}
|f(x,t)|<\infty
\qquad\text{for every }R>0.
\]
Then the formula
\begin{align}
u(x,t)
={}&
\sum_{y=1}^{\infty}u_yP_{x,y}^{(a)}(t)
+
a\int_0^t g(s)P_{x,1}^{(a)}(t-s)\,{\mathrm d}s
+
\int_0^t
\sum_{y=1}^{\infty}f(y,s)P_{x,y}^{(a)}(t-s)\,{\mathrm d}s,
\label{nonhomogeneous-solution-a}
\end{align}
for $x\in\mathbb N$ and $t\geq0$, together with $u(0,t)=g(t)$,
defines a solution of
\[
\frac{\partial u}{\partial t}(x,t)
=
a\Delta u(x,t)+f(x,t),
\qquad x\in\mathbb N,\quad t\geq0,
\]
satisfying $u(x,0)=u_x$ for every $x\in\mathbb N$.
Moreover, for every $T>0$,
$u(x,t)=O(x^M)$ as $x\to\infty$, uniformly for $t\in[0,T]$.
The function $u$ is the unique solution satisfying this
growth condition on every compact time interval.

The same conclusion holds for complex-valued data.
\end{theorem}

\begin{proof}
We first consider $a=1$, when
\eqref{nonhomogeneous-solution-a} reduces to
\eqref{nonhomogeneous-solution}.
By Theorem~\ref{theo-1}, the first two terms in
\eqref{nonhomogeneous-solution} already give the required
initial-data and boundary-data contributions.
It remains to verify the forcing contribution
\[
u_f(x,t)
=
\int_0^t
\sum_{y=1}^{\infty}f(y,s)P_{x,y}(t-s)\,{\mathrm d}s.
\]
Extend $f$ oddly to $\mathbb Z$ by setting
\[
\widetilde f(k,t):=
\begin{cases}
f(k,t),&k\geq1,\\
0,&k=0,\\
-f(-k,t),&k\leq-1.
\end{cases}
\]
The extension is continuous in time at each spatial index
and bounded on $\mathbb Z\times[0,T]$ for every $T>0$.
By Theorem~\ref{duhamel}, the function
\[
w(x,t)
=
\int_0^t
\sum_{k\in\mathbb Z}
\widetilde f(k,s)K_{x-k}(t-s)\,{\mathrm d}s
\]
satisfies
\[
\frac{\partial w}{\partial t}(x,t)
=
\Delta w(x,t)+\widetilde f(x,t),
\qquad x\in\mathbb Z,\quad t\geq0,
\]
has zero initial data, and is bounded on
$\mathbb Z\times[0,T]$ for every $T>0$.

Splitting the absolutely convergent spatial sum into
positive and negative indices gives
\[
w(x,t)
=
\int_0^t
\sum_{y=1}^{\infty}
f(y,s)\bigl[K_{x-y}(t-s)-K_{x+y}(t-s)\bigr]
\,{\mathrm d}s.
\]
By the definition of $P_{x,y}$, this agrees with $u_f(x,t)$
for $x\in\mathbb N$ and gives $w(0,t)=0$.
Consequently, $u_f$, extended by $u_f(0,t)=0$, satisfies
the required equation with zero initial and boundary data.

Adding this contribution to the homogeneous solution supplied
by Theorem~\ref{theo-1} proves the assertion for $a=1$.

For general $a>0$, consider the problem with diffusion
coefficient $1$, initial data $(u_x)$, boundary data
$g(\tau/a)$, and forcing $a^{-1}f(x,\tau/a)$.
These transformed data satisfy the same assumptions.
Let $v(x,\tau)$ be the solution just constructed, and set
\[
u(x,t):=v(x,at).
\]
The chain rule gives
\[
\frac{\partial u}{\partial t}(x,t)
=
a\frac{\partial v}{\partial\tau}(x,at)
=
a\Delta u(x,t)+f(x,t).
\]
Also,
\[
u(x,0)=u_x,
\qquad
u(0,t)=g(t).
\]
Inserting the transformed data into
\eqref{nonhomogeneous-solution}, evaluating at $\tau=at$,
and substituting $\sigma=as$ in the time integrals gives
\eqref{nonhomogeneous-solution-a}.
The factor $a$ from ${\mathrm d}\sigma=a\,{\mathrm d}s$
remains in the boundary contribution and cancels the factor
$a^{-1}$ in the forcing contribution.

The initial-data and boundary-data contributions satisfy the
stated growth bound by Theorem~\ref{theo-1} and the time change
$\tau=at$.
The forcing contribution is bounded on
$\mathbb N_0\times[0,T]$ for every $T>0$, by
Theorem~\ref{duhamel} and the same time change.
Since $x^M\geq1$ for $x\in\mathbb N$, the full solution
has the required polynomial growth.

The difference of two solutions in this growth class satisfies
the homogeneous equation with zero initial and boundary data.
Theorem~\ref{uniquenessForward}, applied with
$m=1$, $a_{-1}=a_1=a$, and $a_0=-2a$, therefore gives uniqueness.
For complex-valued data, the same argument applies separately
to the real and imaginary parts.
\end{proof}

We now record some properties of the Dirichlet diffusion kernel.

\begin{lemma}\label{lem:dirichlet-kernel-properties}
Let $a>0$. For $x,y\in\mathbb N$ and $t\geq0$,
\[
P_{x,y}^{(a)}(t)\geq0,
\]
and
\begin{equation}\label{dirichlet-kernel-mass}
0\leq
\sum_{y=1}^{\infty}P_{x,y}^{(a)}(t)
=
\sum_{\ell=1-x}^{x}K_\ell^{(a)}(t)
\leq1.
\end{equation}
Moreover,
\begin{equation}\label{dirichlet-kernel-sine}
P_{x,y}^{(a)}(t)
=
\frac{2}{\pi}\int_0^\pi
{\mathrm e}^{-2at(1-\cos\theta)}
\sin(x\theta)\sin(y\theta)\,{\mathrm d}\theta.
\end{equation}
For every fixed $x\in\mathbb N$,
\begin{equation}\label{dirichlet-kernel-mass-decay}
\sum_{y=1}^{\infty}P_{x,y}^{(a)}(t)
\longrightarrow0
\qquad\text{as }t\to\infty.
\end{equation}
\end{lemma}

\begin{proof}
The symmetry $I_{-\ell}=I_\ell$ gives
$K_{-\ell}^{(a)}=K_\ell^{(a)}$ and
$P_{x,y}^{(a)}=P_{y,x}^{(a)}$.
Moreover,
\[
P_{x,y}^{(a)}(0)=\delta_{x,y},
\qquad x,y\in\mathbb N.
\]

We first prove nonnegativity.
The recurrence relation
\[
I_n(z)-I_{n+2}(z)
=
\frac{2(n+1)}{z}I_{n+1}(z),
\qquad n\geq0,\quad z>0,
\]
implies
\[
K_n^{(a)}(t)-K_{n+2}^{(a)}(t)
=
\frac{n+1}{at}K_{n+1}^{(a)}(t)\geq0,
\qquad n\geq0,\quad t>0.
\]
If $x\geq y$, repeated application gives
\[
K_{x-y}^{(a)}(t)
\geq K_{x-y+2}^{(a)}(t)
\geq\cdots\geq K_{x+y}^{(a)}(t).
\]
If $x<y$, symmetry gives
$K_{x-y}^{(a)}(t)=K_{y-x}^{(a)}(t)$, and
\[
K_{y-x}^{(a)}(t)
\geq K_{y-x+2}^{(a)}(t)
\geq\cdots\geq K_{y+x}^{(a)}(t).
\]
Thus $P_{x,y}^{(a)}(t)\geq0$ in both cases.
The same conclusion holds at $t=0$.

Next, recall the function
$\phi$ from Theorem~\ref{phi-function}, with diffusion coefficient $a$.
The defining formula for
$\phi$ and \eqref{phi-finite-sum} give
\[
\sum_{y=1}^{\infty}P_{x,y}^{(a)}(t)
=
-\phi(x,t)
=
\sum_{\ell=1-x}^{x}K_\ell^{(a)}(t).
\]
By Theorem~\ref{phi-function}, the sequence
$x\mapsto\phi(x,t)$ is nonincreasing, $\phi(0,t)=0$, and
$\phi(x,t)\geq-1$.
Hence
$0\leq\sum_{y=1}^{\infty}P_{x,y}^{(a)}(t)\leq1$.
For every fixed $x\in\mathbb N$, the same theorem gives
$\phi(x,t)\to0$ as $t\to\infty$, so
\[
\sum_{y=1}^{\infty}P_{x,y}^{(a)}(t)\longrightarrow0
\qquad\text{as }t\to\infty.
\]

Finally, we derive the sine representation.
The integral representation of the modified Bessel functions gives
\[
I_n(2at)
=
\frac1\pi\int_0^\pi
{\mathrm e}^{2at\cos\theta}\cos(n\theta)\,{\mathrm d}\theta,
\qquad n\in\mathbb Z,
\]
and therefore
\[
K_n^{(a)}(t)
=
\frac1\pi\int_0^\pi
{\mathrm e}^{-2at(1-\cos\theta)}
\cos(n\theta)\,{\mathrm d}\theta.
\]
Subtracting the expressions with indices $x-y$ and $x+y$,
and using
\[
\cos((x-y)\theta)-\cos((x+y)\theta)
=
2\sin(x\theta)\sin(y\theta),
\]
we obtain
\[
P_{x,y}^{(a)}(t)
=
\frac2\pi\int_0^\pi
{\mathrm e}^{-2at(1-\cos\theta)}
\sin(x\theta)\sin(y\theta)\,{\mathrm d}\theta.
\]
This completes the proof.
\end{proof}

\begin{lemma}\label{lem:integrated-dirichlet-kernel}
For $a>0$ and $x,y\in\mathbb N$,
\begin{equation}\label{integrated-dirichlet-kernel}
\int_0^\infty P_{x,y}^{(a)}(r)\,{\mathrm d}r
=
\frac{\min\{x,y\}}{a}.
\end{equation}
\end{lemma}

\begin{proof}
Integrating \eqref{dirichlet-kernel-sine} over $0\leq r\leq R$
gives
\[
\int_0^R P_{x,y}^{(a)}(r)\,{\mathrm d}r
=
\frac{1}{a\pi}
\int_0^\pi
\frac{\sin(x\theta)\sin(y\theta)}{1-\cos\theta}
\left[1-{\mathrm e}^{-2aR(1-\cos\theta)}\right]
\,{\mathrm d}\theta.
\]
The quotient has a continuous extension at $\theta=0$,
with value $2xy$.
When the right-hand side is rewritten as the difference of two integrals,
the integral containing the exponential tends to zero as
$R\to\infty$: its absolute value over $[0,\delta]$ is bounded
by a constant times $\delta$, while on $[\delta,\pi]$ the
exponential tends uniformly to zero.
Therefore,
\[
\int_0^\infty P_{x,y}^{(a)}(r)\,{\mathrm d}r
=
\frac{1}{a\pi}
\int_0^\pi
\frac{\sin(x\theta)\sin(y\theta)}{1-\cos\theta}
\,{\mathrm d}\theta.
\]

To evaluate this integral, use
\[
\frac{\sin(x\theta)}{\sin(\theta/2)}
=
2\sum_{j=0}^{x-1}\cos\bigl((j+\tfrac12)\theta\bigr)
\]
and
\[
\int_0^\pi
\cos\bigl((j+\tfrac12)\theta\bigr)
\cos\bigl((k+\tfrac12)\theta\bigr)\,{\mathrm d}\theta
=
\begin{cases}
\pi/2,&j=k,\\
0,&j\neq k.
\end{cases}
\]
Since $1-\cos\theta=2\sin^2(\theta/2)$, the finite-sum
identity above gives
\[
\frac{\sin(x\theta)\sin(y\theta)}{1-\cos\theta}
=
2\sum_{j=0}^{x-1}\sum_{k=0}^{y-1}
\cos\bigl((j+\tfrac12)\theta\bigr)
\cos\bigl((k+\tfrac12)\theta\bigr).
\]
Integrating and using the orthogonality relations, we obtain
\begin{align*}
\int_0^\pi
\frac{\sin(x\theta)\sin(y\theta)}{1-\cos\theta}
\,{\mathrm d}\theta
&=
2\sum_{j=0}^{x-1}\sum_{k=0}^{y-1}
\int_0^\pi
\cos\bigl((j+\tfrac12)\theta\bigr)
\cos\bigl((k+\tfrac12)\theta\bigr)
\,{\mathrm d}\theta=
2\sum_{j=0}^{\min\{x,y\}-1}\frac{\pi}{2}
=
\pi\min\{x,y\}.
\end{align*}
Indeed, the terms with $j\neq k$ vanish, while each term
with $j=k$ contributes $\pi/2$. There are exactly
$\min\{x,y\}$ such terms.
This proves \eqref{integrated-dirichlet-kernel}.
\end{proof}

\subsection{Asymptotic behavior}

We now analyze the asymptotic behavior of the solution of the initial-boundary value
problem, allowing a nonhomogeneous term that tends to zero.

\begin{theorem}\label{asymptotic-behavior}
Let $a>0$, let $(u_y)_{y\in\mathbb N}$ be a bounded sequence
of real numbers, and let $g\colon[0,\infty)\to\mathbb R$
be continuous and bounded.
Suppose that $f\colon\mathbb N\times[0,\infty)\to\mathbb R$
is continuous in $t$ for each fixed $y\in\mathbb N$ and satisfies
\[
\lim_{t\to\infty}f(y,t)=0
\qquad\text{for every }y\in\mathbb N.
\]
Assume also that there is a nonnegative sequence
$(h_y)_{y\in\mathbb N}$ such that
\[
|f(y,t)|\leq h_y,
\qquad y\in\mathbb N,\quad t\geq0,
\qquad
\sum_{y=1}^{\infty}h_y<\infty.
\]
Let $u$ be the solution 
of the problem
\begin{gather*}
\frac{\partial u}{\partial t}(x,t)
=
a\Delta u(x,t)+f(x,t),
\qquad x\in\mathbb N,\quad t\geq0,\\
u(x,0)=u_x,\qquad x\in\mathbb N,
\qquad
u(0,t)=g(t),\qquad t\geq0,
\end{gather*}
given by \eqref{nonhomogeneous-solution-a}.
Then, for every fixed $x\in\mathbb N$,
\[
\liminf_{t\to\infty}g(t)
\leq
\liminf_{t\to\infty}u(x,t)
\leq
\limsup_{t\to\infty}u(x,t)
\leq
\limsup_{t\to\infty}g(t).
\]
In particular, if $g(t)\to L$, then
$u(x,t)\to L$ for every fixed $x\in\mathbb N$.
\end{theorem}

\begin{proof}
We write \eqref{nonhomogeneous-solution-a} as
\[
u(x,t)=u_a(x,t)+u_g(x,t)+u_f(x,t),
\]
where
\begin{align*}
u_a(x,t)
&=
\sum_{y=1}^{\infty}u_yP_{x,y}^{(a)}(t),\\
u_g(x,t)
&=
a\int_0^t g(s)P_{x,1}^{(a)}(t-s)\,{\mathrm d}s,\\
u_f(x,t)
&=
\int_0^t
\sum_{y=1}^{\infty}f(y,s)P_{x,y}^{(a)}(t-s)
\,{\mathrm d}s,
\end{align*}
and  the kernel $P_{x,y}^{(a)}$ was introduced before
Theorem~\ref{theo-nonhomogeneous}.
These are the contributions from the initial data, boundary data,
and forcing, respectively.

Fix $x\in\mathbb N$.
By nonnegativity and \eqref{dirichlet-kernel-mass-decay},
\[
|u_a(x,t)|
\leq
\sup_{y\in\mathbb N}|u_y|
\sum_{y=1}^{\infty}P_{x,y}^{(a)}(t)
\longrightarrow0.
\]

We next show that $u_f(x,t)\to0$.
For each fixed $y$, put
\[
v_y(t):=
\int_0^t f(y,s)P_{x,y}^{(a)}(t-s)\,{\mathrm d}s.
\]
Given $\varepsilon>0$, choose $s_0\geq0$ such that
$|f(y,s)|\leq\varepsilon$ for $s\geq s_0$.
For $t\geq s_0$, positivity and
\eqref{integrated-dirichlet-kernel} give
\begin{align*}
|v_y(t)|
&\leq
\int_0^{s_0}|f(y,s)|P_{x,y}^{(a)}(t-s)\,{\mathrm d}s
+
\varepsilon\int_{s_0}^tP_{x,y}^{(a)}(t-s)\,{\mathrm d}s\leq
s_0h_y
\sup_{t-s_0\leq r\leq t}P_{x,y}^{(a)}(r)
+
\varepsilon\frac{\min\{x,y\}}{a}.
\end{align*}
The supremum tends to zero by
\eqref{dirichlet-kernel-mass-decay}.
Since $\varepsilon$ is arbitrary, $v_y(t)\to0$.

By \eqref{dirichlet-kernel-mass},
\[
|f(y,s)P_{x,y}^{(a)}(t-s)|\leq h_y.
\]
Since $\sum_y h_y<\infty$, this bound justifies termwise
integration, giving
\[
u_f(x,t)=\sum_{y=1}^{\infty}v_y(t).
\]
Furthermore, Lemma~\ref{lem:integrated-dirichlet-kernel} gives
\[
|v_y(t)|
\leq
h_y\int_0^\infty P_{x,y}^{(a)}(r)\,{\mathrm d}r
=
\frac{h_y\min\{x,y\}}{a}.
\]
Hence, for every $N\in\mathbb N$,
\[
\left|u_f(x,t)-\sum_{y=1}^N v_y(t)\right|
\leq
\frac{x}{a}\sum_{y>N}h_y.
\]
Given $\varepsilon>0$, choose $N$ sufficiently large that the right-hand side is less than $\varepsilon/2$.
For this fixed $N$, we have
$\sum_{y=1}^N v_y(t)\longrightarrow0$ as $t\to\infty$,
since each $v_y(t)$ tends to zero.
Thus there exists $t_0$ such that, for every $t\geq t_0$,
\[
|u_f(x,t)|
\leq
\left|u_f(x,t)-\sum_{y=1}^N v_y(t)\right|
+
\left|\sum_{y=1}^N v_y(t)\right|
<
\frac{\varepsilon}{2}+\frac{\varepsilon}{2}
=\varepsilon.
\]
This proves that $u_f(x,t)\to0$.
It remains to estimate the boundary contribution.
Set
\[
\underline L:=\liminf_{t\to\infty}g(t),
\qquad
\overline L:=\limsup_{t\to\infty}g(t).
\]
Given $\varepsilon>0$, choose $s_0\geq0$ such that
\[
\underline L-\varepsilon
\leq g(s)\leq
\overline L+\varepsilon,
\qquad s\geq s_0.
\]
For $t\geq s_0$, split
\[
u_g(x,t)
=
a\int_0^{s_0}g(s)P_{x,1}^{(a)}(t-s)\,{\mathrm d}s
+
a\int_{s_0}^t g(s)P_{x,1}^{(a)}(t-s)\,{\mathrm d}s.
\]
The first term tends to zero, since its absolute value
is bounded by
\[
a\left(\int_0^{s_0}|g(s)|\,{\mathrm d}s\right)
\sup_{t-s_0\leq r\leq t}P_{x,1}^{(a)}(r)
\longrightarrow0.
\]
Positivity gives
\begin{align*}
(\underline L-\varepsilon)
a\int_0^{t-s_0}P_{x,1}^{(a)}(r)\,{\mathrm d}r
&\leq
a\int_{s_0}^t g(s)P_{x,1}^{(a)}(t-s)\,{\mathrm d}s\leq
(\overline L+\varepsilon)
a\int_0^{t-s_0}P_{x,1}^{(a)}(r)\,{\mathrm d}r.
\end{align*}
By \eqref{integrated-dirichlet-kernel}, with $y=1$,
\[
a\int_0^{t-s_0}P_{x,1}^{(a)}(r)\,{\mathrm d}r
\longrightarrow1.
\]
Consequently,
\[
\underline L-\varepsilon
\leq\liminf_{t\to\infty}u_g(x,t)
\leq\limsup_{t\to\infty}u_g(x,t)
\leq\overline L+\varepsilon.
\]
Letting $\varepsilon\downarrow0$ and using
$u_a(x,t)\to0$ and $u_f(x,t)\to0$ proves the stated bounds.
If $g(t)\to L$, both outer bounds equal $L$, which proves
the final assertion.
\end{proof}

In general, the convergence $u(x,t)\to L$ in
Theorem~\ref{asymptotic-behavior} is not uniform in $x$;
see Example~\ref{nonuniform-example}.

\subsection{Asymptotic periodicity}

In this subsection we prove the asymptotic periodicity of the solution
of the semidiscrete diffusion equation under periodic Dirichlet data and
a forcing term that is periodic in time. Here we take $a=1$.

Proofs of asymptotic periodicity for related problems date back
at least to the 2008 paper of Bona and Fokas \cite{BF}, who
established this property for the linearized KdV equation.
For nonlinear equations, such as the focusing nonlinear
Schrödinger equation, the analysis is considerably more
involved, see \cite{BKSZ} and \cite{AK}. Sometimes a second frequency 
can be generated and the behavior becomes quasi-periodic (not necessarily periodic).

\begin{theorem}\label{thm:qualitative-asymptotic-periodicity}
Let $(u_y)_{y\in\mathbb N}$ be a bounded sequence of complex
numbers, and let $g\colon[0,\infty)\to\mathbb C$ be continuous.
Suppose that $f\colon\mathbb N\times[0,\infty)\to\mathbb C$
is continuous in $t$ for every fixed $x\in\mathbb N$ and is bounded
on $\mathbb N\times[0,R]$ for every $R>0$.
Assume that, for some $T>0$,
\[
g(t+T)=g(t),
\qquad
f(y,t+T)=f(y,t),
\qquad y\in\mathbb N,\quad t\geq0.
\]
Let $u$ be the solution given by
\eqref{nonhomogeneous-solution-a}.
Then, for every fixed $x\in\mathbb N$,
\[
\lim_{t\to\infty}[u(x,t+T)-u(x,t)]=0.
\]
\end{theorem}

\begin{proof}
We write
\[
u(x,t)=u_a(x,t)+u_g(x,t)+u_f(x,t),
\]
where
\begin{align*}
u_a(x,t)
&=\sum_{y=1}^{\infty}u_yP_{x,y}(t),\\
u_g(x,t)
&=\int_0^t g(s)P_{x,1}(t-s)\,{\mathrm d}s,\\
u_f(x,t)
&=\int_0^t
\sum_{y=1}^{\infty}f(y,s)P_{x,y}(t-s)\,{\mathrm d}s.
\end{align*}
These are the contributions from the initial data, boundary data,
and forcing, respectively.

Fix $x\in\mathbb N$.
By Lemma~\ref{lem:dirichlet-kernel-properties}, the kernel is
nonnegative and its total mass tends to zero. Hence
\[
|u_a(x,t)|
\leq
\sup_{y\in\mathbb N}|u_y|
\sum_{y=1}^{\infty}P_{x,y}(t)
\longrightarrow0,
\]
and therefore
\[
u_a(x,t+T)-u_a(x,t)\longrightarrow0.
\]

For the boundary contribution, split the time integral at $T$
and substitute $s=r+T$ in its second part. Periodicity gives
\begin{equation}\label{periodicity-g}
\begin{aligned}
u_g(x,t+T)
&=
\int_0^T g(s)P_{x,1}(t+T-s)\,{\mathrm d}s+
\int_0^t g(r+T)P_{x,1}(t-r)\,{\mathrm d}r\\
&=
\int_0^T g(s)P_{x,1}(t+T-s)\,{\mathrm d}s
+u_g(x,t).
\end{aligned}
\end{equation}
Thus, with
\[
G_T:=\max_{0\leq s\leq T}|g(s)|,
\]
we have
\[
|u_g(x,t+T)-u_g(x,t)|
\leq
TG_T\sup_{t\leq r\leq t+T}P_{x,1}(r)
\longrightarrow0
\qquad\text{as }t\to\infty
\]
by
\eqref{dirichlet-kernel-mass-decay}. Therefore,
\[
u_g(x,t+T)-u_g(x,t)\longrightarrow0.
\]

The same splitting and substitution, using periodicity of $f$,
give
\begin{equation}\label{periodicity-f}
u_f(x,t+T)-u_f(x,t)
=
\int_0^T
\sum_{y=1}^{\infty}
f(y,s)P_{x,y}(t+T-s)\,{\mathrm d}s.
\end{equation}
Put
\[
F_T:=
\sup_{\substack{y\in\mathbb N\\0\leq s\leq T}}
|f(y,s)|<\infty.
\]
Then
\[
|u_f(x,t+T)-u_f(x,t)|
\leq
TF_T
\sup_{t\leq r\leq t+T}
\sum_{y=1}^{\infty}P_{x,y}(r)
\longrightarrow0.
\]
Adding the three limits proves the assertion.
\end{proof}

The preceding theorem establishes asymptotic periodicity under
boundedness assumptions. Here asymptotic periodicity means that the difference
$u(x,t+T)-u(x,t)$ tends to zero at each fixed lattice point.
It does not, under the present assumptions, assert convergence
to a periodic function.

Under additional second-moment assumptions,
we can identify the leading terms explicitly.

\begin{theorem}\label{thm:periodicity-asymptotics}
Under the assumptions of
Theorem~\ref{thm:qualitative-asymptotic-periodicity}, suppose additionally
that
\[
\sum_{y=1}^{\infty}y^2|u_y|<\infty,
\qquad
\sup_{0\leq s\leq T}
\sum_{y=1}^{\infty}y^2|f(y,s)|<\infty.
\]
Let $u_a$, $u_g$, and $u_f$ denote the contributions to $u$
from the initial data, boundary data, and forcing, respectively.
Then, for each fixed $x\in\mathbb N$, as $t\to\infty$,
\begin{align*}
u_a(x,t)
&=
\frac{x}{2\sqrt{\pi}}
\left(\sum_{y=1}^{\infty}yu_y\right)t^{-3/2}
+o(t^{-2}),\\
u_g(x,t+T)-u_g(x,t)
&=
\frac{x}{2\sqrt{\pi}}
\left(\int_0^Tg(s)\,{\mathrm d}s\right)t^{-3/2}
+O(t^{-5/2}),\\
u_f(x,t+T)-u_f(x,t)
&=
\frac{x}{2\sqrt{\pi}}
\left\{
\int_0^T
\left[\sum_{y=1}^{\infty}yf(y,s)\right]
\,{\mathrm d}s
\right\}t^{-3/2}
+o(t^{-2}).
\end{align*}
Consequently,
\[
u(x,t+T)-u(x,t)
=
\frac{x}{2\sqrt{\pi}}
\left[
\int_0^T
\left(g(s)+\sum_{y=1}^{\infty}yf(y,s)\right)
\,{\mathrm d}s
\right]t^{-3/2}
+o(t^{-2}).
\]
\end{theorem}

\begin{proof}
Fix $x\in\mathbb N$, and put
\[
q(\theta):=2(1-\cos\theta),
\qquad -\pi\leq\theta\leq\pi.
\]
By Lemma~\ref{lem:dirichlet-kernel-properties}, with $a=1$,
\begin{equation}\label{pxy-formula}
P_{x,y}(t)
=
\frac{2}{\pi}\int_0^\pi
{\mathrm e}^{-tq(\theta)}
\sin(x\theta)\sin(y\theta)\,{\mathrm d}\theta.
\end{equation}

We next establish an elementary asymptotic estimate.
If $H$ is continuous on $[0,\pi]$ and
\[
H(\theta)=A\theta^2+o(\theta^3)
\qquad\text{as }\theta\downarrow0,
\]
then
\begin{equation}\label{sine-laplace-estimate}
\frac{2}{\pi}
\int_0^\pi
{\mathrm e}^{-tq(\theta)}H(\theta)\,{\mathrm d}\theta
=
\frac{A}{2\sqrt{\pi}}t^{-3/2}+o(t^{-2}).
\end{equation}
If $H(\theta)=A\theta^2+O(\theta^4)$, the remainder
is $O(t^{-5/2})$.

To prove this, observe that, for $0\leq\theta\leq\pi$,
\[
\frac{4}{\pi^2}\theta^2\leq q(\theta)\leq\theta^2,
\qquad
0\leq\theta^2-q(\theta)\leq\frac{\theta^4}{12}.
\]
The first estimate follows from
$q(\theta)=4\sin^2(\theta/2)$ and
$\sin(\theta/2)\geq\theta/\pi$.
For the second estimate, use
\[
0\leq r-\sin r
=
\int_0^r(1-\cos v)\,{\mathrm d}v
\leq\frac{r^3}{6}
\]
and
\[
\theta^2-q(\theta)
=
2\int_0^\theta(r-\sin r)\,{\mathrm d}r.
\]
Writing $c=4/\pi^2$ and using $1-{\mathrm e}^{-r}\leq r$
for $r\geq0$, we obtain
\begin{align*}
0
&\leq
{\mathrm e}^{-tq(\theta)}-{\mathrm e}^{-t\theta^2}=
{\mathrm e}^{-tq(\theta)}
\left(1-{\mathrm e}^{-t(\theta^2-q(\theta))}\right)\leq
t(\theta^2-q(\theta)){\mathrm e}^{-tq(\theta)}\leq
\frac{t\theta^4}{12}{\mathrm e}^{-ct\theta^2}.
\end{align*}
Therefore,
\begin{align*}
&\left|
\int_0^\pi\theta^2{\mathrm e}^{-tq(\theta)}\,{\mathrm d}\theta
-
\int_0^\pi\theta^2{\mathrm e}^{-t\theta^2}\,{\mathrm d}\theta
\right|
\leq
\frac{t}{12}
\int_0^\infty\theta^6{\mathrm e}^{-ct\theta^2}\,{\mathrm d}\theta
=
\frac{t^{-5/2}}{12c^{7/2}}
\int_0^\infty r^6{\mathrm e}^{-r^2}\,{\mathrm d}r=
O(t^{-5/2}).
\end{align*}
The substitution $r=\sqrt{t}\,\theta$ gives
\[
\int_0^\infty\theta^2{\mathrm e}^{-t\theta^2}\,{\mathrm d}\theta
=
t^{-3/2}\int_0^\infty r^2{\mathrm e}^{-r^2}\,{\mathrm d}r
=
\frac{\sqrt{\pi}}4t^{-3/2}.
\]
Moreover, for $\theta\geq\pi$,
\[
{\mathrm e}^{-t\theta^2}
\leq
{\mathrm e}^{-t\pi^2/2}{\mathrm e}^{-t\theta^2/2}.
\]
Consequently,
\[
0\leq
\int_\pi^\infty\theta^2{\mathrm e}^{-t\theta^2}\,{\mathrm d}\theta
\leq
{\mathrm e}^{-t\pi^2/2}
\int_0^\infty\theta^2{\mathrm e}^{-(t/2)\theta^2}
\,{\mathrm d}\theta
=
\frac{\sqrt{\pi}}4
\left(\frac{2}{t}\right)^{3/2}
{\mathrm e}^{-t\pi^2/2}.
\]
This bound is $o(t^{-5/2})$ as $t\to\infty$, since
$t{\mathrm e}^{-t\pi^2/2}\to0$. Subtracting this tail from the integral over $[0,\infty)$, we obtain
\begin{align*}
\int_0^\pi\theta^2{\mathrm e}^{-t\theta^2}\,{\mathrm d}\theta
&=
\int_0^\infty\theta^2{\mathrm e}^{-t\theta^2}\,{\mathrm d}\theta
-
\int_\pi^\infty\theta^2{\mathrm e}^{-t\theta^2}\,{\mathrm d}\theta=
\frac{\sqrt{\pi}}4t^{-3/2}+o(t^{-5/2}).
\end{align*}
Combining this with the preceding estimate gives
\[
\int_0^\pi\theta^2{\mathrm e}^{-tq(\theta)}\,{\mathrm d}\theta
=
\frac{\sqrt{\pi}}4t^{-3/2}+O(t^{-5/2}).
\]

Now write $R(\theta)=H(\theta)-A\theta^2$.
Given $\varepsilon>0$, choose $\delta\in(0,\pi)$ such that
$|R(\theta)|\leq\varepsilon\theta^3$ on $[0,\delta]$.
Then
\[
\left|
\int_0^\delta
{\mathrm e}^{-tq(\theta)}R(\theta)\,{\mathrm d}\theta
\right|
\leq
\varepsilon\int_0^\infty
\theta^3{\mathrm e}^{-ct\theta^2}\,{\mathrm d}\theta
=
\frac{\varepsilon}{2c^2}t^{-2}.
\]
On $[\delta,\pi]$, the function $q$ has a positive minimum,
so the remaining integral is exponentially small.
It follows that
\[
\limsup_{t\to\infty}
t^2\left|
\int_0^\pi
{\mathrm e}^{-tq(\theta)}R(\theta)\,{\mathrm d}\theta
\right|
\leq\frac{\varepsilon}{2c^2}.
\]
Since $\varepsilon>0$ is arbitrary and the expression inside
the limit superior is nonnegative, it follows that
\[
\int_0^\pi
{\mathrm e}^{-tq(\theta)}R(\theta)\,{\mathrm d}\theta
=
o(t^{-2}).
\]
Using $H(\theta)=A\theta^2+R(\theta)$ and the preceding
estimate for the quadratic term, we therefore obtain
\begin{align*}
\frac{2}{\pi}\int_0^\pi
{\mathrm e}^{-tq(\theta)}H(\theta)\,{\mathrm d}\theta
&=
\frac{2A}{\pi}\int_0^\pi
\theta^2{\mathrm e}^{-tq(\theta)}\,{\mathrm d}\theta
+
\frac{2}{\pi}\int_0^\pi
{\mathrm e}^{-tq(\theta)}R(\theta)\,{\mathrm d}\theta\\
&=
\frac{A}{2\sqrt{\pi}}t^{-3/2}
+O(t^{-5/2})+o(t^{-2})=
\frac{A}{2\sqrt{\pi}}t^{-3/2}+o(t^{-2}),
\end{align*}
which proves \eqref{sine-laplace-estimate}.
If $R(\theta)=O(\theta^4)$, the same argument uses
\[
\int_0^\infty
\theta^4{\mathrm e}^{-ct\theta^2}\,{\mathrm d}\theta
=
O(t^{-5/2})
\]
and gives the stronger remainder.

We shall also use the following consequence of a finite second
moment. For any sequence $(a_y)$ satisfying
$\sum_{y=1}^{\infty}y^2|a_y|<\infty$, we have
\begin{equation}\label{sine-series-expansion}
\sum_{y=1}^{\infty}a_y\sin(y\theta)
=
\theta\sum_{y=1}^{\infty}ya_y+o(\theta^2)
\qquad\text{as }\theta\downarrow0.
\end{equation}
Indeed,
\[
\frac{1}{\theta^2}
\sum_{y=1}^{\infty}a_y[\sin(y\theta)-y\theta]
\longrightarrow0.
\]
To justify the limit, each summand tends to zero, and the bound
$|\sin r-r|\leq2r^2$ gives the summable majorant
$2y^2|a_y|$.

For the initial contribution, \eqref{pxy-formula} gives
\[
u_a(x,t)
=
\frac{2}{\pi}
\int_0^\pi
{\mathrm e}^{-tq(\theta)}
\sin(x\theta)
\left[\sum_{y=1}^{\infty}u_y\sin(y\theta)\right]
\,{\mathrm d}\theta.
\]
The interchange of the series and integral is justified by
$\sum_{y=1}^{\infty}|u_y|<\infty$.
By \eqref{sine-series-expansion} and
$\sin(x\theta)=x\theta+O(\theta^3)$,
\[
\sin(x\theta)\sum_{y=1}^{\infty}u_y\sin(y\theta)
=
x\left(\sum_{y=1}^{\infty}yu_y\right)\theta^2
+o(\theta^3).
\]
Applying \eqref{sine-laplace-estimate} proves
\begin{equation}\label{initial-expansion}
u_a(x,t)
=
\frac{x}{2\sqrt{\pi}}
\left(\sum_{y=1}^{\infty}yu_y\right)t^{-3/2}
+o(t^{-2}).
\end{equation}

For the boundary contribution, the periodicity identity \eqref{periodicity-g} gives
\[
u_g(x,t+T)-u_g(x,t)
=
\int_0^T g(s)P_{x,1}(t+T-s)\,{\mathrm d}s.
\]
Using \eqref{pxy-formula}, we obtain
\[
u_g(x,t+T)-u_g(x,t)
=
\frac{2}{\pi}
\int_0^\pi
{\mathrm e}^{-tq(\theta)}
\sin(x\theta)\sin\theta\,J_g(\theta)
\,{\mathrm d}\theta,
\]
where
\[
J_g(\theta):=
\int_0^T g(s){\mathrm e}^{-q(\theta)(T-s)}\,{\mathrm d}s.
\]
Since
\[
\left|{\mathrm e}^{-q(\theta)(T-s)}-1\right|
\leq q(\theta)(T-s),
\]
we have
\[
J_g(\theta)
=
\int_0^Tg(s)\,{\mathrm d}s+O(\theta^2).
\]
Consequently,
\[
\sin(x\theta)\sin\theta\,J_g(\theta)
=
x\left(\int_0^Tg(s)\,{\mathrm d}s\right)\theta^2
+O(\theta^4).
\]
The stronger version of \eqref{sine-laplace-estimate} proves
the asserted boundary asymptotic.

Finally, write
\[
M_2:=
\sup_{0\leq s\leq T}
\sum_{y=1}^{\infty}y^2|f(y,s)|<\infty.
\]
For every $N\geq1$,
\[
\sup_{0\leq s\leq T}
\sum_{y>N}y|f(y,s)|
\leq\frac{M_2}{N+1}.
\]
Thus the series $\sum_y yf(y,s)$, and also $\sum_y f(y,s)$,
converge absolutely and uniformly on $[0,T]$.
Their sums are continuous, and the interchanges of sums and
integrals below are justified.

The forcing periodicity identity \eqref{periodicity-f} and
\eqref{pxy-formula} give
\[
u_f(x,t+T)-u_f(x,t)
=
\frac{2}{\pi}
\int_0^\pi
{\mathrm e}^{-tq(\theta)}
\sin(x\theta)S_f(\theta)\,{\mathrm d}\theta,
\]
where
\[
S_f(\theta):=
\int_0^T
\left[\sum_{y=1}^{\infty}f(y,s)\sin(y\theta)\right]
{\mathrm e}^{-q(\theta)(T-s)}\,{\mathrm d}s.
\]
Put
\[
b_y:=\int_0^T f(y,s)\,{\mathrm d}s.
\]
For every $N$,
\[
\sum_{y=1}^N y^2|b_y|
\leq
\int_0^T\sum_{y=1}^N y^2|f(y,s)|\,{\mathrm d}s
\leq TM_2.
\]
Hence $\sum_{y=1}^{\infty}y^2|b_y|<\infty$, and
\eqref{sine-series-expansion} yields
\[
\sum_{y=1}^{\infty}b_y\sin(y\theta)
=
\theta\int_0^T
\left[\sum_{y=1}^{\infty}yf(y,s)\right]\,{\mathrm d}s
+o(\theta^2).
\]
Moreover, termwise integration gives
\[
S_f(\theta)-\sum_{y=1}^{\infty}b_y\sin(y\theta)
=
\int_0^T
\left[\sum_{y=1}^{\infty}f(y,s)\sin(y\theta)\right]
\left[{\mathrm e}^{-q(\theta)(T-s)}-1\right]
\,{\mathrm d}s.
\]
Here termwise integration is justified by uniform convergence,
since
\[
\sup_{0\leq s\leq T}
\sum_{y>N}|f(y,s)\sin(y\theta)|
\leq
\frac{M_2}{(N+1)^2}.
\]
Using $1-{\mathrm e}^{-r}\leq r$ for $r\geq0$ and
$|\sin(y\theta)|\leq y|\theta|$, we obtain
\begin{gather*}
\left|
S_f(\theta)-\sum_{y=1}^{\infty}b_y\sin(y\theta)
\right|
\leq
\int_0^T
\sum_{y=1}^{\infty}|f(y,s)|\,|\sin(y\theta)|
\left[1-{\mathrm e}^{-q(\theta)(T-s)}\right]
\,{\mathrm d}s\\
\leq
q(\theta)|\theta|
\int_0^T(T-s)
\sum_{y=1}^{\infty}y|f(y,s)|\,{\mathrm d}s\leq
Tq(\theta)|\theta|
\int_0^T
\sum_{y=1}^{\infty}y^2|f(y,s)|\,{\mathrm d}s\leq
T^2M_2q(\theta)|\theta|\leq
T^2M_2|\theta|^3.
\end{gather*}
Thus the difference is $O(|\theta|^3)$ as $\theta\to0$.

It follows that
\[
\sin(x\theta)S_f(\theta)
=
x\left\{
\int_0^T
\left[\sum_{y=1}^{\infty}yf(y,s)\right]\,{\mathrm d}s
\right\}\theta^2
+o(\theta^3).
\]
Applying \eqref{sine-laplace-estimate} proves the asserted
forcing asymptotic.

The initial-data asymptotic expansion \eqref{initial-expansion} also implies
\[
u_a(x,t+T)-u_a(x,t)=o(t^{-2}),
\]
because the mean value theorem yields
\[
(t+T)^{-3/2}-t^{-3/2}=O(t^{-5/2})
\]
and so both initial-data remainders are $o(t^{-2})$ by \eqref{initial-expansion}.
Adding the three contributions gives the final formula.
\end{proof}

\subsection{Robin-type problems}

We now replace the prescribed boundary values by a linear
relation between $u(0,t)$ and $u(1,t)$.
We first derive contour representations and then justify
the resulting solution formulas, including the contribution
of a forcing term.

Let $g\colon[0,\infty)\to\mathbb C$ be continuous, and let
$\{u_x\}_{x\in\mathbb N_0}$ satisfy
\[
\sum_{x\in\mathbb N_0}|u_x|<\infty.
\]
Suppose that $\kappa\in\mathbb C\setminus\{0\}$,
$|\kappa|\neq1$, and
\[
u_0+\kappa u_1=g(0).
\]
We seek a function $u\colon\mathbb N_0\times[0,\infty)\to\mathbb C$
satisfying
\begin{align}
u(x,0)&=u_x,
&&x\in\mathbb N_0,
\label{robin1}\\
\frac{\partial u}{\partial t}(x,t)
&=u(x+1,t)-2u(x,t)+u(x-1,t),
&&x\in\mathbb N,\quad t\geq0,
\label{robin2}\\
u(0,t)+\kappa u(1,t)&=g(t),
&&t\geq0.
\label{robin3}
\end{align}
The compatibility condition ensures that the initial and boundary
conditions agree at $t=0$.

\subsubsection{Derivation of the contour representation}

As in the preceding sections, we begin with a formal transform
calculation. Define
\[
\hat u(z,t):=\sum_{x=1}^{\infty}u(x,t)z^x,
\qquad |z|\leq1,
\]
and put
\[
g_0(t):=u(0,t),\qquad g_1(t):=u(1,t).
\]
Applying the discrete transform to \eqref{robin2} gives
\begin{align*}
\frac{\partial\hat u}{\partial t}(z,t)
&=
z^{-1}\bigl[\hat u(z,t)-zg_1(t)\bigr]
-2\hat u(z,t)+z\bigl[\hat u(z,t)+g_0(t)\bigr]\\
&=
(z+z^{-1}-2)\hat u(z,t)+zg_0(t)-g_1(t).
\end{align*}
For this subsection, write
\[
\omega(z,t):={\mathrm e}^{(z+z^{-1}-2)t}
\]
and define
\[
\tilde g_j(z,t):=
\int_0^t g_j(s)\omega(z,t-s)\,{\mathrm d}s,
\qquad j\in\{0,1\}.
\]
Solving the transformed equation yields the global relation
\[
\hat u(z,t)
=
\hat u(z,0)\omega(z,t)+z\tilde g_0(z,t)-\tilde g_1(z,t).
\]
Since $\omega(z^{-1},t)=\omega(z,t)$, we also have, on $|z|=1$,
\[
\hat u(z^{-1},t)
=
\hat u(z^{-1},0)\omega(z,t)
+z^{-1}\tilde g_0(z,t)-\tilde g_1(z,t).
\]

Subtracting these two identities gives
\begin{align*}
\hat u(z,t)-\hat u(z^{-1},t)
&=
[\hat u(z,0)-\hat u(z^{-1},0)]\omega(z,t)+(z-z^{-1})\tilde g_0(z,t).
\end{align*}
Alternatively, multiplying the first identity by $z^{-1}$ and the
second by $z$, and then subtracting, gives
\begin{align*}
z^{-1}\hat u(z,t)-z\hat u(z^{-1},t)
&=
[z^{-1}\hat u(z,0)-z\hat u(z^{-1},0)]\omega(z,t)+(z-z^{-1})\tilde g_1(z,t).
\end{align*}
Multiply the latter identity by $\kappa$ and add it to the former.
The boundary condition implies
\[
\tilde g_0(z,t)+\kappa\tilde g_1(z,t)=\tilde g(z,t),
\qquad
\tilde g(z,t):=\int_0^t g(s)\omega(z,t-s)\,{\mathrm d}s.
\]
Therefore,
\begin{align*}
\left(1+\frac{\kappa}{z}\right)\hat u(z,t)
-(1+\kappa z)\hat u(z^{-1},t)
&=
\left[
\left(1+\frac{\kappa}{z}\right)\hat u(z,0)
-(1+\kappa z)\hat u(z^{-1},0)
\right]\omega(z,t)+(z-z^{-1})\tilde g(z,t).
\end{align*}
Since $|\kappa|\neq1$, we may divide by $1+\kappa/z$ on the unit
circle. This gives
\[
\begin{aligned}
\hat u(z,t)-\frac{z+\kappa z^2}{z+\kappa}\hat u(z^{-1},t)
={}&
\left[
\hat u(z,0)-\frac{z+\kappa z^2}{z+\kappa}\hat u(z^{-1},0)
\right]\omega(z,t)+\frac{z^2-1}{z+\kappa}\tilde g(z,t).
\end{aligned}
\]

Multiplying by $z^{-x-1}$ and integrating counterclockwise around
the unit circle, we obtain, for $x\in\mathbb N$,
\begin{equation}\label{robin4}
\begin{aligned}
u(x,t)
&-
\frac{1}{2\pi{\mathrm i}}
\int_{|z|=1}
\frac{z+\kappa z^2}{z+\kappa}\hat u(z^{-1},t)
\frac{{\mathrm d}z}{z^{x+1}}\\
&=
\frac{1}{2\pi{\mathrm i}}
\int_{|z|=1}
\left[
\hat u(z,0)-\frac{z+\kappa z^2}{z+\kappa}\hat u(z^{-1},0)
\right]\omega(z,t)
\frac{{\mathrm d}z}{z^{x+1}}+
\frac{1}{2\pi{\mathrm i}}
\int_{|z|=1}
\frac{z^2-1}{z+\kappa}\tilde g(z,t)
\frac{{\mathrm d}z}{z^{x+1}}.
\end{aligned}
\end{equation}
To evaluate the remaining term involving the unknown transform, set
$w=z^{-1}$. The minus sign from the differential cancels with the
reversal of orientation, giving
\begin{equation}\label{robin5}
\begin{aligned}
&\int_{|z|=1}
\frac{z+\kappa z^2}{z+\kappa}\hat u(z^{-1},t)
\frac{{\mathrm d}z}{z^{x+1}}=
\int_{|w|=1}
\frac{w+\kappa}{1+\kappa w}\hat u(w,t)w^{x-2}\,{\mathrm d}w.
\end{aligned}
\end{equation}
Since the defining series for $\hat u(w,t)$ starts at index $1$,
we have
$\hat u(w,t)
=
w\sum_{y=1}^{\infty}u(y,t)w^{y-1}$.
Thus
\[
\frac{w+\kappa}{1+\kappa w}\hat u(w,t)w^{x-2}
=
\frac{w+\kappa}{1+\kappa w}
w^{x-1}\sum_{y=1}^{\infty}u(y,t)w^{y-1}
\]
extends analytically to $w=0$ for every $x\geq1$. The evaluation therefore
depends on whether the possible pole at $w=-1/\kappa$ lies inside
the unit circle.

\subsubsection{The case $|\kappa|<1$}

Here $|-1/\kappa|>1$, so the integral in \eqref{robin5} vanishes.
It follows from \eqref{robin4} that
\begin{equation*}
\begin{aligned}
u(x,t)
&=
\frac{1}{2\pi{\mathrm i}}
\int_{|z|=1}
\left[
\hat u(z,0)-\frac{z+\kappa z^2}{z+\kappa}\hat u(z^{-1},0)
\right]\omega(z,t)
\frac{{\mathrm d}z}{z^{x+1}}+
\frac{1}{2\pi{\mathrm i}}
\int_{|z|=1}
\frac{z^2-1}{z+\kappa}\tilde g(z,t)
\frac{{\mathrm d}z}{z^{x+1}},
\qquad x\geq1.
\end{aligned}
\end{equation*}
The boundary value is then defined by
\[
u(0,t)=g(t)-\kappa u(1,t).
\]

We next express the contour integrals in terms of the modified Bessel
functions. 
For the first term, we use
$\hat u(z,0)=\sum_{y=1}^{\infty}u_yz^y$ and the Laurent
expansion $\omega(z,t)
=
{\mathrm e}^{-2t}
\sum_{\ell\in\mathbb Z}I_\ell(2t)z^\ell$ to obtain
\begin{align*}
\frac{1}{2\pi{\mathrm i}}
\int_{|z|=1}\hat u(z,0)\omega(z,t)
\frac{{\mathrm d}z}{z^{x+1}}
&=
{\mathrm e}^{-2t}
\sum_{y=1}^{\infty}u_y
\sum_{\ell\in\mathbb Z}I_\ell(2t)
\frac{1}{2\pi{\mathrm i}}
\int_{|z|=1}z^{y+\ell-x-1}\,{\mathrm d}z\\
&=
{\mathrm e}^{-2t}
\sum_{y=1}^{\infty}u_yI_{x-y}(2t)
=
{\mathrm e}^{-2t}
\sum_{y=1}^{\infty}u_yI_{y-x}(2t).
\end{align*}
On $|z|=1$, the geometric series gives
\[
\frac{z+\kappa z^2}{z+\kappa}
=
(1+\kappa z)\sum_{m=1}^{\infty}(-\kappa)^{m-1}z^{1-m}
=
\sum_{m=1}^{\infty}(-\kappa)^{m-1}
\bigl(z^{1-m}+\kappa z^{2-m}\bigr).
\]
Since $\hat u(z^{-1},0)=\sum_{y=1}^{\infty}u_yz^{-y}$,
multiplying these two expansions yields
\[
\frac{z+\kappa z^2}{z+\kappa}\hat u(z^{-1},0)
=
\sum_{y=1}^{\infty}\sum_{m=1}^{\infty}
u_y(-\kappa)^{m-1}
\bigl(z^{-y+1-m}+\kappa z^{-y+2-m}\bigr).
\]
Using also the expansion of $\omega(z,t)$,
we obtain
\begin{align*}
&\frac{1}{2\pi{\mathrm i}}
\int_{|z|=1}
\frac{z+\kappa z^2}{z+\kappa}
\hat u(z^{-1},0)\omega(z,t)
\frac{{\mathrm d}z}{z^{x+1}}\\
&\quad=
{\mathrm e}^{-2t}
\sum_{y=1}^{\infty}\sum_{m=1}^{\infty}
u_y(-\kappa)^{m-1}
\sum_{\ell\in\mathbb Z}I_\ell(2t)
\frac{1}{2\pi{\mathrm i}}
\int_{|z|=1}
\bigl[
z^{\ell-y-m-x}
+\kappa z^{\ell-y-m-x+1}
\bigr]\,{\mathrm d}z.
\end{align*}
For each fixed pair $y,m$, the first term contributes only
when $\ell=y+x+m-1$, and the second only when
$\ell=y+x+m-2$. Thus the sum over~$\ell$ reduces to two terms,
and we conclude that
\begin{equation*}
\frac{1}{2\pi{\mathrm i}}
\int_{|z|=1}
\frac{z+\kappa z^2}{z+\kappa}
\hat u(z^{-1},0)\omega(z,t)
\frac{{\mathrm d}z}{z^{x+1}}=
{\mathrm e}^{-2t}
\sum_{y=1}^{\infty}u_y
\sum_{m=1}^{\infty}(-\kappa)^{m-1}
\bigl[
I_{y+x+m-1}(2t)+\kappa I_{y+x+m-2}(2t)
\bigr].
\end{equation*}
Similarly,
\[
\frac{z^2-1}{z+\kappa}
=
\sum_{m=1}^{\infty}(-\kappa)^{m-1}
\bigl(z^{2-m}-z^{-m}\bigr),
\]
so
\begin{align*}
&\frac{1}{2\pi{\mathrm i}}
\int_{|z|=1}
\frac{z^2-1}{z+\kappa}\tilde g(z,t)
\frac{{\mathrm d}z}{z^{x+1}}=
\sum_{m=1}^{\infty}(-\kappa)^{m-1}
\int_0^t
{\mathrm e}^{-2(t-s)}g(s)
\bigl[
I_{x+m-2}\bigl(2(t-s)\bigr)-I_{x+m}\bigl(2(t-s)\bigr)
\bigr]
\,{\mathrm d}s.
\end{align*}
In the previous calculation, the geometric series and the Laurent series converge absolutely
and uniformly on the unit circle, locally uniformly in time.
The summability of the initial data and continuity of $g$ therefore
justify the interchanges of sums and integrals.

Consequently,
\begin{equation}\label{robin10}
\begin{aligned}
u(x,t)
&=
{\mathrm e}^{-2t}\sum_{y=1}^{\infty}u_y
\Bigg\{
I_{y-x}(2t)
-
\sum_{m=1}^{\infty}(-\kappa)^{m-1}
\bigl[
I_{y+x+m-1}(2t)+\kappa I_{y+x+m-2}(2t)
\bigr]
\Bigg\}\\
&\quad+
\sum_{m=1}^{\infty}(-\kappa)^{m-1}
\int_0^t
{\mathrm e}^{-2(t-s)}g(s)
\bigl[
I_{x+m-2}\bigl(2(t-s)\bigr)-I_{x+m}\bigl(2(t-s)\bigr)
\bigr]
\,{\mathrm d}s,
\qquad x\geq1.
\end{aligned}
\end{equation}

This representation also permits bounded initial data:
\[
\sup_{x\in\mathbb N_0}|u_x|<\infty.
\]
To see the relevant bounds, recall that
\[
K_\ell(t):={\mathrm e}^{-2t}I_\ell(2t)\geq0,
\qquad
\sum_{\ell\in\mathbb Z}K_\ell(t)=1.
\]
Each sum over a subset of kernel indices is therefore bounded by $1$,
while
\[
\sum_{m=1}^{\infty}|\kappa|^{m-1}
=\frac{1}{1-|\kappa|}.
\]
Put $C:=\sup_{y\geq1}|u_y|$ and $G_T:=\max_{0\leq s\leq T}|g(s)|$.
By the triangle inequality, for $x\geq1$ and $0\leq t\leq T$,
\begin{align*}
|u(x,t)|
&\leq
C\sum_{y=1}^{\infty}K_{x-y}(t)+
C\sum_{m=1}^{\infty}|\kappa|^{m-1}
\sum_{y=1}^{\infty}
\bigl[K_{x+y+m-1}(t)
+|\kappa|K_{x+y+m-2}(t)\bigr]\\
&\quad+
G_T\sum_{m=1}^{\infty}|\kappa|^{m-1}
\int_0^t
\bigl[K_{x+m-2}(t-s)+K_{x+m}(t-s)\bigr]
\,{\mathrm d}s\\
&\leq
C+C(1+|\kappa|)
\sum_{m=1}^{\infty}|\kappa|^{m-1}
+2TG_T\sum_{m=1}^{\infty}|\kappa|^{m-1}=
\left(1+\frac{1+|\kappa|}{1-|\kappa|}\right)C
+\frac{2TG_T}{1-|\kappa|}.
\end{align*}
The boundary relation gives the corresponding bound at $x=0$.

Absolute convergence, locally uniform in time for each fixed $x$,
follows from the factorial bound on $I_\ell(2t)$ established earlier.
The verification that these formulas give solutions, including
the extension to bounded initial data, is provided at the end
of this section.

Likewise, summing the same kernel estimates over $x\geq1$ shows that
\[
\sup_{0\leq t\leq T}\sum_{x=1}^{\infty}|u(x,t)|
\leq
\left(1+\frac{1+|\kappa|}{1-|\kappa|}\right)
\sum_{y=1}^{\infty}|u_y|
+\frac{2TG_T}{1-|\kappa|}.
\]
Thus summable initial data give a solution that remains summable,
uniformly on compact time intervals, including the boundary value
at $x=0$.

\subsubsection{The case $|\kappa|>1$}

Here the point $w=-1/\kappa$ lies inside the unit circle, so
\eqref{robin5} has an additional residue contribution.
We first determine $\hat u(-1/\kappa,t)$.

Evaluating the global relation at $z=-1/\kappa$, we obtain
\begin{align*}
\hat u(-1/\kappa,t)
&=
\hat u(-1/\kappa,0)\omega(-1/\kappa,t)
-\frac1\kappa\tilde g_0(-1/\kappa,t)
-\tilde g_1(-1/\kappa,t)\\
&=
\hat u(-1/\kappa,0)\omega(-1/\kappa,t)
-\frac1\kappa\tilde g(-1/\kappa,t),
\end{align*}
where the second equality uses
$\tilde g_0+\kappa\tilde g_1=\tilde g$.
By the definitions of the transforms,
\[
\hat u(-1/\kappa,0)
=
\sum_{y=1}^{\infty}\frac{(-1)^y}{\kappa^y}u_y.
\]
Moreover, since
\[
-\frac1\kappa+(-1/\kappa)^{-1}-2
=
-(\kappa+\kappa^{-1}+2),
\]
we have
\[
\omega(-1/\kappa,t)
=
{\mathrm e}^{-(\kappa+\kappa^{-1}+2)t}
\]
and
\[
\tilde g(-1/\kappa,t)
=
\int_0^t
g(s){\mathrm e}^{-(\kappa+\kappa^{-1}+2)(t-s)}
\,{\mathrm d}s.
\]
Using
$1+\kappa w=\kappa(w+1/\kappa)$, we calculate
\begin{align*}
\operatorname*{Res}_{w=-1/\kappa}
\left[
\frac{w+\kappa}{1+\kappa w}
\hat u(w,t)w^{x-2}
\right]=
\lim_{w\to-1/\kappa}
\frac{w+\kappa}{\kappa}
\hat u(w,t)w^{x-2}=
\frac{\kappa-\kappa^{-1}}{\kappa}
\hat u(-1/\kappa,t)
\left(-\frac1\kappa\right)^{x-2}.
\end{align*}
The residue theorem therefore gives
\begin{equation*}
\int_{|w|=1}
\frac{w+\kappa}{1+\kappa w}
\hat u(w,t)w^{x-2}\,{\mathrm d}w=
2\pi{\mathrm i}\,
\frac{\kappa-\kappa^{-1}}{\kappa}
\hat u(-1/\kappa,t)
\left(-\frac1\kappa\right)^{x-2}=
2\pi{\mathrm i}\,
\frac{(-1)^x}{\kappa^x}C(t),
\end{equation*}
where
\begin{align*}
C(t)
&:=
(\kappa^2-1)\hat u(-1/\kappa,t)=
(\kappa^2-1)
\Bigg\{
\left[
\sum_{y=1}^{\infty}\frac{(-1)^y}{\kappa^y}u_y
\right]
{\mathrm e}^{-(\kappa+\kappa^{-1}+2)t}
-\frac1\kappa
\int_0^t
g(s){\mathrm e}^{-(\kappa+\kappa^{-1}+2)(t-s)}
\,{\mathrm d}s
\Bigg\}.
\end{align*}
Consequently,
\begin{equation}\label{robin-contour-large}
\begin{aligned}
u(x,t)
&=
\frac{(-1)^x}{\kappa^x}C(t)+
\frac{1}{2\pi{\mathrm i}}
\int_{|z|=1}
\left[
\hat u(z,0)-\frac{z+\kappa z^2}{z+\kappa}\hat u(z^{-1},0)
\right]\omega(z,t)
\frac{{\mathrm d}z}{z^{x+1}}\\
&\quad+
\frac{1}{2\pi{\mathrm i}}
\int_{|z|=1}
\frac{z^2-1}{z+\kappa}\tilde g(z,t)
\frac{{\mathrm d}z}{z^{x+1}},
\qquad x\geq1.
\end{aligned}
\end{equation}
Again, $u(0,t)=g(t)-\kappa u(1,t)$.

For $|\kappa|>1$, the appropriate geometric expansions on $|z|=1$
are
\begin{align*}
\frac{z+\kappa z^2}{z+\kappa}
&=
\sum_{m=1}^{\infty}
\frac{(-1)^{m-1}}{\kappa^m}
\bigl(z^m+\kappa z^{m+1}\bigr),\\
\frac{z^2-1}{z+\kappa}
&=
\sum_{m=1}^{\infty}
\frac{(-1)^{m-1}}{\kappa^m}
\bigl(z^{m+1}-z^{m-1}\bigr).
\end{align*}
It follows that
\begin{gather*}
\frac{1}{2\pi{\mathrm i}}
\int_{|z|=1}
\frac{z+\kappa z^2}{z+\kappa}
\hat u(z^{-1},0)\omega(z,t)
\frac{{\mathrm d}z}{z^{x+1}}=
{\mathrm e}^{-2t}
\sum_{y=1}^{\infty}u_y
\sum_{m=1}^{\infty}
\frac{(-1)^{m-1}}{\kappa^m}
\bigl[
I_{y+x-m}(2t)+\kappa I_{y+x-m-1}(2t)
\bigr],\\
\frac{1}{2\pi{\mathrm i}}
\int_{|z|=1}
\frac{z^2-1}{z+\kappa}\tilde g(z,t)
\frac{{\mathrm d}z}{z^{x+1}}=
\sum_{m=1}^{\infty}\frac{(-1)^m}{\kappa^m}
\int_0^t
{\mathrm e}^{-2(t-s)}g(s)
\bigl[
I_{m-x-1}\bigl(2(t-s)\bigr)-I_{m-x+1}\bigl(2(t-s)\bigr)
\bigr]
\,{\mathrm d}s.
\end{gather*}
Thus
\begin{align*}
u(x,t)
&=
\frac{(-1)^x}{\kappa^x}C(t)+
{\mathrm e}^{-2t}\sum_{y=1}^{\infty}u_y
\Bigg\{
I_{y-x}(2t)
+
\sum_{m=1}^{\infty}\frac{(-1)^m}{\kappa^m}
\bigl[
I_{y+x-m}(2t)+\kappa I_{y+x-m-1}(2t)
\bigr]
\Bigg\}\\
&\quad+
\sum_{m=1}^{\infty}\frac{(-1)^m}{\kappa^m}
\int_0^t
{\mathrm e}^{-2(t-s)}g(s)
\bigl[
I_{m-x-1}\bigl(2(t-s)\bigr)-I_{m-x+1}\bigl(2(t-s)\bigr)
\bigr]
\,{\mathrm d}s,
\qquad x\geq1.
\end{align*}

\subsubsection{The nonhomogeneous problem}

We now include a forcing term $f(x,t)$, retaining the assumptions
on $g$, $\kappa$, and the initial data.
We first proceed formally, assuming that the transforms and
the required interchanges of sums, derivatives, and integrals
are justified. Precise assumptions and a justification of the
resulting formulas are given below.

Let $u^{(0)}$ denote the solution of the unforced Robin
problem considered above, with the same initial and
boundary data. We now seek the solution
$u\colon\mathbb N_0\times[0,\infty)\to\mathbb C$ satisfying
\begin{align}
u(x,0)&=u_x,
&&x\in\mathbb N_0,
\label{robin24}\\
\frac{\partial u}{\partial t}(x,t)
&=
u(x+1,t)-2u(x,t)+u(x-1,t)+f(x,t),
&&x\in\mathbb N,\quad t\geq0,
\label{robin25}\\
u(0,t)+\kappa u(1,t)&=g(t),
&&t\geq0.
\label{robin26}
\end{align}

Define
\[
\hat u(z,t):=\sum_{x=1}^{\infty}u(x,t)z^x,
\qquad
\hat f(z,t):=\sum_{x=1}^{\infty}f(x,t)z^x.
\]
Here $g_0(t):=u(0,t)$ and $g_1(t):=u(1,t)$.
The transformed equation is
\[
\frac{\partial\hat u}{\partial t}(z,t)
=
(z+z^{-1}-2)\hat u(z,t)
+zg_0(t)-g_1(t)+\hat f(z,t),
\]
so
\[
\hat u(z,t)
=
\hat u(z,0)\omega(z,t)
+z\tilde g_0(z,t)-\tilde g_1(z,t)
+\int_0^t\hat f(z,s)\omega(z,t-s)\,{\mathrm d}s.
\]
The same elimination of the boundary transforms gives
\begin{align*}
\hat u(z,t)
-\frac{z+\kappa z^2}{z+\kappa}\hat u(z^{-1},t)
&=
\left[
\hat u(z,0)
-\frac{z+\kappa z^2}{z+\kappa}\hat u(z^{-1},0)
\right]\omega(z,t)+
\frac{z^2-1}{z+\kappa}\tilde g(z,t)\\
&\quad+
\int_0^t
\left[
\hat f(z,s)
-\frac{z+\kappa z^2}{z+\kappa}\hat f(z^{-1},s)
\right]\omega(z,t-s)\,{\mathrm d}s.
\end{align*}

If $|\kappa|<1$, the reflected transform
again has zero contribution to the inverse transform, and hence
\begin{align*}
u(x,t)
&=
u^{(0)}(x,t)
+
\frac{1}{2\pi{\mathrm i}}
\int_{|z|=1}
\Bigg\{
\int_0^t
\left[
\sum_{y=1}^{\infty}
f(y,s)
\left(
z^y-\frac{z+\kappa z^2}{z+\kappa}z^{-y}
\right)
\right]
\omega(z,t-s)\,{\mathrm d}s
\Bigg\}
\frac{{\mathrm d}z}{z^{x+1}},
\qquad x\geq1.
\end{align*}
The boundary value is recovered from
\[
u(0,t)=g(t)-\kappa u(1,t).
\]

If $|\kappa|>1$, the reflected-transform integral again
has a residue at $w=-1/\kappa$ after the substitution
$w=z^{-1}$.
To evaluate its contribution, we evaluate the global
relation at $z=-1/\kappa$:
\begin{align*}
\hat u(-1/\kappa,t)
&=
\hat u(-1/\kappa,0)\omega(-1/\kappa,t)
-\frac1\kappa\tilde g(-1/\kappa,t)+
\int_0^t\hat f(-1/\kappa,s)\omega(-1/\kappa,t-s)
\,{\mathrm d}s.
\end{align*}
The corresponding formula for $u^{(0)}$ has the same
initial-data and boundary-data terms, but no forcing term.
Subtracting the two formulas and using
$\hat f(-1/\kappa,s)
=\sum_{y=1}^{\infty}(-1)^y\kappa^{-y}f(y,s)$ gives
\[
\hat u(-1/\kappa,t)-\widehat{u^{(0)}}(-1/\kappa,t)
=
\int_0^t
\left[
\sum_{y=1}^{\infty}\frac{(-1)^y}{\kappa^y}f(y,s)
\right]
{\mathrm e}^{-(\kappa+\kappa^{-1}+2)(t-s)}
\,{\mathrm d}s.
\]
Consequently,
\begin{align*}
u(x,t)
&=
u^{(0)}(x,t)+
\frac{(-1)^x}{\kappa^x}(\kappa^2-1)
\int_0^t
\left[
\sum_{y=1}^{\infty}\frac{(-1)^y}{\kappa^y}f(y,s)
\right]
{\mathrm e}^{-(\kappa+\kappa^{-1}+2)(t-s)}
\,{\mathrm d}s\\
&\quad+
\frac{1}{2\pi{\mathrm i}}
\int_{|z|=1}
\Bigg\{
\int_0^t
\left[
\sum_{y=1}^{\infty}
f(y,s)
\left(
z^y-\frac{z+\kappa z^2}{z+\kappa}z^{-y}
\right)
\right]
\omega(z,t-s)\,{\mathrm d}s
\Bigg\}
\frac{{\mathrm d}z}{z^{x+1}},
\qquad x\geq1.
\end{align*}
As before,
\[
u(0,t)=g(t)-\kappa u(1,t).
\]

We now justify the preceding representations and state assumptions
under which their Bessel forms give classical solutions.

For $x,y\in\mathbb N$, let $R^\kappa_{x,y}(t)$ denote the coefficient
of $u_y$ in the homogeneous Bessel representation with $g=0$.
Thus, for $|\kappa|<1$,
\[
R^\kappa_{x,y}(t)
=
K_{x-y}(t)
-
\sum_{m=1}^{\infty}(-\kappa)^{m-1}
\bigl[
K_{x+y+m-1}(t)+\kappa K_{x+y+m-2}(t)
\bigr],
\]
whereas, for $|\kappa|>1$,
\begin{align*}
R^\kappa_{x,y}(t)
&=
(\kappa^2-1)\frac{(-1)^{x+y}}{\kappa^{x+y}}
{\mathrm e}^{-(\kappa+\kappa^{-1}+2)t}
+K_{x-y}(t)+
\sum_{m=1}^{\infty}\frac{(-1)^m}{\kappa^m}
\bigl[
K_{x+y-m}(t)+\kappa K_{x+y-m-1}(t)
\bigr].
\end{align*}
Since $R^\kappa_{x,y}$ is the initial-data kernel for the
homogeneous Robin problem, Duhamel's principle suggests
\begin{equation}\label{robin-forcing-bessel}
u(x,t)
=
u^{(0)}(x,t)
+
\int_0^t
\sum_{y=1}^{\infty}
R^\kappa_{x,y}(t-s)f(y,s)\,{\mathrm d}s,
\qquad x\in\mathbb N,\quad t\geq0.
\end{equation}
The same formula can also be obtained by expanding the forcing contour integrals in the same way as the
initial-data integrals.

We now justify this representation under the assumptions
of the following theorem.

\begin{theorem}\label{theo-robin}
Let $\kappa\in\mathbb C\setminus\{0\}$ satisfy $|\kappa|\neq1$.
Let $(u_x)_{x\in\mathbb N_0}$ be a bounded sequence of complex
numbers, and let $g\colon[0,\infty)\to\mathbb C$ be continuous,
with
\[
u_0+\kappa u_1=g(0).
\]
Suppose that $f\colon\mathbb N\times[0,\infty)\to\mathbb C$
is continuous in $t$ for each $x$ and satisfies
\[
\sup_{\substack{x\in\mathbb N\\0\leq t\leq R}}
|f(x,t)|<\infty
\qquad\text{for every }R>0.
\]

Let $u^{(0)}$ be given by the homogeneous Bessel representation
derived above for the corresponding case $|\kappa|<1$ or
$|\kappa|>1$, with
$u^{(0)}(0,t)=g(t)-\kappa u^{(0)}(1,t)$.
Then $u^{(0)}$ is a solution of
\eqref{robin1}--\eqref{robin3}.

Moreover, formula \eqref{robin-forcing-bessel}, together with
\[
u(0,t)=g(t)-\kappa u(1,t),
\]
defines a solution of
\eqref{robin24}--\eqref{robin26}.
\end{theorem}

\begin{proof}
We first verify the homogeneous representations under the
additional assumption
\[
\sum_{x=1}^{\infty}|u_x|<\infty.
\]
Let $V$ be the Dirichlet solution with these initial data,
boundary value $g$, and zero forcing, as supplied by
Theorem~\ref{theo-1}.
For a continuous function $h$, the same theorem gives the
Dirichlet solution with boundary value $g-\kappa h$:
\begin{align*}
w(x,t)
&=
\sum_{y=1}^{\infty}u_yP_{x,y}(t)
+
\int_0^t[g(s)-\kappa h(s)]P_{x,1}(t-s)\,{\mathrm d}s=
V(x,t)-\kappa\int_0^t h(s)P_{x,1}(t-s)\,{\mathrm d}s,
\qquad x\in\mathbb N,
\end{align*}
together with
\[
w(0,t)=g(t)-\kappa h(t).
\]
For every continuous $h$, this function satisfies the
homogeneous diffusion equation and the prescribed initial
data at every $x\in\mathbb N$.

We choose $h$ so that the Robin condition holds.
Since
\[
w(0,t)+\kappa w(1,t)-g(t)
=
\kappa[w(1,t)-h(t)],
\]
this condition is equivalent to $w(1,t)=h(t)$, or
\begin{equation}\label{robin-volterra}
h(t)+\kappa\int_0^t P_{1,1}(t-s)h(s)\,{\mathrm d}s
=
V(1,t).
\end{equation}
The kernel and the right-hand side are continuous.
The standard existence and uniqueness theorem for linear
Volterra integral equations therefore gives a unique
continuous solution on every bounded time interval.
By uniqueness, these solutions agree on overlapping
intervals and define $h$ on $[0,\infty)$.

With this choice of $h$, the function $w$ satisfies the Robin
condition. Moreover, \eqref{robin-volterra} gives $h(0)=u_1$,
so compatibility yields
\[
w(0,0)=g(0)-\kappa u_1=u_0.
\]

We next justify the transform calculations for this solution.
Symmetry and the mass bound
\eqref{dirichlet-kernel-mass}, with $a=1$, give
\[
\sum_{x=1}^{\infty}P_{x,y}(t)\leq1.
\]
Consequently,
\[
\sum_{x=1}^{\infty}|w(x,t)|
\leq
\sum_{y=1}^{\infty}|u_y|
+
\int_0^t|g(s)-\kappa h(s)|\,{\mathrm d}s,
\]
which is bounded on every compact time interval.

The function $w(\cdot,t)$ also depends continuously on $t$
in the $\ell^1$ norm. Indeed, for each fixed $y$, the
factorial bound \eqref{K-factorial-bound} gives continuity
of $(P_{x,y}(t))_{x\in\mathbb N}$ in that norm.
The mass bound then permits approximation of the initial-data
series by finite sums. Continuity of the boundary integral
follows from the same kernel continuity and continuity of
$g-\kappa h$.
The differential equation shows that
$\partial_t w(\cdot,t)$ is continuous in the $\ell^1$ norm
as well: the spatial shifts are bounded operations on
$\ell^1$, and the boundary value contributes only at $x=1$.

Thus the transform may be differentiated termwise, and the
preceding transform calculations apply to~$w$.
They yield the homogeneous contour formulas and, by
termwise integration of absolutely convergent series,
their Bessel representations.

In particular, apply this conclusion to initial data
$u_x=\delta_{x,y}$ for $x\in\mathbb N$ and boundary data
$g=0$, with the compatible value
$u_0=-\kappa\delta_{1,y}$.
The resulting solution is $R^\kappa_{x,y}(t)$ for
$x\in\mathbb N$. Defining
\[
R^\kappa_{0,y}(t):=-\kappa R^\kappa_{1,y}(t),
\]
we therefore have
\[
\frac{\partial R^\kappa_{x,y}}{\partial t}(t)
=
R^\kappa_{x+1,y}(t)-2R^\kappa_{x,y}(t)
+R^\kappa_{x-1,y}(t),
\qquad x,y\in\mathbb N,
\]
and
\[
R^\kappa_{x,y}(0)=\delta_{x,y},
\qquad x,y\in\mathbb N.
\]

We now establish the convergence needed for bounded initial
data and forcing. For every fixed $x\in\mathbb N$ and $R>0$,
the explicit Robin kernel formulas give
\begin{equation}\label{robin-kernel-summability}
\sum_{y=1}^{\infty}
\sup_{0\leq t\leq R}|R^\kappa_{x,y}(t)|<\infty.
\end{equation}
Indeed, \eqref{K-factorial-bound} gives
\[
S_R:=\sum_{\ell\in\mathbb Z}
\sup_{0\leq t\leq R}K_\ell(t)
\leq
{\mathrm e}^{R^2}
\left(1+2\sum_{\ell=1}^{\infty}\frac{R^\ell}{\ell!}\right)
<\infty.
\]
For $|\kappa|<1$, the explicit kernel formula is
\[
R^\kappa_{x,y}(t)
=
K_{x-y}(t)
-
\sum_{m=1}^{\infty}(-\kappa)^{m-1}
\bigl[
K_{x+y+m-1}(t)+\kappa K_{x+y+m-2}(t)
\bigr].
\]
Consequently,
\begin{align*}
\sum_{y=1}^{\infty}
\sup_{0\leq t\leq R}|R^\kappa_{x,y}(t)|
&\leq
S_R+(1+|\kappa|)S_R
\sum_{m=1}^{\infty}|\kappa|^{m-1}=
S_R\left(1+\frac{1+|\kappa|}{1-|\kappa|}\right)
<\infty.
\end{align*}
For $|\kappa|>1$, the expansion
\[
\frac1{z+\kappa}
=
\sum_{m=1}^{\infty}(-1)^{m-1}\kappa^{-m}z^{m-1}
\]
gives the kernel formula
\begin{align*}
R^\kappa_{x,y}(t)
&=
K_{x-y}(t)
-
\sum_{m=1}^{\infty}(-1)^{m-1}\kappa^{-m}
\bigl[
K_{x+y-m}(t)+\kappa K_{x+y-m-1}(t)
\bigr]+
(\kappa^2-1)\frac{(-1)^{x+y}}{\kappa^{x+y}}
{\mathrm e}^{-(\kappa+\kappa^{-1}+2)t}.
\end{align*}
The first two terms are estimated in the same way, using
\[
\sum_{m=1}^{\infty}|\kappa|^{-m}
=
\frac1{|\kappa|-1}.
\]
For the last term, fix $x$ and put
\[
C_{x,R}:=
|\kappa^2-1|\,|\kappa|^{-x}
\max_{0\leq t\leq R}
\left|{\mathrm e}^{-(\kappa+\kappa^{-1}+2)t}\right|.
\]
Its absolute value is bounded by $C_{x,R}|\kappa|^{-y}$.
Thus
\[
\sum_{y=1}^{\infty}
\sup_{0\leq t\leq R}|R^\kappa_{x,y}(t)|
\leq
S_R\left(1+\frac{1+|\kappa|}{|\kappa|-1}\right)
+\frac{C_{x,R}}{|\kappa|-1}
<\infty.
\]
Finally, since
$R^\kappa_{0,y}(t)=-\kappa R^\kappa_{1,y}(t)$,
the same summability conclusion holds at $x=0$.

The kernel equation consequently implies
\[
\sum_{y=1}^{\infty}
\sup_{0\leq t\leq R}
\left|
\frac{\partial R^\kappa_{x,y}}{\partial t}(t)
\right|<\infty,
\qquad x\in\mathbb N,
\]
by applying \eqref{robin-kernel-summability} at
$x-1$, $x$, and $x+1$.

We now remove the additional summability assumption on the
initial data and assume only that $(u_y)$ is bounded.
Define the initial-data contribution by
\[
u_a(x,t):=
\sum_{y=1}^{\infty}u_yR^\kappa_{x,y}(t),
\qquad x\in\mathbb N_0.
\]
By \eqref{robin-kernel-summability} and boundedness of $(u_y)$,
this series converges absolutely and uniformly for
$t\in[0,R]$, for every fixed $x$ and $R>0$.
The kernel bounds justify termwise differentiation for
$x\in\mathbb N$. Hence $u_a$ satisfies the homogeneous
diffusion equation, with
\[
u_a(x,0)=u_x,\qquad x\in\mathbb N,
\qquad
u_a(0,t)+\kappa u_a(1,t)=0.
\]
Next, let $u_g$ denote the homogeneous Robin solution with
zero initial data at every $x\in\mathbb N$, boundary data~$g$,
and the compatible initial value $u_g(0,0)=g(0)$.
Its existence and Bessel representation were established
above, since these initial data are summable.
Its boundary value is
\[
u_g(0,t)=g(t)-\kappa u_g(1,t).
\]
Thus $u_a+u_g$ is precisely the homogeneous Bessel
representation $u^{(0)}$, and compatibility gives
$u^{(0)}(0,0)=u_0$.

Finally, consider the Duhamel contribution
\[
u_f(x,t):=
\int_0^t\sum_{y=1}^{\infty}
R^\kappa_{x,y}(t-s)f(y,s)\,{\mathrm d}s,
\qquad x\in\mathbb N_0.
\]
For every $R>0$ and for each fixed $x$,  the bound
\[
|R^\kappa_{x,y}(t-s)f(y,s)|
\leq
\left(\sup_{0\leq\tau\leq R}|R^\kappa_{x,y}(\tau)|\right)
\left(\sup_{\substack{j\in\mathbb N\\0\leq r\leq R}}
|f(j,r)|\right)
\]
and \eqref{robin-kernel-summability} give absolute and
uniform convergence on $0\leq s\leq t\leq R$.
Each summand is continuous, so the sum is continuous and
the time integral is an ordinary Riemann integral.
The corresponding derivative bound gives uniform convergence
of the differentiated series for $x\in\mathbb N$.
The Leibniz rule and the kernel equation therefore yield
\begin{align*}
\frac{\partial u_f}{\partial t}(x,t)
&=
\sum_{y=1}^{\infty}R^\kappa_{x,y}(0)f(y,t)
+
\int_0^t\sum_{y=1}^{\infty}
\frac{\partial R^\kappa_{x,y}}{\partial t}(t-s)
f(y,s)\,{\mathrm d}s\\
&=
f(x,t)+u_f(x+1,t)-2u_f(x,t)+u_f(x-1,t).
\end{align*}
Moreover,
\[
u_f(x,0)=0,\qquad x\in\mathbb N_0,
\qquad
u_f(0,t)+\kappa u_f(1,t)=0.
\]
Adding $u_f$ to $u^{(0)}$ proves
\eqref{robin-forcing-bessel} and all the required initial
and boundary conditions. 
\end{proof}

\section{Non-symmetric right-hand sides}\label{sec-7}

We next consider diffusion-type equations with nonsymmetric coefficients.
We first use symmetrization to obtain solution formulas,
then give a unified transform derivation that also extends
to a class of time-dependent coefficients.
Finally, we establish asymptotic periodicity under suitable
conditions on the coefficients and data.

\subsection{Symmetrization}

The next goal is to solve the initial-boundary value problem for
the more general equation
\begin{equation}\label{nonsymmetric}
\frac{\partial u}{\partial t}(x,t)
=
\alpha u(x-1,t)+\beta u(x,t)+\gamma u(x+1,t)+f(x,t).
\end{equation}

The following lemma makes it possible to convert equations with
non-symmetric coefficients $\alpha$, $\beta$, $\gamma$ into symmetric
ones.

\begin{lemma}\label{symmetrization}
Let $\alpha,\gamma\in\mathbb R\setminus\{0\}$ and
$\beta\in\mathbb C$.
Choose $r\in\mathbb C\setminus\{0\}$ such that
\[
r^2=\frac{\gamma}{\alpha},
\qquad
A:=\alpha r=\frac{\gamma}{r}.
\]
A function $u\colon\mathbb N_0\times[0,\infty)\to\mathbb C$
satisfies
\begin{equation}\label{u-eqn}
\frac{\partial u}{\partial t}(x,t)
=
\alpha u(x-1,t)+\beta u(x,t)+\gamma u(x+1,t)+f(x,t),
\qquad x\in\mathbb N,\quad t\geq0,
\end{equation}
if and only if the function
\[
v(x,t):=r^x u(x,t)
\]
satisfies
\begin{equation}\label{v-eqn}
\frac{\partial v}{\partial t}(x,t)
=
Av(x-1,t)+\beta v(x,t)+Av(x+1,t)+r^x f(x,t),
\qquad x\in\mathbb N,\quad t\geq0.
\end{equation}
If $\alpha\gamma>0$, we may choose
$r=\sqrt{\gamma/\alpha}>0$.
Then $A=\sqrt{\alpha\gamma}$ when $\alpha,\gamma>0$,
and $A=-\sqrt{\alpha\gamma}$ when $\alpha,\gamma<0$.
\end{lemma}

\begin{proof}
Multiplying \eqref{u-eqn} by $r^x$ gives
\[
\frac{\partial v}{\partial t}(x,t)
=
\alpha r\,v(x-1,t)+\beta v(x,t)
+\frac{\gamma}{r}v(x+1,t)+r^x f(x,t),
\]
which is \eqref{v-eqn}.
The converse follows by substituting $u(x,t)=r^{-x}v(x,t)$.
\end{proof}

The same result holds if $\mathbb N_0$ and $\mathbb N$
are replaced by $\mathbb Z$.

By the previous lemma, solving an initial-boundary value problem
for \eqref{u-eqn} can be reduced to solving a corresponding
initial-boundary value problem for \eqref{v-eqn}.
However, the transformation does not necessarily preserve
boundedness of the initial data.
Also, the calculations from subsections \ref{sec51} and \ref{sec52}
apply only to the special case $\beta=-2A$.
Although it would be possible to extend them to general $A$, $\beta$,
it seems easier to guess the result by inserting suitable scaling
factors.
Thus, we will not use the previous lemma directly in the present section, but it helps
identify the appropriate coefficients.

We now allow arbitrary real coefficients $\alpha,\beta,\gamma$,
subject to $\alpha,\gamma\neq0$.
To express the fundamental solution in terms of modified Bessel
functions, choose $\eta\in\mathbb C$ such that
\[
\eta^2=\alpha\gamma.
\]
The fundamental solution of the homogeneous equation
corresponding to \eqref{nonsymmetric} is
\[
G(x,t)
=
{\mathrm e}^{\beta t}
\left(\frac{\alpha}{\eta}\right)^x I_x(2\eta t),
\qquad x\in\mathbb Z,\quad t\geq0.
\]
This expression is independent of the choice between $\eta$
and $-\eta$, since
\[
I_j(-z)=(-1)^jI_j(z),
\qquad j\in\mathbb Z.
\]
In particular, when $\alpha,\gamma>0$, we may take
$\eta=\sqrt{\alpha\gamma}$, and then
$\alpha/\eta=\sqrt{\alpha/\gamma}$.
When both coefficients are negative, the factor $\alpha/\eta$
retains the sign needed in the fundamental solution.

If $\alpha\gamma<0$, then $\eta$ is imaginary.
As explained in subsection \ref{sub-general-case}, the modified
Bessel functions can then be expressed using Bessel functions
$J_j$ with real arguments.
Despite its possibly complex representation, $G$ is real-valued
for real coefficients, as is also clear from its power series.

First, we find a solution of the homogeneous equation corresponding
to \eqref{nonsymmetric} with the zero boundary condition at $x=0$:
\begin{align}
\frac{\partial u}{\partial t}(x,t)
&=
\alpha u(x-1,t)+\beta u(x,t)+\gamma u(x+1,t),
\qquad x\in\mathbb N,\quad t\geq0,
\label{nonsym1}\\
u(x,0)&=u_x,\qquad x\in\mathbb N,
\label{nonsym2}\\
u(0,t)&=0,\qquad t\geq0.
\label{nonsym3}
\end{align}

\begin{theorem}\label{dirichlet-problem}
Let $\alpha,\beta,\gamma\in\mathbb R$ with $\alpha,\gamma\neq0$,
and choose $\eta\in\mathbb C$ such that $\eta^2=\alpha\gamma$.
Suppose that $\{u_x\}_{x\in\mathbb N}$ is bounded.
Then the function
$u\colon\mathbb N_0\times[0,\infty)\to\mathbb R$ given by
\[
u(x,t)=\sum_{k\in\mathbb N}u_kQ_{x,k}(t),
\]
where
\begin{align*}
Q_{x,k}(t)
&:=
G(x-k,t)-\left(\frac{\gamma}{\alpha}\right)^kG(x+k,t)\\
&=
{\mathrm e}^{\beta t}
\left(\frac{\alpha}{\eta}\right)^{x-k}
\bigl[I_{x-k}(2\eta t)-I_{x+k}(2\eta t)\bigr],
\qquad x\in\mathbb N_0,\quad k\in\mathbb N,\quad t\geq0,
\end{align*}
satisfies \eqref{nonsym1}--\eqref{nonsym3} and is bounded on
$\mathbb N_0\times[0,T]$ for every $T>0$.
\end{theorem}

\begin{proof}
Since $Q_{0,k}(t)=0$ for all $t\geq0$, we get $u(0,t)=0$.
Since $G$ is the fundamental solution, we also have
$Q_{x,k}(0)=\delta_{kx}$ for all $k,x\in\mathbb N$, so
$u(x,0)=u_x$ for every $x\in\mathbb N$.

The fact that $u$ satisfies \eqref{nonsym1} follows from
Theorem~\ref{superposition-2}, the superposition principle for
solutions on $\mathbb N_0\times[0,T]$.
Here $u$ is an infinite linear combination of the functions
$(x,t)\mapsto Q_{x,k}(t)$
with bounded coefficients $u_k$, and each of these functions is a finite
linear combination of shifted fundamental solutions.

Moreover, for $x\in\mathbb N$,
\[
\sum_{k\in\mathbb N}|Q_{x,k}(0)|
=
\sum_{k\in\mathbb N}\delta_{kx}=1,
\]
whereas
\[
\sum_{k\in\mathbb N}|Q_{0,k}(0)|=0.
\]
Finally, for each fixed $k$, the function
$(x,t)\mapsto Q_{x,k}(t)$ is bounded on $\mathbb N_0\times[0,T]$,
because $G$ is bounded on $\mathbb Z\times[0,T]$.
These observations verify the assumptions of the superposition
theorem.
\end{proof}


\begin{theorem}\label{nonsymmetric-solution}
Let $\alpha,\beta,\gamma\in\mathbb R$ with $\alpha,\gamma\neq0$.
Suppose that:
\begin{enumerate}
\item
$g\colon[0,\infty)\to\mathbb R$ is continuous.
\item
$\{u_x\}_{x\in\mathbb N}$ is bounded.
\item
$f\colon\mathbb N\times[0,\infty)\to\mathbb R$ is bounded on
$\mathbb N\times[0,T]$ for every $T>0$, and
$t\mapsto f(x,t)$ is continuous on $[0,\infty)$ for every
$x\in\mathbb N$.
\end{enumerate}
Then the function
$u\colon\mathbb N_0\times[0,\infty)\to\mathbb R$ given by
\begin{align*}
u(x,t)
&=
g(t)(1-\sgn x)
+\alpha\int_0^t Q_{x,1}(t-s)g(s)\,{\mathrm d}s
+\sum_{k\in\mathbb N}u_kQ_{x,k}(t)+
\int_0^t\sum_{k\in\mathbb N}
f(k,s)Q_{x,k}(t-s)\,{\mathrm d}s,
\end{align*}
is a solution of the initial-boundary value problem
\begin{align*}
\frac{\partial u}{\partial t}(x,t)
&=
\alpha u(x-1,t)+\beta u(x,t)+\gamma u(x+1,t)+f(x,t),
\qquad x\in\mathbb N,\quad t\geq0,\\
u(0,t)&=g(t),\qquad t\geq0,\\
u(x,0)&=u_x,\qquad x\in\mathbb N.
\end{align*}
Moreover, $u$ is the unique solution of this problem that is
bounded on $\mathbb N_0\times[0,T]$ for every $T>0$.
\end{theorem}

\begin{proof}
We use the decomposition
\[
u(x,t)=u_a(x,t)+u_g(x,t)+u_f(x,t),
\]
where
\begin{align*}
u_a(x,t)
&=
\sum_{k\in\mathbb N}u_kQ_{x,k}(t),\\
u_g(x,t)
&=
g(t)(1-\sgn x)
+\alpha\int_0^t Q_{x,1}(t-s)g(s)\,{\mathrm d}s,\\
u_f(x,t)
&=
\int_0^t\sum_{k\in\mathbb N}
f(k,s)Q_{x,k}(t-s)\,{\mathrm d}s.
\end{align*}
Since $Q_{0,1}(t)=0$ for all $t\geq0$, we have $u_g(0,t)=g(t)$.
Also, $u_g(x,0)=0$ for every $x\in\mathbb N$.

The kernel $Q_{x,1}$ satisfies
\[
\frac{\partial Q_{x,1}}{\partial t}(t)
=
\alpha Q_{x-1,1}(t)+\beta Q_{x,1}(t)+\gamma Q_{x+1,1}(t),
\qquad x\in\mathbb N,\quad t\geq0,
\]
because it is a linear combination of shifted fundamental solutions.

We now use the Leibniz rule to calculate
$\frac{\partial u_g}{\partial t}$.
For $x\geq2$, we have $Q_{x,1}(0)=0$, and
\begin{align*}
\frac{\partial u_g}{\partial t}(x,t)
&=
\alpha Q_{x,1}(0)g(t)
+\alpha\int_0^t
\frac{\partial Q_{x,1}}{\partial t}(t-s)g(s)\,{\mathrm d}s\\
&=
\alpha\int_0^t
\bigl[
\alpha Q_{x-1,1}(t-s)+\beta Q_{x,1}(t-s)
+\gamma Q_{x+1,1}(t-s)
\bigr]g(s)\,{\mathrm d}s\\
&=
\alpha u_g(x-1,t)+\beta u_g(x,t)+\gamma u_g(x+1,t).
\end{align*}
For $x=1$, we have $Q_{1,1}(0)=1$, and therefore
\begin{align*}
\frac{\partial u_g}{\partial t}(1,t)
&=
\alpha Q_{1,1}(0)g(t)
+\alpha\int_0^t
\frac{\partial Q_{1,1}}{\partial t}(t-s)g(s)\,{\mathrm d}s\\
&=
\alpha g(t)
+\alpha\int_0^t
\bigl[
\alpha Q_{0,1}(t-s)+\beta Q_{1,1}(t-s)
+\gamma Q_{2,1}(t-s)
\bigr]g(s)\,{\mathrm d}s\\
&=
\alpha u_g(0,t)+\beta u_g(1,t)+\gamma u_g(2,t).
\end{align*}

By Theorem~\ref{dirichlet-problem}, the function $u_a$ satisfies
$u_a(0,t)=0$ for all $t\geq0$, $u_a(x,0)=u_x$ for every
$x\in\mathbb N$, and
\[
\frac{\partial u_a}{\partial t}(x,t)
=
\alpha u_a(x-1,t)+\beta u_a(x,t)+\gamma u_a(x+1,t),
\qquad x\in\mathbb N,\quad t\geq0.
\]

As in the proof of Theorem~\ref{dirichlet-problem}, the
superposition principle for solutions on
$\mathbb N_0\times[0,T]$ (Theorem~\ref{superposition-2})
implies that, for every fixed $s\in[0,\infty)$, the function
\begin{equation}\label{integrand}
(x,t)\mapsto\sum_{k\in\mathbb N}f(k,s)Q_{x,k}(t),
\qquad x\in\mathbb N_0,\quad t\geq0,
\end{equation}
is well defined and satisfies
\begin{align*}
\frac{\partial}{\partial t}
\sum_{k\in\mathbb N}f(k,s)Q_{x,k}(t)
&=
\sum_{k\in\mathbb N}f(k,s)
\bigl[
\alpha Q_{x-1,k}(t)+\beta Q_{x,k}(t)+\gamma Q_{x+1,k}(t)
\bigr]
\end{align*}
for all $x\in\mathbb N$ and $t\geq0$.

We next justify that the integral defining $u_f$ exists and that
the Leibniz rule can be applied to it.
For each fixed $x$, we shall prove uniform convergence of
\[
\sum_{k\in\mathbb N}f(k,s)Q_{x,k}(t-s)
\]
and of the series obtained by differentiating its summands
with respect to $t$, on $0\leq s\leq t\leq T$.
Since the summands and their $t$-derivatives are continuous,
this will justify both integration and differentiation.
Fix $T>0$ and put
\[
H_T:=T\max\{|\alpha|,|\gamma|\},
\qquad
D_T:={\mathrm e}^{|\beta|T+H_T^2}.
\]
For $j\geq0$, the series for the fundamental solution gives
\[
G(j,r)
=
{\mathrm e}^{\beta r}
\sum_{n=0}^{\infty}
\frac{\alpha^{n+j}\gamma^n r^{2n+j}}{n!(n+j)!},
\]
whereas, for $j<0$,
\[
G(j,r)
=
{\mathrm e}^{\beta r}
\sum_{n=0}^{\infty}
\frac{\alpha^n\gamma^{n+|j|}r^{2n+|j|}}
{n!(n+|j|)!}.
\]
Since $|\alpha|r\leq H_T$ and $|\gamma|r\leq H_T$
for $0\leq r\leq T$, both cases yield
\begin{align*}
|G(j,r)|
&\leq
{\mathrm e}^{|\beta|T}
\sum_{n=0}^{\infty}
\frac{H_T^{2n+|j|}}{n!(n+|j|)!}\leq
{\mathrm e}^{|\beta|T}
\frac{H_T^{|j|}}{|j|!}
\sum_{n=0}^{\infty}\frac{H_T^{2n}}{n!}=
{\mathrm e}^{|\beta|T+H_T^2}
\frac{H_T^{|j|}}{|j|!}
=
D_T\frac{H_T^{|j|}}{|j|!}.
\end{align*}
Consequently, for each fixed $x\in\mathbb N_0$,
\[
\sup_{0\leq r\leq T}|Q_{x,k}(r)|
\leq
D_T\left(
\frac{H_T^{|x-k|}}{|x-k|!}
+
\left|\frac{\gamma}{\alpha}\right|^k
\frac{H_T^{x+k}}{(x+k)!}
\right).
\]
The right-hand side is summable over $k$, by the ratio test.
Since $f$ is bounded on $\mathbb N\times[0,T]$, the series
\[
\sum_{k\in\mathbb N}f(k,s)Q_{x,k}(t-s)
\]
therefore converges uniformly on $0\leq s\leq t\leq T$.
Each summand is continuous in $(s,t)$, so the sum is continuous
and the integral defining $u_f$ exists.

Moreover, for $x\in\mathbb N$,
\[
\frac{\partial Q_{x,k}}{\partial r}(r)
=
\alpha Q_{x-1,k}(r)+\beta Q_{x,k}(r)+\gamma Q_{x+1,k}(r).
\]
Applying the same bounds at the neighboring spatial indices
shows that the series
\[
\sum_{k\in\mathbb N}
f(k,s)\frac{\partial Q_{x,k}}{\partial r}(t-s)
\]
also converges uniformly on the triangle $0\leq s\leq t\leq T$.
Its sum is continuous and equals the partial derivative
with respect to $t$ of the preceding series.
Thus the Leibniz rule applies to $u_f$.
Hence,
\begin{align*}
\frac{\partial u_f}{\partial t}(x,t)
&=
\sum_{k\in\mathbb N}f(k,t)Q_{x,k}(0)+
\int_0^t\sum_{k\in\mathbb N}f(k,s)
\bigl[
\alpha Q_{x-1,k}(t-s)+\beta Q_{x,k}(t-s)
+\gamma Q_{x+1,k}(t-s)
\bigr]\,{\mathrm d}s\\
&=
f(x,t)+\alpha u_f(x-1,t)+\beta u_f(x,t)+\gamma u_f(x+1,t).
\end{align*}
Moreover, $u_f(x,0)=0$ for all $x\in\mathbb N$, and $u_f(0,t)
=
0$ for $t\geq0$.

Adding the equations for $u_a$, $u_g$, and $u_f$, together with
their initial and boundary conditions, shows that
$u=u_a+u_g+u_f$ solves the stated initial-boundary value problem.

For every $T>0$, the function $u_a$ is bounded on
$\mathbb N_0\times[0,T]$ by Theorem~\ref{dirichlet-problem}.
Also, $Q_{x,1}(t)=G(x-1,t)-(\gamma/\alpha)G(x+1,t)$
is bounded on $\mathbb N_0\times[0,T]$ because $G$
is bounded on $\mathbb Z\times[0,T]$, and therefore $u_g$
is bounded as well.
The function in \eqref{integrand} is bounded on
$\mathbb N_0\times[0,T]$, uniformly with respect to $s\in[0,T]$;
this follows from boundedness of $f$ and the estimate
\eqref{super-estimate} with $m=1$.
Hence $u_f$ is bounded on $\mathbb N_0\times[0,T]$.
Consequently, $u$ is bounded on the same domain.
Uniqueness for bounded solutions follows from
Theorem~\ref{uniquenessForward}.
\end{proof}

\begin{example}\label{diffusion-on-graph}
A well-studied problem leading to a semidiscrete diffusion
equation with a nonsymmetric right-hand side is diffusion on
a $(q+1)$-regular tree, also known as the Bethe lattice.
This is an infinite tree in which each vertex has degree $q+1$.
Fundamental solutions for both continuous and discrete time,
which make it possible to solve initial-value problems, are
available in the literature; see \cite{CHJSV,CJKS} and the
references therein.

Instead of initial-value problems on the $(q+1)$-regular tree,
we now consider an initial-boundary value problem on a homogeneous
rooted tree, where the root has degree $q$ and every other vertex
has degree $q+1$.
We prescribe the initial values at all vertices except the root,
where a boundary condition is imposed.

Label the layers of the tree by nonnegative integers, with layer
zero corresponding to the root.
Suppose that the initial values are constant on each layer.
By symmetry, this property will be preserved for all $t\geq0$.
Hence, we can denote by $u(x,t)$ the value of the solution at time
$t$ at any vertex in layer $x$.
Since each vertex in layer $x\geq1$ has one neighbor in layer
$x-1$ and $q$ neighbors in layer $x+1$, the problem reduces to
the initial-boundary value problem
\begin{align*}
\frac{\partial u}{\partial t}(x,t)
&=
d\bigl[u(x-1,t)-(q+1)u(x,t)+qu(x+1,t)\bigr],
\qquad x\in\mathbb N,\quad t\geq0,\\
u(0,t)&=g(t),\qquad t\geq0,\\
u(x,0)&=u_x,\qquad x\in\mathbb N,
\end{align*}
whose solution is provided by
Theorem~\ref{nonsymmetric-solution}.
\end{example}

\subsection{Derivation from the unified transform method}\label{sect-7-2}

We now derive an alternative representation of the solution
of the same initial-boundary value problem using the unified
transform method. This approach will also extend to the
time-dependent coefficients considered in the next subsection.

Let $\alpha,\beta,\gamma\in\mathbb C\setminus\{0\}$ and put
\[
\rho:=\sqrt{\left|\frac{\gamma}{\alpha}\right|}.
\]
Given a continuous function $g\colon[0,\infty)\to\mathbb C$ and a
sequence $\{u_x\}_{x\in\mathbb N}$ of complex numbers satisfying
\[
\sum_{x\in\mathbb N}|u_x|\rho^x<\infty,
\]
we seek a function $u\colon\mathbb N_0\times[0,\infty)\to\mathbb C$
such that
\begin{align*}
\frac{\partial u}{\partial t}(x,t)
&=
\alpha u(x-1,t)+\beta u(x,t)+\gamma u(x+1,t),
\quad x\in\mathbb N,\quad t\geq0,\\
u(0,t)&=g(t),\quad t\geq0,\quad u(x,0)=u_x,
\quad x\in\mathbb N.
\end{align*}

We proceed formally, assuming the summability and regularity needed
for the following transform calculations. Define
\[
\hat u(z,t):=\sum_{x\in\mathbb N}u(x,t)z^x,
\qquad |z|\leq\rho,
\]
and put $h(t):=u(1,t)$. Multiplying the equation by $z^x$ and summing
over $x\in\mathbb N$, we obtain
\begin{align*}
\frac{\partial\hat u}{\partial t}(z,t)
&=
\alpha\sum_{x\in\mathbb N}u(x-1,t)z^x
+\beta\sum_{x\in\mathbb N}u(x,t)z^x
+\gamma\sum_{x\in\mathbb N}u(x+1,t)z^x\\
&=
\alpha z\bigl[\hat u(z,t)+g(t)\bigr]
+\beta\hat u(z,t)
+\gamma\bigl[z^{-1}\hat u(z,t)-h(t)\bigr]\\
&=
\left(\alpha z+\beta+\frac{\gamma}{z}\right)\hat u(z,t)
+\alpha zg(t)-\gamma h(t),
\qquad 0<|z|\leq\rho.
\end{align*}
Solving this ordinary differential equation gives
\[
\hat u(z,t)
=
{\mathrm e}^{(\alpha z+\beta+\gamma/z)t}\hat u(z,0)
+\alpha z\tilde g(z,t)-\gamma\tilde h(z,t),
\]
where
\begin{align*}
\hat u(z,0)
&=\sum_{x\in\mathbb N}u_xz^x,\\
\tilde g(z,t)
&=
\int_0^t
g(s){\mathrm e}^{(\alpha z+\beta+\gamma/z)(t-s)}
\,{\mathrm d}s,\\
\tilde h(z,t)
&=
\int_0^t
h(s){\mathrm e}^{(\alpha z+\beta+\gamma/z)(t-s)}
\,{\mathrm d}s.
\end{align*}

The transformation $z\mapsto\gamma/(\alpha z)$ leaves
$\alpha z+\beta+\gamma/z$ unchanged and maps the circle $|z|=\rho$
onto itself. Thus, for $|z|=\rho$, both $z$ and $\gamma/(\alpha z)$ lie in the
domain of the discrete transform, so we may evaluate the transformed
equation at these two points:
\[
\hat u\left(\frac{\gamma}{\alpha z},t\right)
=
{\mathrm e}^{(\alpha z+\beta+\gamma/z)t}
\hat u\left(\frac{\gamma}{\alpha z},0\right)
+\frac{\gamma}{z}\tilde g(z,t)-\gamma\tilde h(z,t).
\]
Subtracting this identity from the preceding formula for $\hat u(z,t)$
eliminates $\tilde h$ and yields
\begin{align*}
\hat u(z,t)-\hat u\left(\frac{\gamma}{\alpha z},t\right)
&=
\left[
\hat u(z,0)-\hat u\left(\frac{\gamma}{\alpha z},0\right)
\right]
{\mathrm e}^{(\alpha z+\beta+\gamma/z)t}+
\left(\alpha z-\frac{\gamma}{z}\right)\tilde g(z,t).
\end{align*}

We now apply the inverse transform, with all circles oriented
counterclockwise. For $x\in\mathbb N$, the residue theorem gives
\[
\frac{1}{2\pi{\mathrm i}}
\int_{|z|=\rho}\hat u(z,t)\frac{{\mathrm d}z}{z^{x+1}}
=u(x,t).
\]
The term involving $\hat u(\gamma/(\alpha z),t)$ has zero contribution.
Indeed, setting
\[
w=\frac{\gamma}{\alpha z},
\qquad
{\mathrm d}z=-\frac{\gamma}{\alpha w^2}\,{\mathrm d}w,
\]
we obtain
\begin{align*}
\int_{|z|=\rho}
\hat u\left(\frac{\gamma}{\alpha z},t\right)
\frac{{\mathrm d}z}{z^{x+1}}
&=
\left(\frac{\alpha}{\gamma}\right)^x
\int_{|w|=\rho}\hat u(w,t)w^{x-1}\,{\mathrm d}w=0.
\end{align*}
Here the minus sign from the differential is canceled by the reversal
of orientation. The last integral vanishes because
\[
\hat u(w,t)w^{x-1}
=
\sum_{y\in\mathbb N}u(y,t)w^{y+x-1}
\]
contains no $w^{-1}$ term.

Consequently,
\begin{align*}
u(x,t)
&=
\frac{1}{2\pi{\mathrm i}}
\int_{|z|=\rho}
\left[
\hat u(z,0)-\hat u\left(\frac{\gamma}{\alpha z},0\right)
\right]
{\mathrm e}^{(\alpha z+\beta+\gamma/z)t}
\frac{{\mathrm d}z}{z^{x+1}}+
\frac{1}{2\pi{\mathrm i}}
\int_{|z|=\rho}
\left(\alpha z-\frac{\gamma}{z}\right)\tilde g(z,t)
\frac{{\mathrm d}z}{z^{x+1}}.
\end{align*}
Substituting the expressions for the initial transform and $\tilde g$,
we arrive at
\begin{align*}
u(x,t)
&=
\frac{1}{2\pi{\mathrm i}}
\int_{|z|=\rho}
\left\{
\sum_{y\in\mathbb N}
u_y\left[
z^y-\left(\frac{\gamma}{\alpha z}\right)^y
\right]
\right\}
{\mathrm e}^{(\alpha z+\beta+\gamma/z)t}
\frac{{\mathrm d}z}{z^{x+1}}\\
&\quad+
\frac{1}{2\pi{\mathrm i}}
\int_{|z|=\rho}
\left(\alpha z-\frac{\gamma}{z}\right)
\left\{
\int_0^t
g(s){\mathrm e}^{(\alpha z+\beta+\gamma/z)(t-s)}
\,{\mathrm d}s
\right\}
\frac{{\mathrm d}z}{z^{x+1}},
\qquad x\in\mathbb N,\quad t\geq0.
\end{align*}
This completes the formal derivation for constant coefficients $\alpha$, $\beta$, $\gamma$.
A rigorous justification will be given in the next subsection,
where we treat the more general case of time-dependent coefficients
with constant ratio $\gamma(t)/\alpha(t)$; see
Theorem~\ref{theo-time-dependent-nonhomogeneous}.
The present problem is recovered by taking the coefficients
to be constant and setting $f=0$.

\subsection{Time-dependent coefficients}

Let $\alpha,\beta,\gamma\colon[0,\infty)\to\mathbb C\setminus\{0\}$
be continuous functions satisfying
\[
\gamma(t)=\lambda\alpha(t),\qquad t\geq0,
\]
for a fixed $\lambda\in\mathbb C\setminus\{0\}$, and put
\[
\rho:=\sqrt{|\lambda|}.
\]
The constant ratio between $\gamma$ and $\alpha$ ensures that the
transformed equation has the same symmetry at every time.

\subsubsection{Homogeneous problem}

Let $g\colon[0,\infty)\to\mathbb C$ be continuous, and let
$\{u_x\}_{x\in\mathbb N}$ satisfy
\[
\sum_{x\in\mathbb N}|u_x|\rho^x<\infty.
\]
We seek a function $u\colon\mathbb N_0\times[0,\infty)\to\mathbb C$
satisfying
\begin{align*}
\frac{\partial u}{\partial t}(x,t)
&=
\alpha(t)u(x-1,t)+\beta(t)u(x,t)+\gamma(t)u(x+1,t),
&&x\in\mathbb N,\quad t\geq0,\\
u(0,t)&=g(t),
&&t\geq0,\\
u(x,0)&=u_x,
&&x\in\mathbb N.
\end{align*}

As before, we proceed formally, assuming the summability and regularity
needed for the transform calculations. Define
\[
\hat u(z,t):=\sum_{x\in\mathbb N}u(x,t)z^x,
\qquad |z|\leq\rho,
\]
and put $h(t):=u(1,t)$. Applying the discrete transform gives
\[
\frac{\partial\hat u}{\partial t}(z,t)
=
\left(\alpha(t)z+\beta(t)+\frac{\gamma(t)}{z}\right)\hat u(z,t)
+z\alpha(t)g(t)-\gamma(t)h(t),
\qquad 0<|z|\leq\rho.
\]

Put
\[
\omega(z,t):=\alpha(t)z+\beta(t)+\frac{\gamma(t)}{z},
\qquad
\Omega(z,t):=\int_0^t\omega(z,s)\,{\mathrm d}s.
\]
Also define
\[
A(t):=\int_0^t\alpha(s)\,{\mathrm d}s,\qquad
B(t):=\int_0^t\beta(s)\,{\mathrm d}s,\qquad
C(t):=\int_0^t\gamma(s)\,{\mathrm d}s.
\]
Since $C(t)=\lambda A(t)$, we have
\[
\Omega(z,t)
=
A(t)z+B(t)+\frac{C(t)}{z}
=
B(t)+A(t)\left(z+\frac{\lambda}{z}\right).
\]

Solving the transformed ordinary differential equation yields
\begin{align*}
\hat u(z,t)
&=
{\mathrm e}^{\Omega(z,t)}\hat u(z,0)+
z\int_0^t
\alpha(s)g(s)
{\mathrm e}^{\Omega(z,t)-\Omega(z,s)}\,{\mathrm d}s-
\int_0^t
\gamma(s)h(s)
{\mathrm e}^{\Omega(z,t)-\Omega(z,s)}\,{\mathrm d}s.
\end{align*}
The relation $\gamma(t)=\lambda\alpha(t)$ implies
\[
\omega(\lambda/z,t)=\omega(z,t),
\qquad
\Omega(\lambda/z,t)=\Omega(z,t).
\]
Moreover, $z\mapsto\lambda/z$ maps the circle $|z|=\rho$ onto itself.
Thus both transformed identities are available on this circle.
Replacing $z$ by $\lambda/z$ and subtracting eliminates the unknown
function $h$:
\begin{align*}
\hat u(z,t)-\hat u(\lambda/z,t)
&=
\bigl[\hat u(z,0)-\hat u(\lambda/z,0)\bigr]
{\mathrm e}^{\Omega(z,t)}+
\left(z-\frac{\lambda}{z}\right)
\int_0^t
\alpha(s)g(s)
{\mathrm e}^{\Omega(z,t)-\Omega(z,s)}\,{\mathrm d}s.
\end{align*}

All contour integrals below are taken counterclockwise. For
$x\in\mathbb N$, the substitution $w=\lambda/z$ gives
\[
\int_{|z|=\rho}
\hat u(\lambda/z,t)\frac{{\mathrm d}z}{z^{x+1}}
=
\lambda^{-x}
\int_{|w|=\rho}\hat u(w,t)w^{x-1}\,{\mathrm d}w
=0.
\]
The minus sign from the differential is canceled by the reversal of
orientation, and the last integral vanishes by termwise integration.
Consequently, inversion gives
\begin{align}
u(x,t)
&=
\frac{1}{2\pi{\mathrm i}}
\int_{|z|=\rho}
\left\{
\sum_{y=1}^{\infty}
u_y\left[z^y-\left(\frac{\lambda}{z}\right)^y\right]
\right\}
{\mathrm e}^{\Omega(z,t)}
\frac{{\mathrm d}z}{z^{x+1}}
\notag\\
&\quad+
\frac{1}{2\pi{\mathrm i}}
\int_{|z|=\rho}
\left(z-\frac{\lambda}{z}\right)
\left\{
\int_0^t
\alpha(s)g(s)
{\mathrm e}^{\Omega(z,t)-\Omega(z,s)}\,{\mathrm d}s
\right\}
\frac{{\mathrm d}z}{z^{x+1}},
\label{time-dependent-homogeneous}
\end{align}
for $x\in\mathbb N$ and $t\geq0$, together with $u(0,t)=g(t)$.

\subsubsection{Nonhomogeneous problem}

We retain the assumptions on the coefficients, initial data, and
boundary data, and consider
\[
\frac{\partial u}{\partial t}(x,t)
=
\alpha(t)u(x-1,t)+\beta(t)u(x,t)+\gamma(t)u(x+1,t)
+f(x,t),
\qquad x\in\mathbb N,\quad t\geq0,
\]
with the same initial and boundary conditions.

For the transform calculation, we temporarily assume that
\[
\sum_{x=1}^{\infty}|f(x,t)|\rho^x<\infty
\]
for every $t\geq0$, and that
\[
\sum_{x=1}^{\infty}
|f(x,t)-f(x,t_0)|\rho^x\longrightarrow0
\qquad\text{as }t\to t_0
\]
for every $t_0\geq0$.
As in the homogeneous case, the derivation is formal.
After evaluating the contour integrals, we shall justify the
resulting series representation under alternative assumptions
that require only continuity in time at each spatial point
and boundedness on compact time intervals, uniformly in space.

Define
\[
\hat f(z,t):=\sum_{x=1}^{\infty}f(x,t)z^x,
\qquad |z|\leq\rho.
\]
The transformed equation now reads
\[
\frac{\partial\hat u}{\partial t}(z,t)
=
\omega(z,t)\hat u(z,t)
+z\alpha(t)g(t)-\gamma(t)h(t)+\hat f(z,t).
\]
Its solution has the additional term
\[
\int_0^t
\hat f(z,s)
{\mathrm e}^{\Omega(z,t)-\Omega(z,s)}\,{\mathrm d}s.
\]
Subtracting the identity evaluated at $\lambda/z$ therefore gives
\begin{align*}
\hat u(z,t)-\hat u(\lambda/z,t)
&=
\bigl[\hat u(z,0)-\hat u(\lambda/z,0)\bigr]
{\mathrm e}^{\Omega(z,t)}\\
&\quad+
\left(z-\frac{\lambda}{z}\right)
\int_0^t
\alpha(s)g(s)
{\mathrm e}^{\Omega(z,t)-\Omega(z,s)}\,{\mathrm d}s+
\int_0^t
\bigl[\hat f(z,s)-\hat f(\lambda/z,s)\bigr]
{\mathrm e}^{\Omega(z,t)-\Omega(z,s)}\,{\mathrm d}s.
\end{align*}
Inverting as above, we obtain
\begin{align}
u(x,t)
&=
\frac{1}{2\pi{\mathrm i}}
\int_{|z|=\rho}
\bigl[\hat u(z,0)-\hat u(\lambda/z,0)\bigr]
{\mathrm e}^{\Omega(z,t)}
\frac{{\mathrm d}z}{z^{x+1}}
\notag\\
&\quad+
\frac{1}{2\pi{\mathrm i}}
\int_{|z|=\rho}
\left(z-\frac{\lambda}{z}\right)
\left\{
\int_0^t
\alpha(s)g(s)
{\mathrm e}^{\Omega(z,t)-\Omega(z,s)}\,{\mathrm d}s
\right\}
\frac{{\mathrm d}z}{z^{x+1}}
\notag\\
&\quad+
\frac{1}{2\pi{\mathrm i}}
\int_{|z|=\rho}
\left\{
\int_0^t
\bigl[\hat f(z,s)-\hat f(\lambda/z,s)\bigr]
{\mathrm e}^{\Omega(z,t)-\Omega(z,s)}\,{\mathrm d}s
\right\}
\frac{{\mathrm d}z}{z^{x+1}}.
\label{time-dependent-nonhomogeneous}
\end{align}

\subsubsection{Evaluation of the contour integrals}

To evaluate the contour integrals, we expand the exponential in a
Laurent series. For $0\leq s\leq t$ and $\ell\in\mathbb Z$, define
\[
E_\ell(s,t):=
{\mathrm e}^{B(t)-B(s)}
\lambda^{\max\{0,-\ell\}}
\sum_{n=0}^{\infty}
\frac{\lambda^n[A(t)-A(s)]^{2n+|\ell|}}
{n!(n+|\ell|)!}.
\]
In particular, for $\ell\geq0$,
\[
E_{-\ell}(s,t)=\lambda^\ell E_\ell(s,t).
\]
This weighted symmetry replaces the symmetry of the kernels in the
case $\lambda=1$.
Indeed,
\begin{align*}
{\mathrm e}^{\Omega(z,t)-\Omega(z,s)}
&=
{\mathrm e}^{B(t)-B(s)}
{\mathrm e}^{[A(t)-A(s)]z}
{\mathrm e}^{\lambda[A(t)-A(s)]/z}\\
&=
{\mathrm e}^{B(t)-B(s)}
\sum_{n,m\geq0}
\frac{\lambda^m[A(t)-A(s)]^{n+m}}{n!m!}z^{n-m}.
\end{align*}
For $\ell\geq0$, the coefficient of $z^\ell$ is obtained by setting
$n=m+\ell$, whereas the coefficient of $z^{-\ell}$ is obtained by
setting $m=n+\ell$. Hence
\begin{equation}\label{time-dependent-laurent}
{\mathrm e}^{\Omega(z,t)-\Omega(z,s)}
=
\sum_{\ell\in\mathbb Z}E_\ell(s,t)z^\ell.
\end{equation}
Since $A(0)=B(0)=0$, this also gives
\[
{\mathrm e}^{\Omega(z,t)}
=
\sum_{\ell\in\mathbb Z}E_\ell(0,t)z^\ell.
\]

The Laurent series converges absolutely and uniformly on $|z|=\rho$,
uniformly for $0\leq s\leq t\leq T$, for every $T>0$.
Together with the weighted summability of the initial data and the temporary assumptions on $f$ and continuity of $g$, this justifies the following termwise
integrations.

For the initial-data contribution, extracting the coefficient of
$z^{-1}$ gives
\begin{align*}
&\frac{1}{2\pi{\mathrm i}}
\int_{|z|=\rho}
\bigl[\hat u(z,0)-\hat u(\lambda/z,0)\bigr]
{\mathrm e}^{\Omega(z,t)}
\frac{{\mathrm d}z}{z^{x+1}}=
\sum_{y=1}^{\infty}
u_y\bigl[E_{x-y}(0,t)-\lambda^yE_{x+y}(0,t)\bigr].
\end{align*}
Here $u_yz^y$ selects the coefficient $E_{x-y}$, while
$u_y\lambda^yz^{-y}$ selects $\lambda^yE_{x+y}$.

Similarly, the boundary-data contribution is
\begin{align*}
&\frac{1}{2\pi{\mathrm i}}
\int_{|z|=\rho}
\left(z-\frac{\lambda}{z}\right)
\left\{
\int_0^t
\alpha(s)g(s)
{\mathrm e}^{\Omega(z,t)-\Omega(z,s)}\,{\mathrm d}s
\right\}
\frac{{\mathrm d}z}{z^{x+1}}=
\int_0^t
\alpha(s)g(s)
\bigl[E_{x-1}(s,t)-\lambda E_{x+1}(s,t)\bigr]
\,{\mathrm d}s.
\end{align*}
Finally, the forcing contribution is
\begin{align*}
&\frac{1}{2\pi{\mathrm i}}
\int_{|z|=\rho}
\left\{
\int_0^t
\bigl[\hat f(z,s)-\hat f(\lambda/z,s)\bigr]
{\mathrm e}^{\Omega(z,t)-\Omega(z,s)}\,{\mathrm d}s
\right\}
\frac{{\mathrm d}z}{z^{x+1}}=
\int_0^t
\sum_{y=1}^{\infty}
f(y,s)\bigl[E_{x-y}(s,t)-\lambda^yE_{x+y}(s,t)\bigr]
\,{\mathrm d}s.
\end{align*}

Combining these expressions, we obtain
\begin{equation}\label{time-dependent-series-solution}
\begin{aligned}
u(x,t)
&=
\sum_{y=1}^{\infty}
u_y\bigl[E_{x-y}(0,t)-\lambda^yE_{x+y}(0,t)\bigr]+
\int_0^t
\alpha(s)g(s)
\bigl[E_{x-1}(s,t)-\lambda E_{x+1}(s,t)\bigr]
\,{\mathrm d}s\\
&\quad+
\int_0^t
\sum_{y=1}^{\infty}
f(y,s)\bigl[E_{x-y}(s,t)-\lambda^yE_{x+y}(s,t)\bigr]
\,{\mathrm d}s,
\qquad x\in\mathbb N,\quad t\geq0,
\end{aligned}
\end{equation}
together with $u(0,t)=g(t)$. 

\begin{theorem}\label{theo-time-dependent-nonhomogeneous}
Let $\alpha,\gamma\colon[0,\infty)\to\mathbb C\setminus\{0\}$
and $\beta\colon[0,\infty)\to\mathbb C$ be continuous and satisfy
\[
\gamma(t)=\lambda\alpha(t),
\qquad t\geq0,
\]
for a fixed $\lambda\in\mathbb C\setminus\{0\}$.
Put $\rho=\sqrt{|\lambda|}$.
Suppose that $g\colon[0,\infty)\to\mathbb C$ is continuous and
\[
\sum_{y=1}^{\infty}|u_y|\rho^y<\infty.
\]
Let $f\colon\mathbb N\times[0,\infty)\to\mathbb C$ be continuous
in $t$ for each $x\in\mathbb N$ and satisfy
\[
\sup_{\substack{x\in\mathbb N\\0\leq t\leq T}}
|f(x,t)|<\infty
\qquad\text{for every }T>0.
\]
Then \eqref{time-dependent-series-solution}, together with
$u(0,t)=g(t)$, defines a solution of
\[
\frac{\partial u}{\partial t}(x,t)
=
\alpha(t)u(x-1,t)+\beta(t)u(x,t)+\gamma(t)u(x+1,t)
+f(x,t),
\qquad x\in\mathbb N,\quad t\geq0,
\]
with $u(x,0)=u_x$ for $x\in\mathbb N$.
\end{theorem}

\begin{proof}
Put
\[
Q_{x,y}(s,t):=E_{x-y}(s,t)-\lambda^yE_{x+y}(s,t),
\qquad x,y\in\mathbb N,\quad 0\leq s\leq t.
\]
When $\alpha$, $\beta$, and $\gamma$ are constant, this kernel
depends only on $t-s$ and agrees with the kernel introduced
in Theorem \ref{dirichlet-problem}:
\[
Q_{x,y}(s,t)=Q_{x,y}(t-s).
\]
The weighted symmetry of $E_\ell$ gives
\[
Q_{0,y}(s,t):=E_{-y}(s,t)-\lambda^yE_y(s,t)=0.
\]
Using the definition of $Q_{x,y}$, we can rewrite the
proposed solution as
\begin{align*}
u(x,t)
&=
\sum_{y=1}^{\infty}u_yQ_{x,y}(0,t)
+
\int_0^t\alpha(s)g(s)Q_{x,1}(s,t)\,{\mathrm d}s+
\int_0^t\sum_{y=1}^{\infty}
f(y,s)Q_{x,y}(s,t)\,{\mathrm d}s,
\qquad x\in\mathbb N.
\end{align*}

We first establish convergence.
Fix $T>0$ and put
\[
L_T:=\rho\int_0^T|\alpha(r)|\,{\mathrm d}r,
\qquad
D_T:=
\exp\left(\int_0^T|\beta(r)|\,{\mathrm d}r+L_T^2\right).
\]
For $0\leq s\leq t\leq T$, write
$d:=|A(t)-A(s)|$. Then
\[
d
=
\left|\int_s^t\alpha(r)\,{\mathrm d}r\right|
\leq
\int_0^T|\alpha(r)|\,{\mathrm d}r,
\qquad
\rho d\leq L_T.
\]
Similarly,
\[
\left|{\mathrm e}^{B(t)-B(s)}\right|
=
{\mathrm e}^{\operatorname{Re}(B(t)-B(s))}
\leq
\exp\left(\int_0^T|\beta(r)|\,{\mathrm d}r\right).
\]

Since $|\lambda|=\rho^2$, the defining series for $E_\ell$
therefore gives
\begin{align*}
|E_\ell(s,t)|
&\leq
\exp\left(\int_0^T|\beta(r)|\,{\mathrm d}r\right)
\rho^{2\max\{0,-\ell\}}
\sum_{n=0}^{\infty}
\frac{\rho^{2n}d^{2n+|\ell|}}
{n!(n+|\ell|)!}\\
&=
\exp\left(\int_0^T|\beta(r)|\,{\mathrm d}r\right)
\rho^{-\ell}(\rho d)^{|\ell|}
\sum_{n=0}^{\infty}
\frac{(\rho d)^{2n}}{n!(n+|\ell|)!}.
\end{align*}
Here we used
$2\max\{0,-\ell\}=|\ell|-\ell$, which implies
$\rho^{2\max\{0,-\ell\}}d^{|\ell|}
=
\rho^{-\ell}(\rho d)^{|\ell|}$.
Using $(n+|\ell|)!\geq|\ell|!$, we further obtain
\begin{align*}
\sum_{n=0}^{\infty}
\frac{(\rho d)^{2n}}{n!(n+|\ell|)!}
&\leq
\frac{1}{|\ell|!}
\sum_{n=0}^{\infty}\frac{(\rho d)^{2n}}{n!}=
\frac{{\mathrm e}^{(\rho d)^2}}{|\ell|!}
\leq
\frac{{\mathrm e}^{L_T^2}}{|\ell|!}.
\end{align*}
Combining these estimates with $\rho d\leq L_T$ yields
\[
|E_\ell(s,t)|
\leq
D_T\rho^{-\ell}
\frac{L_T^{|\ell|}}{|\ell|!},
\qquad 0\leq s\leq t\leq T.
\]
Consequently,
\begin{equation}\label{time-dependent-kernel-bound}
|Q_{x,y}(s,t)|
\leq
D_T\rho^{y-x}
\left(
\frac{L_T^{|x-y|}}{|x-y|!}
+
\frac{L_T^{x+y}}{(x+y)!}
\right).
\end{equation}

For each fixed $x$, the right-hand side is summable over $y$,
by the ratio test.
It therefore gives a summable uniform majorant for
$f(y,s)Q_{x,y}(s,t)$ on $0\leq s\leq t\leq T$.

For the initial-data series, we also have
\[
|u_yQ_{x,y}(0,t)|
\leq
2D_T\rho^{-x}{\mathrm e}^{L_T}|u_y|\rho^y.
\]
For each fixed $x$, these bounds are summable over $y$, since
\[
\sum_{y=1}^{\infty}
2D_T\rho^{-x}{\mathrm e}^{L_T}|u_y|\rho^y
=
2D_T\rho^{-x}{\mathrm e}^{L_T}
\sum_{y=1}^{\infty}|u_y|\rho^y
<\infty.
\]
The Weierstrass M-test proves absolute and uniform convergence
of the initial-data series on $[0,T]$ and of the forcing series
on the triangle $0\leq s\leq t\leq T$.
Their summands are continuous, so the sums are continuous and
the time integrals exist.

We next verify the equation.
Differentiating the Laurent expansion
\eqref{time-dependent-laurent} with respect to $t$ and comparing
coefficients gives
\[
\frac{\partial E_\ell}{\partial t}(s,t)
=
\alpha(t)E_{\ell-1}(s,t)
+\beta(t)E_\ell(s,t)
+\gamma(t)E_{\ell+1}(s,t).
\]
Here differentiation of the Laurent coefficients is justified
by differentiation under their contour integrals, whose
integrands and $t$-derivatives are continuous on the fixed
circle and compact time intervals.
Hence
\begin{equation}\label{time-dependent-kernel-equation}
\frac{\partial Q_{x,y}}{\partial t}(s,t)
=
\alpha(t)Q_{x-1,y}(s,t)
+\beta(t)Q_{x,y}(s,t)
+\gamma(t)Q_{x+1,y}(s,t).
\end{equation}
Also, $E_\ell(s,s)=\delta_{\ell,0}$, so
\[
Q_{x,y}(s,s)=\delta_{x,y},
\qquad x,y\in\mathbb N.
\]

Since $\alpha$, $\beta$, and $\gamma$ are continuous,
they are bounded on $[0,T]$.
Thus \eqref{time-dependent-kernel-equation} and
\eqref{time-dependent-kernel-bound}, applied at the neighboring
spatial indices, give summable uniform majorants for the
differentiated series as well.
At $x=1$, the term with spatial index $0$ is identically zero.
Termwise differentiation and differentiation under the time
integrals are therefore justified.

Write the three contributions as $u_a$, $u_g$, and $u_f$,
respectively, and extend each temporarily by zero at $x=0$.
Using \eqref{time-dependent-kernel-equation},
the identity $Q_{x,y}(t,t)=\delta_{x,y}$, and the Leibniz
rule, we obtain, for $x\in\mathbb N$,
\begin{align*}
\frac{\partial u_a}{\partial t}(x,t)
&=
\alpha(t)u_a(x-1,t)+\beta(t)u_a(x,t)
+\gamma(t)u_a(x+1,t),\\
\frac{\partial u_g}{\partial t}(x,t)
&=
\alpha(t)u_g(x-1,t)+\beta(t)u_g(x,t)
+\gamma(t)u_g(x+1,t)
+\alpha(t)g(t)\delta_{x,1},\\
\frac{\partial u_f}{\partial t}(x,t)
&=
\alpha(t)u_f(x-1,t)+\beta(t)u_f(x,t)
+\gamma(t)u_f(x+1,t)+f(x,t).
\end{align*}
Adding these identities and assigning the full solution its
prescribed boundary value $u(0,t)=g(t)$ gives the required equation:
at $x=1$, the term $\alpha(t)g(t)\delta_{x,1}$ is precisely the
boundary contribution $\alpha(t)u(0,t)$.

Finally, $Q_{x,y}(0,0)=\delta_{x,y}$ and both time integrals
vanish at $t=0$, so $u(x,0)=u_x$.
This completes the proof.
\end{proof}

\subsection{Asymptotic periodicity}

We now investigate asymptotic periodicity for equations with
non-symmetric coefficients and periodic boundary data and forcing.
Unlike in the symmetric diffusion problem, this behavior requires
restrictions on the coefficients.
The following example identifies conditions that are necessary
if asymptotic periodicity is to hold for every admissible choice
of data.

\begin{example}\label{nonsymmetric-periodicity-necessity}
Let $\alpha,\gamma\in\mathbb R\setminus\{0\}$,
$\beta\in\mathbb C$, and $T>0$.
Consider the problem
\begin{gather*}
\frac{\partial u}{\partial t}(x,t)
=
\alpha u(x-1,t)+\beta u(x,t)+\gamma u(x+1,t),
\qquad x\in\mathbb N,\quad t\geq0,\\
u(0,t)=0,\qquad t\geq0,
\qquad
u(x,0)=\delta_{x,1},\qquad x\in\mathbb N.
\end{gather*}
The solution is $u(x,t)=Q_{x,1}(t)$.
In particular,
\[
u(1,t)
=
{\mathrm e}^{\beta t}
\bigl[I_0(2\eta t)-I_2(2\eta t)\bigr],
\qquad \eta^2=\alpha\gamma.
\]
Using the recurrence relation
\[
I_0(z)-I_2(z)=\frac{2}{z}I_1(z),
\]
we obtain
\[
u(1,t)
=
{\mathrm e}^{\beta t}\frac{I_1(2\eta t)}{\eta t},
\qquad t>0.
\]
We examine whether
\[
u(1,t+T)-u(1,t)\longrightarrow0
\qquad\text{as }t\to\infty.
\]

Suppose first that $\alpha\gamma>0$, and put
$d:=\sqrt{\alpha\gamma}>0$, so that we may choose $\eta=d$.
The standard asymptotic formula
\[
I_1(s)
=
\frac{{\mathrm e}^{s}}{\sqrt{2\pi s}}
\bigl(1+O(s^{-1})\bigr),
\qquad s\to+\infty,
\]
gives
\[
u(1,t)
=
\frac{{\mathrm e}^{(\beta+2d)t}}
{2\sqrt{\pi}\,d^{3/2}t^{3/2}}
\bigl(1+O(t^{-1})\bigr).
\]
Since $T$ is fixed and
\[
(t+T)^{-3/2}
=
t^{-3/2}\bigl(1+O(t^{-1})\bigr),
\]
it follows that
\[
u(1,t+T)-u(1,t)
=
\frac{{\mathrm e}^{(\beta+2d)t}}
{2\sqrt{\pi}\,d^{3/2}t^{3/2}}
\left[
{\mathrm e}^{(\beta+2d)T}-1+O(t^{-1})
\right].
\]
If $\operatorname{Re}\beta+2d>0$, then
\[
\left|{\mathrm e}^{(\beta+2d)T}\right|>1,
\]
and hence ${\mathrm e}^{(\beta+2d)T}-1\neq0$.
Consequently,
\[
|u(1,t+T)-u(1,t)|\longrightarrow\infty.
\]
Thus asymptotic $T$-periodicity fails in this case.

Suppose next that $\alpha\gamma<0$, and put
$d:=\sqrt{|\alpha\gamma|}>0$, so that we may choose
$\eta={\mathrm i}d$.
The identity $I_1({\mathrm i}s)={\mathrm i}J_1(s)$ yields
\[
u(1,t)
=
{\mathrm e}^{\beta t}\frac{J_1(2dt)}{dt}.
\]
Using the standard asymptotic formula
\[
J_1(s)
=
\sqrt{\frac{2}{\pi s}}
\left[
\cos\left(s-\frac{3\pi}{4}\right)+O(s^{-1})
\right],
\qquad s\to+\infty,
\]
we obtain
\[
u(1,t)
=
\frac{{\mathrm e}^{\beta t}}
{\sqrt{\pi}\,d^{3/2}t^{3/2}}
\left[
\cos\left(2dt-\frac{3\pi}{4}\right)+O(t^{-1})
\right].
\]
Therefore, writing
\[
\vartheta(t):=2dt-\frac{3\pi}{4},
\]
we have
\[
u(1,t+T)-u(1,t)
=
\frac{{\mathrm e}^{\beta t}}
{\sqrt{\pi}\,d^{3/2}t^{3/2}}
\left[
{\mathrm e}^{\beta T}\cos(\vartheta(t)+2dT)
-\cos\vartheta(t)+O(t^{-1})
\right].
\]

Assume that $\operatorname{Re}\beta>0$.
The $2\pi$-periodic function
\[
F(\vartheta)
:=
{\mathrm e}^{\beta T}\cos(\vartheta+2dT)-\cos\vartheta
=
\bigl[{\mathrm e}^{\beta T}\cos(2dT)-1\bigr]\cos\vartheta
-
{\mathrm e}^{\beta T}\sin(2dT)\sin\vartheta
\]
is not identically zero. Otherwise,
 both coefficients would vanish:
\[
{\mathrm e}^{\beta T}\cos(2dT)=1,
\qquad
{\mathrm e}^{\beta T}\sin(2dT)=0.
\]
The second equality implies $\sin(2dT)=0$, so
$|\cos(2dT)|=1$.
The first would then imply
$|{\mathrm e}^{\beta T}|=1$, contradicting
$\operatorname{Re}\beta>0$.

Choose $\vartheta_0\in[0,2\pi)$ such that
$F(\vartheta_0)\neq0$, and set
\[
t_n:=\frac{\vartheta_0+2\pi n+3\pi/4}{2d},
\qquad n\in\mathbb N.
\]
Then $t_n\to\infty$ and
$F(\vartheta(t_n))=F(\vartheta_0)$.
Consequently,
\[
|u(1,t_n+T)-u(1,t_n)|
=
\frac{{\mathrm e}^{\operatorname{Re}\beta\,t_n}}
{\sqrt{\pi}\,d^{3/2}t_n^{3/2}}
\left|F(\vartheta_0)+O(t_n^{-1})\right|
\longrightarrow\infty.
\]
Thus asymptotic $T$-periodicity fails in this case as well.

We conclude that, even for finitely supported initial data
and zero boundary data and forcing, asymptotic $T$-periodicity
can fail.
For this property to hold for every admissible choice of data,
it is necessary that
\[
\begin{cases}
\operatorname{Re}\beta+2\sqrt{\alpha\gamma}\leq0,
&\alpha\gamma>0,\\
\operatorname{Re}\beta\leq0,
&\alpha\gamma<0.
\end{cases}
\]
\end{example}

The next theorem shows that these necessary conditions are also
sufficient under weighted summability assumptions on the initial
data and forcing.

\begin{theorem}\label{nonsymmetric-asymptotic-periodicity}
Let $\alpha,\gamma\in\mathbb R\setminus\{0\}$,
$\beta\in\mathbb C$, and $T>0$, and put
\[
\rho:=\sqrt{\left|\frac{\gamma}{\alpha}\right|}.
\]
Suppose that
\[
\begin{cases}
\operatorname{Re}\beta+2\sqrt{\alpha\gamma}\leq0,
&\alpha\gamma>0,\\
\operatorname{Re}\beta\leq0,
&\alpha\gamma<0.
\end{cases}
\]
Let $(u_y)_{y\in\mathbb N}$ be a complex sequence satisfying
\[
\sum_{y=1}^{\infty}|u_y|\rho^y<\infty,
\]
and let $g\colon[0,\infty)\to\mathbb C$ be continuous and
$T$-periodic.
Suppose that $f\colon\mathbb N\times[0,\infty)\to\mathbb C$
is continuous and $T$-periodic in $t$ for every $y\in\mathbb N$,
and that there exists a nonnegative sequence $(h_y)$ such that
\[
|f(y,t)|\leq h_y,
\qquad y\in\mathbb N,\quad t\geq0,
\qquad
\sum_{y=1}^{\infty}h_y\rho^y<\infty.
\]

Consider the problem
\[
\frac{\partial u}{\partial t}(x,t)
=
\alpha u(x-1,t)+\beta u(x,t)+\gamma u(x+1,t)+f(x,t),
\qquad x\in\mathbb N,\quad t\geq0,
\]
with
\[
u(x,0)=u_x,\qquad x\in\mathbb N,
\qquad
u(0,t)=g(t),\qquad t\geq0.
\]
Define
\[
\eta:=
\begin{cases}
\sqrt{\alpha\gamma},&\alpha\gamma>0,\\
{\mathrm i}\sqrt{|\alpha\gamma|},&\alpha\gamma<0,
\end{cases}
\qquad
r:=\frac{\eta}{\alpha},
\]
and
\[
Q_{x,y}(t):=
r^{y-x}{\mathrm e}^{\beta t}
\bigl[I_{x-y}(2\eta t)-I_{x+y}(2\eta t)\bigr].
\]
Let $u$ be the solution represented by
\[
u(x,t)=u_a(x,t)+u_g(x,t)+u_f(x,t),\quad x\in\mathbb N,
\]
where
\begin{align*}
u_a(x,t)
&=
\sum_{y=1}^{\infty}u_yQ_{x,y}(t),\\
u_g(x,t)
&=
\alpha\int_0^t g(s)Q_{x,1}(t-s)\,{\mathrm d}s,\\
u_f(x,t)
&=
\int_0^t\sum_{y=1}^{\infty}
f(y,s)Q_{x,y}(t-s)\,{\mathrm d}s
\end{align*}
are the contributions from the initial data, boundary data,
and forcing, respectively.
Then, for every fixed $x\in\mathbb N$,
\[
u(x,t+T)-u(x,t)\longrightarrow0
\qquad\text{as }t\to\infty.
\]
\end{theorem}

\begin{proof}
Since $r^2=\gamma/\alpha$ and $|r|=\rho$,
Lemma~\ref{symmetrization} shows that
$v(x,t):=r^xu(x,t)$ satisfies the symmetric equation
\[
\frac{\partial v}{\partial t}(x,t)
=
\eta v(x-1,t)+\beta v(x,t)+\eta v(x+1,t)+r^xf(x,t).
\]
The transformed initial data are $v(y,0)=r^yu_y$, and
they form an absolutely summable sequence because
\[
\sum_{y=1}^{\infty}|v(y,0)|
=
\sum_{y=1}^{\infty}|r|^y|u_y|
=
\sum_{y=1}^{\infty}\rho^y|u_y|
<\infty.
\]
Moreover,
\[
\sup_{\substack{y\in\mathbb N\\t\geq0}}|r^yf(y,t)|
\leq
\sum_{y=1}^{\infty}\rho^yh_y<\infty.
\]
Thus Theorem~\ref{theo-time-dependent-nonhomogeneous},
applied with $\alpha(t)=\gamma(t)=\eta$ and $\beta(t)=\beta$,
justifies the solution representation.
Introduce the symmetric kernel
\[
H_{x,y}(t):=
{\mathrm e}^{\beta t}
\bigl[I_{x-y}(2\eta t)-I_{x+y}(2\eta t)\bigr].
\]
The same calculation as Lemma \ref{lem:dirichlet-kernel-properties} gives
\begin{equation}\label{nonsymmetric-sine-kernel}
H_{x,y}(t)
=
\frac2\pi\int_0^\pi
{\mathrm e}^{(\beta+2\eta\cos\theta)t}
\sin(x\theta)\sin(y\theta)\,{\mathrm d}\theta.
\end{equation}
This identity also holds for complex $\eta$.
The assumptions on the coefficients imply
\[
\left|{\mathrm e}^{(\beta+2\eta\cos\theta)t}\right|
\leq1,
\qquad 0\leq\theta\leq\pi,\quad t\geq0.
\]
Consequently,
\begin{equation}\label{nonsymmetric-kernel-bound}
|H_{x,y}(t)|\leq2,
\qquad x,y\in\mathbb N,\quad t\geq0.
\end{equation}

We next prove that
\begin{equation}\label{nonsymmetric-kernel-decay}
H_{x,y}(t)\longrightarrow0
\qquad\text{as }t\to\infty
\end{equation}
for every fixed $x,y\in\mathbb N$.
For each fixed integer $n$, the standard Bessel asymptotics give
\[
I_n(s)=\frac{{\mathrm e}^{s}}{\sqrt{2\pi s}}
\bigl(1+O(s^{-1})\bigr),
\qquad
J_n(s)=O(s^{-1/2}),
\qquad s\to+\infty.
\]
If $\alpha\gamma>0$, then $\eta>0$, and hence
\[
\left|{\mathrm e}^{\beta t}I_n(2\eta t)\right|
=
O\!\left(
{\mathrm e}^{(\operatorname{Re}\beta+2\eta)t}t^{-1/2}
\right)
\longrightarrow0.
\]
If $\alpha\gamma<0$, write $\eta={\mathrm i}d$, where $d>0$.
Using $I_n({\mathrm i}s)={\mathrm i}^{\,n}J_n(s)$, we obtain
\[
\left|{\mathrm e}^{\beta t}I_n(2\eta t)\right|
=
O\!\left(
{\mathrm e}^{\operatorname{Re}\beta\,t}t^{-1/2}
\right)
\longrightarrow0.
\]
Applying these limits with $n=x-y$ and $n=x+y$ proves
\eqref{nonsymmetric-kernel-decay}.

Since
\[
r^xQ_{x,y}(t)=r^yH_{x,y}(t),
\qquad
\alpha r=\eta,
\]
multiplying the formulas for $u_a$, $u_g$, and $u_f$
by $r^x$ gives
\[
v(x,t)=v_a(x,t)+v_g(x,t)+v_f(x,t),
\]
where
\begin{align*}
v_a(x,t)
&=
\sum_{y=1}^{\infty}r^yu_yH_{x,y}(t),\\
v_g(x,t)
&=
\eta\int_0^t g(s)H_{x,1}(t-s)\,{\mathrm d}s,\\
v_f(x,t)
&=
\int_0^t\sum_{y=1}^{\infty}
r^yf(y,s)H_{x,y}(t-s)\,{\mathrm d}s.
\end{align*}
The bound \eqref{nonsymmetric-kernel-bound} and the summable
majorants in the assumptions justify the interchanges of
sums and time integrals below.

Fix $x\in\mathbb N$.
For every $N\in\mathbb N$, \eqref{nonsymmetric-kernel-bound} implies
\[
\left|
v_a(x,t)-\sum_{y=1}^N r^yu_yH_{x,y}(t)
\right|
\leq
2\sum_{y>N}|u_y|\rho^y.
\]
The right-hand side tends to zero as $N\to\infty$,
uniformly in $t$, while each finite sum tends to zero
as $t\to\infty$ by \eqref{nonsymmetric-kernel-decay}.
Therefore $v_a(x,t)\to0$, and hence
\[
v_a(x,t+T)-v_a(x,t)\longrightarrow0.
\]

Splitting the boundary integral at $s=T$ and using
$g(s+T)=g(s)$ gives
\[
v_g(x,t+T)-v_g(x,t)
=
\eta\int_0^T
g(s)H_{x,1}(t+T-s)\,{\mathrm d}s.
\]
Consequently,
\[
|v_g(x,t+T)-v_g(x,t)|
\leq
|\eta|
\left(\int_0^T|g(s)|\,{\mathrm d}s\right)
\sup_{t\leq\tau\leq t+T}|H_{x,1}(\tau)|
\longrightarrow0.
\]

Similarly, periodicity of $f$ gives
\[
v_f(x,t+T)-v_f(x,t)
=
\int_0^T\sum_{y=1}^{\infty}
r^yf(y,s)H_{x,y}(t+T-s)\,{\mathrm d}s.
\]
The contribution from $y>N$ has absolute value at most
\[
2T\sum_{y>N}h_y\rho^y,
\]
uniformly in $t$.
For each fixed $N$, the remaining contribution has
absolute value at most
\[
T\sum_{y=1}^N
h_y\rho^y
\sup_{t\leq\tau\leq t+T}|H_{x,y}(\tau)|,
\]
which tends to zero by \eqref{nonsymmetric-kernel-decay}.
Choosing $N$ sufficiently large and then $t$ sufficiently
large therefore gives
\[
v_f(x,t+T)-v_f(x,t)\longrightarrow0.
\]

Combining the three contributions, we obtain
\[
v(x,t+T)-v(x,t)\longrightarrow0.
\]
Since $u(x,t)=r^{-x}v(x,t)$ and $x$ is fixed, the same
conclusion holds for $u$.
\end{proof}

We now obtain quantitative estimates under stronger moment assumptions.

\begin{theorem}\label{nonsymmetric-quantitative-asymptotic-periodicity}
Let $\alpha,\gamma\in\mathbb R\setminus\{0\}$,
$\beta\in\mathbb C$, $T>0$, and put
\[
\rho:=\sqrt{\left|\frac{\gamma}{\alpha}\right|}.
\]
Suppose that $g\colon[0,\infty)\to\mathbb C$ is continuous
and $T$-periodic, and that
$f\colon\mathbb N\times[0,\infty)\to\mathbb C$ is continuous
and $T$-periodic in $t$ for every $x\in\mathbb N$.
Assume that 
\[
\sum_{x=1}^{\infty}x^3|u_x|\rho^x<\infty,
\qquad
\sum_{x=1}^{\infty}x^3\rho^x
\sup_{0\leq s\leq T}|f(x,s)|<\infty.
\]
Let $u$ be the previously established solution of the problem
\begin{align*}
\frac{\partial u}{\partial t}(x,t)
&=
\alpha u(x-1,t)+\beta u(x,t)+\gamma u(x+1,t)+f(x,t),
\quad x\in\mathbb N,\quad t\geq0,\\
u(0,t)&=g(t),\quad t\geq0,\quad
u(x,0)=u_x,\quad x\in\mathbb N.
\end{align*}
Write
\[
u(x,t)=u_a(x,t)+u_g(x,t)+u_f(x,t),
\]
where $u_a$, $u_g$, and $u_f$ are the contributions to $u(x,t)$
from the initial data, boundary data, and forcing, respectively.
Let $x\in\mathbb N$ be fixed.
\begin{enumerate}
\item
If $\alpha\gamma>0$, put
\[
d:=\sqrt{\alpha\gamma},
\qquad r:=\frac d\alpha.
\]
Then for $t\to\infty$,
\begin{align*}
u_a(x,t)
&=
\frac{xr^{-x}{\mathrm e}^{(\beta+2d)t}}
{2\sqrt{\pi}\,d^{3/2}}
\left[
\left(\sum_{y=1}^{\infty}yu_yr^y\right)t^{-3/2}
+O(t^{-5/2})
\right],\\
u_g(x,t+T)-u_g(x,t)
&=
\frac{xr^{-x}{\mathrm e}^{(\beta+2d)t}}
{2\sqrt{\pi}\,d^{3/2}}
\left[
d\left(\int_0^Tg(s){\mathrm e}^{(\beta+2d)(T-s)}
\,{\mathrm d}s\right)t^{-3/2}
+O(t^{-5/2})
\right],\\
u_f(x,t+T)-u_f(x,t)
&=
\frac{xr^{-x}{\mathrm e}^{(\beta+2d)t}}
{2\sqrt{\pi}\,d^{3/2}}
\left[
\left\{
\int_0^T
\left[\sum_{y=1}^{\infty}yf(y,s)r^y\right]
{\mathrm e}^{(\beta+2d)(T-s)}\,{\mathrm d}s
\right\}t^{-3/2}
+O(t^{-5/2})
\right].
\end{align*}
Consequently,
$u(x,t+T)-u(x,t)
=
{\mathrm e}^{(\beta+2d)t}
\left[D_xt^{-3/2}+O(t^{-5/2})\right]$,
where
\begin{align*}
D_x
&=
\frac{xr^{-x}}{2\sqrt{\pi}\,d^{3/2}}
\Bigg[
({\mathrm e}^{(\beta+2d)T}-1)
\sum_{y=1}^{\infty}yu_yr^y
+
\int_0^T
\left(dg(s)+\sum_{y=1}^{\infty}yf(y,s)r^y\right)
{\mathrm e}^{(\beta+2d)(T-s)}\,{\mathrm d}s
\Bigg].
\end{align*}

\item
If $\alpha\gamma<0$, then for $t\to\infty$,
\begin{align*}
|u_a(x,t)|
&=
O\!\left({\mathrm e}^{(\operatorname{Re}\beta)t}t^{-1}\right),\\
|u_g(x,t+T)-u_g(x,t)|
&=
O\!\left({\mathrm e}^{(\operatorname{Re}\beta)t}t^{-1}\right),\\
|u_f(x,t+T)-u_f(x,t)|
&=
O\!\left({\mathrm e}^{(\operatorname{Re}\beta)t}t^{-1}\right).
\end{align*}
Consequently,
$|u(x,t+T)-u(x,t)|
=
O\left({\mathrm e}^{(\operatorname{Re}\beta)t}t^{-1}\right)$.
\end{enumerate}
In particular, $u(x,t+T)-u(x,t)\to0$ for every fixed
$x\in\mathbb N$ whenever
\[
\begin{cases}
\operatorname{Re}\beta+2\sqrt{\alpha\gamma}\leq0,
&\alpha\gamma>0,\\
\operatorname{Re}\beta\leq0,
&\alpha\gamma<0.
\end{cases}
\]
\end{theorem}

\begin{proof}
The solution representation is justified by the symmetrization
argument used in the proof of
Theorem~\ref{nonsymmetric-asymptotic-periodicity}.
That argument applies without the inequalities imposed there
on $\operatorname{Re}\beta$.

Choose $\eta$ such that $\eta^2=\alpha\gamma$, and put
$r=\eta/\alpha$. Then $|r|=\rho$ and $\alpha r=\eta$.
By \eqref{nonsymmetric-sine-kernel},
\[
Q_{x,y}(t)
=
\frac{2r^{y-x}}{\pi}
\int_0^\pi
{\mathrm e}^{(\beta+2\eta\cos\theta)t}
\sin(x\theta)\sin(y\theta)\,{\mathrm d}\theta.
\]
The solution formula and the periodicity identities used in
the proof of Theorem~\ref{nonsymmetric-asymptotic-periodicity}
give
\begin{align*}
u_a(x,t)
&=
\sum_{y=1}^{\infty}u_yQ_{x,y}(t),\\
u_g(x,t+T)-u_g(x,t)
&=
\alpha\int_0^Tg(s)Q_{x,1}(t+T-s)\,{\mathrm d}s,\\
u_f(x,t+T)-u_f(x,t)
&=
\int_0^T\sum_{y=1}^{\infty}
f(y,s)Q_{x,y}(t+T-s)\,{\mathrm d}s.
\end{align*}
Substituting the sine representation into these formulas
and using $\alpha r=\eta$, we obtain
\begin{align*}
u_a(x,t)
&=
\frac{2r^{-x}{\mathrm e}^{\beta t}}{\pi}
\int_0^\pi
{\mathrm e}^{2\eta t\cos\theta}
\sin(x\theta)S_a(\theta)\,{\mathrm d}\theta,\\
u_g(x,t+T)-u_g(x,t)
&=
\frac{2\eta r^{-x}{\mathrm e}^{\beta t}}{\pi}
\int_0^\pi
{\mathrm e}^{2\eta t\cos\theta}
\sin(x\theta)\sin\theta B(\theta)\,{\mathrm d}\theta,\\
u_f(x,t+T)-u_f(x,t)
&=
\frac{2r^{-x}{\mathrm e}^{\beta t}}{\pi}
\int_0^\pi
{\mathrm e}^{2\eta t\cos\theta}
\sin(x\theta)S_f(\theta)\,{\mathrm d}\theta,
\end{align*}
where
\begin{align*}
S_a(\theta)
&:=
\sum_{y=1}^{\infty}u_yr^y\sin(y\theta),\\
B(\theta)
&:=
\int_0^T
g(s){\mathrm e}^{(\beta+2\eta\cos\theta)(T-s)}
\,{\mathrm d}s,\\
S_f(\theta)
&:=
\int_0^T
\left[
\sum_{y=1}^{\infty}f(y,s)r^y\sin(y\theta)
\right]
{\mathrm e}^{(\beta+2\eta\cos\theta)(T-s)}
\,{\mathrm d}s.
\end{align*}
The interchanges of sums and integrals are justified by the
moment assumptions: the series are bounded termwise by the
summable sequences $|u_y|\rho^y$ and
$\rho^y\sup_{0\leq s\leq T}|f(y,s)|$, respectively,
and the exponential factors are bounded for
$\theta\in[0,\pi]$ and $s\in[0,T]$.

\paragraph{The case $\alpha\gamma>0$.}
Choose
\[
\eta=d:=\sqrt{\alpha\gamma}>0,
\qquad
r=\frac d\alpha.
\]
Here $r=\rho$ if $\alpha,\gamma>0$ and $r=-\rho$
if $\alpha,\gamma<0$.
Put
\[
q(\theta):=2(1-\cos\theta).
\]
Since
\[
{\mathrm e}^{\beta t}{\mathrm e}^{2dt\cos\theta}
=
{\mathrm e}^{(\beta+2d)t}{\mathrm e}^{-dtq(\theta)},
\]
the preceding formulas become
\begin{align}
u_a(x,t)
&=
\frac{2r^{-x}{\mathrm e}^{(\beta+2d)t}}{\pi}
\int_0^\pi
{\mathrm e}^{-dtq(\theta)}
\sin(x\theta)S_a(\theta)\,{\mathrm d}\theta,
\label{nsp-rates-initial-sine}\\
u_g(x,t+T)-u_g(x,t)
&=
\frac{2dr^{-x}{\mathrm e}^{(\beta+2d)t}}{\pi}
\int_0^\pi
{\mathrm e}^{-dtq(\theta)}
\sin(x\theta)\sin\theta B(\theta)\,{\mathrm d}\theta,
\label{nsp-rates-boundary-sine}\\
u_f(x,t+T)-u_f(x,t)
&=
\frac{2r^{-x}{\mathrm e}^{(\beta+2d)t}}{\pi}
\int_0^\pi
{\mathrm e}^{-dtq(\theta)}
\sin(x\theta)S_f(\theta)\,{\mathrm d}\theta.
\label{nsp-rates-forcing-sine}
\end{align}

We use the elementary estimate established in the proof of
Theorem~\ref{thm:periodicity-asymptotics}.
Its stronger version, with the time variable replaced by $dt$,
shows that, for every continuous complex-valued function $H$
on $[0,\pi]$ satisfying
\[
H(\theta)=A\theta^2+O(\theta^4)
\qquad\text{as }\theta\downarrow0,
\]
we have
\begin{equation}\label{nsp-rates-gaussian-estimate}
\int_0^\pi
{\mathrm e}^{-dtq(\theta)}H(\theta)\,{\mathrm d}\theta
=
\frac{A\sqrt{\pi}}{4d^{3/2}}t^{-3/2}
+O(t^{-5/2}).
\end{equation}
Here $d>0$ is fixed.
The estimate applies to complex-valued functions by applying
it separately to their real and imaginary parts.

To apply \eqref{nsp-rates-gaussian-estimate}, we determine
the expansions at $\theta=0$ of the factors multiplying
${\mathrm e}^{-dtq(\theta)}$ in the three integrals.
We shall show that
\[
S_a(\theta)=M_a\theta+O(\theta^3),
\qquad
B(\theta)=M_g+O(\theta^2),
\qquad
S_f(\theta)=M_f\theta+O(\theta^3),
\]
where
\begin{align*}
M_a
&:=
\sum_{y=1}^{\infty}yu_yr^y,\\
M_g
&:=
\int_0^Tg(s){\mathrm e}^{(\beta+2d)(T-s)}\,{\mathrm d}s,\\
M_f
&:=
\int_0^T
\left[
\sum_{y=1}^{\infty}yf(y,s)r^y
\right]
{\mathrm e}^{(\beta+2d)(T-s)}\,{\mathrm d}s.
\end{align*}
The inequality $|\sin v-v|\leq|v|^3/6$ gives
\[
|S_a(\theta)-M_a\theta|
\leq
\frac{|\theta|^3}{6}
\sum_{y=1}^{\infty}y^3|u_y|\rho^y,
\]
so
$S_a(\theta)=M_a\theta+O(\theta^3)$.
The analogous estimate for the forcing series is uniform
in $s\in[0,T]$, by the moment assumption on $f$.
Moreover, for $0\leq s\leq T$,
\begin{align*}
&\left|
{\mathrm e}^{(\beta+2d\cos\theta)(T-s)}
-
{\mathrm e}^{(\beta+2d)(T-s)}
\right|=
\left|{\mathrm e}^{(\beta+2d)(T-s)}\right|
\left(1-{\mathrm e}^{-dq(\theta)(T-s)}\right)
\leq
{\mathrm e}^{T\max\{\operatorname{Re}\beta+2d,0\}}
dT\theta^2.
\end{align*}
It follows that
$B(\theta)=M_g+O(\theta^2)$,
$S_f(\theta)=M_f\theta+O(\theta^3)$.
Indeed, the error in $S_f$ is the sum of the cubic error
from the sine series and the product of its $O(\theta)$
bound with the $O(\theta^2)$ error in the exponential.

Using $\sin(x\theta)=x\theta+O(\theta^3)$, we therefore obtain
\begin{align*}
\sin(x\theta)S_a(\theta)
&=
xM_a\theta^2+O(\theta^4),\\
\sin(x\theta)\sin\theta B(\theta)
&=
xM_g\theta^2+O(\theta^4),\\
\sin(x\theta)S_f(\theta)
&=
xM_f\theta^2+O(\theta^4).
\end{align*}
Applying \eqref{nsp-rates-gaussian-estimate} in
\eqref{nsp-rates-initial-sine}--\eqref{nsp-rates-forcing-sine}
gives
\begin{align}
u_a(x,t)
&=
{\mathrm e}^{(\beta+2d)t}
\left[
\frac{xr^{-x}M_a}{2\sqrt{\pi}\,d^{3/2}}
t^{-3/2}+O(t^{-5/2})
\right],
\label{nsp-rates-initial-asymptotic}\\
u_g(x,t+T)-u_g(x,t)
&=
{\mathrm e}^{(\beta+2d)t}
\left[
\frac{xr^{-x}M_g}{2\sqrt{\pi d}}
t^{-3/2}+O(t^{-5/2})
\right],
\label{nsp-rates-boundary-asymptotic}\\
u_f(x,t+T)-u_f(x,t)
&=
{\mathrm e}^{(\beta+2d)t}
\left[
\frac{xr^{-x}M_f}{2\sqrt{\pi}\,d^{3/2}}
t^{-3/2}+O(t^{-5/2})
\right].
\label{nsp-rates-forcing-asymptotic}
\end{align}
Since
$(t+T)^{-3/2}=t^{-3/2}+O(t^{-5/2})$,
the initial-data formula also gives
\[
u_a(x,t+T)-u_a(x,t)
=
{\mathrm e}^{(\beta+2d)t}
\left[
\frac{xr^{-x}M_a}{2\sqrt{\pi}\,d^{3/2}}
({\mathrm e}^{(\beta+2d)T}-1)t^{-3/2}
+O(t^{-5/2})
\right].
\]
Adding the three contributions gives
\begin{equation}\label{nsp-rates-same-sign-conclusion}
u(x,t+T)-u(x,t)
=
{\mathrm e}^{(\beta+2d)t}
\left[D_xt^{-3/2}+O(t^{-5/2})\right],
\end{equation}
where
\begin{align*}
D_x
&=
\frac{xr^{-x}}{2\sqrt{\pi}\,d^{3/2}}
\Bigg[
({\mathrm e}^{(\beta+2d)T}-1)
\sum_{y=1}^{\infty}yu_yr^y
+
\int_0^T
\left(dg(s)+\sum_{y=1}^{\infty}yf(y,s)r^y\right)
{\mathrm e}^{(\beta+2d)(T-s)}\,{\mathrm d}s
\Bigg].
\end{align*}
This proves the first assertion.
If $\operatorname{Re}\beta+2\sqrt{\alpha\gamma}\leq0$,
then $|{\mathrm e}^{(\beta+2d)t}|\leq1$, and therefore
\[
u(x,t+T)-u(x,t)\longrightarrow0.
\]

\paragraph{The case $\alpha\gamma<0$.}
Choose
\[
\eta={\mathrm i}d,
\qquad
d:=\sqrt{|\alpha\gamma|}>0,
\qquad
r=\frac{{\mathrm i}d}{\alpha},
\]
and put
\[
\varphi(\theta):=2d\cos\theta.
\]
The three integral representations at the beginning of the
proof become
\begin{align*}
u_a(x,t)
&=
\frac{2r^{-x}{\mathrm e}^{\beta t}}{\pi}
\int_0^\pi
{\mathrm e}^{{\mathrm i}t\varphi(\theta)}
h_a(\theta)\,{\mathrm d}\theta,\\
u_g(x,t+T)-u_g(x,t)
&=
\frac{2{\mathrm i}dr^{-x}{\mathrm e}^{\beta t}}{\pi}
\int_0^\pi
{\mathrm e}^{{\mathrm i}t\varphi(\theta)}
h_g(\theta)\,{\mathrm d}\theta,\\
u_f(x,t+T)-u_f(x,t)
&=
\frac{2r^{-x}{\mathrm e}^{\beta t}}{\pi}
\int_0^\pi
{\mathrm e}^{{\mathrm i}t\varphi(\theta)}
h_f(\theta)\,{\mathrm d}\theta,
\end{align*}
where
\[
h_a(\theta):=\sin(x\theta)S_a(\theta),
\qquad
h_g(\theta):=\sin(x\theta)\sin\theta B(\theta),
\qquad
h_f(\theta):=\sin(x\theta)S_f(\theta).
\]
Here $S_a$, $B$, and $S_f$ are defined as above, with
$\eta={\mathrm i}d$ and $r={\mathrm i}d/\alpha$.

The series in the definitions of $S_a$ and $S_f$
can be differentiated termwise with respect to $\theta$
up to order three. Indeed, after $j$ differentiations,
their summands are bounded in absolute value by
$y^j|u_y|\rho^y$ and
$y^j\rho^y\sup_{0\leq s\leq T}|f(y,s)|$,
respectively. For $0\leq j\leq3$, both bounds are summable
by the assumptions. Hence the differentiated series
converge uniformly in $\theta$ and, for the forcing,
uniformly in $s\in[0,T]$.
The remaining factors are smooth in $\theta$, and
differentiation under the integrals over $[0,T]$ is justified
by continuity on the compact integration domain.
Consequently, $h_a,h_g,h_f$ belong to $C^3([0,\pi])$.

The derivative
\[
\varphi'(\theta)=-2d\sin\theta
\]
vanishes only at the endpoints $\theta=0,\pi$, and
$\varphi''(0),\varphi''(\pi)\neq0$.
Moreover,
\[
S_a(0)=S_a(\pi)=S_f(0)=S_f(\pi)=0.
\]
It follows from the definitions that each of the three
functions $h_a,h_g,h_f$ satisfies
\[
h(0)=h(\pi)=h'(0)=h'(\pi)=0.
\]

We estimate the integrals directly by integration by parts.
For any of the three functions $h$ above, the quotient
\[
F(\theta):=\frac{h(\theta)}{\varphi'(\theta)},
\qquad 0<\theta<\pi,
\]
has a continuously differentiable extension to $[0,\pi]$,
with $F(0)=F(\pi)=0$.
Indeed, for either endpoint $\theta_0\in\{0,\pi\}$,
\[
\frac{h(\theta)}{\theta-\theta_0}
=
\int_0^1
h'\bigl(\theta_0+v(\theta-\theta_0)\bigr)\,{\mathrm d}v,
\]
and
\[
\frac{\varphi'(\theta)}{\theta-\theta_0}
=
\int_0^1
\varphi''\bigl(\theta_0+v(\theta-\theta_0)\bigr)
\,{\mathrm d}v.
\]
Both right-hand sides extend continuously differentiably
to $\theta=\theta_0$.
Their values there are $h'(\theta_0)=0$ and
$\varphi''(\theta_0)\neq0$, respectively.
Their quotient therefore gives the required extension
near each endpoint.

Thus $h=F\varphi'$ throughout $[0,\pi]$, and integration
by parts gives
\begin{align*}
\int_0^\pi
h(\theta){\mathrm e}^{{\mathrm i}t\varphi(\theta)}
\,{\mathrm d}\theta
&=
\frac{1}{{\mathrm i}t}
\left[
F(\theta){\mathrm e}^{{\mathrm i}t\varphi(\theta)}
\right]_0^\pi-
\frac{1}{{\mathrm i}t}
\int_0^\pi
F'(\theta){\mathrm e}^{{\mathrm i}t\varphi(\theta)}
\,{\mathrm d}\theta=\frac{1}{{\mathrm i}t}
\int_0^\pi
F'(\theta){\mathrm e}^{{\mathrm i}t\varphi(\theta)}
\,{\mathrm d}\theta.
\end{align*}
Since $\varphi$ is real-valued,
\[
\left|
\int_0^\pi
h(\theta){\mathrm e}^{{\mathrm i}t\varphi(\theta)}
\,{\mathrm d}\theta
\right|
\leq
\frac1t\int_0^\pi|F'(\theta)|\,{\mathrm d}\theta
=
O(t^{-1}).
\]
Applying this estimate to the three integrals yields
\begin{gather*}
|u_a(x,t)|
=
O\left({\mathrm e}^{(\operatorname{Re}\beta)t}t^{-1}\right),\\
|u_g(x,t+T)-u_g(x,t)|
+
|u_f(x,t+T)-u_f(x,t)|=
O\left({\mathrm e}^{(\operatorname{Re}\beta)t}t^{-1}\right).
\end{gather*}
Since $T$ is fixed, the estimate for $u_a$ also gives
\[
|u_a(x,t+T)-u_a(x,t)|
=
O\!\left({\mathrm e}^{(\operatorname{Re}\beta)t}t^{-1}\right).
\]
Adding the three contributions, we obtain
\begin{equation}\label{nsp-rates-opposite-sign-conclusion}
|u(x,t+T)-u(x,t)|
=
O\!\left({\mathrm e}^{(\operatorname{Re}\beta)t}t^{-1}\right).
\end{equation}
This proves the second assertion.
If $\operatorname{Re}\beta\leq0$, the exponential factor
is bounded, and hence
\[
u(x,t+T)-u(x,t)\longrightarrow0.\tag*{\qedhere}
\]
\end{proof}

\section{The unified transform method: higher-order examples}\label{sec-8}

We now illustrate how the unified transform method extends
to lattice equations of higher order in the spatial variable,
while retaining a first-order time derivative.
We develop a third-order example in detail, then outline
the corresponding calculations for a fourth-order equation
and for a general equation involving finitely many shifts.
Our purpose is to explain the transform calculations and
the elimination of unknown boundary terms; we do not
establish general existence and uniqueness results here.

\subsection{A third-order example}

We consider a semidiscretization of the linear KdV equation:
\begin{equation}\label{main-problem-11}
\frac{\partial u}{\partial t}(x,t)
+
u(x+2,t)-3u(x+1,t)+3u(x,t)-u(x-1,t)
=0,
\qquad x\in\mathbb N,\quad t\geq0,
\end{equation}
subject to the initial and boundary conditions
\begin{equation}\label{main-problem-12}
u(0,t)=g(t),\qquad t\geq0,
\qquad
u(x,0)=u_x,\qquad x\in\mathbb N.
\end{equation}
Here $g\colon[0,\infty)\to\mathbb R$ is continuous, and
$(u_x)_{x\in\mathbb N}$ is a real sequence.

We proceed formally, assuming that the solution has a discrete
transform
\[
\hat u(z,t):=\sum_{x=1}^{\infty}u(x,t)z^x
\]
that is entire in $z$ for each fixed $t\geq0$, and that
termwise differentiation with respect to $t$ is justified. The residue theorem therefore gives
\[
u(x,t)
=
\frac{1}{2\pi{\mathrm i}}
\int_{|z|=1}\hat u(z,t)\frac{{\mathrm d}z}{z^{x+1}},
\]
where the unit circle is oriented counterclockwise.
Put
\begin{align*}
g_0(t)&:=u(0,t)=g(t),\\
g_1(t)&:=u(1,t)-u(0,t),\\
g_2(t)&:=u(2,t)-2u(1,t)+u(0,t).
\end{align*}
Thus
\[
u(1,t)=g_0(t)+g_1(t),
\qquad
u(2,t)=g_0(t)+2g_1(t)+g_2(t).
\]
The function $g_0$ is prescribed, whereas $g_1$ and $g_2$
are unknown boundary traces. 

The transforms of the shifted functions satisfy
\begin{align*}
\sum_{x=1}^{\infty}u(x+1,t)z^x
&=
z^{-1}\hat u(z,t)-u(1,t),\\
\sum_{x=1}^{\infty}u(x-1,t)z^x
&=
z\hat u(z,t)+zu(0,t),\\
\sum_{x=1}^{\infty}u(x+2,t)z^x
&=
z^{-2}\hat u(z,t)-z^{-1}u(1,t)-u(2,t).
\end{align*}
Multiplying \eqref{main-problem-11} by $z^x$, summing over
$x\geq1$, and using these identities, we obtain
\begin{align*}
\frac{\partial\hat u}{\partial t}(z,t)
&=
\left(z-3+\frac3z-\frac1{z^2}\right)\hat u(z,t)
+zg_0(t)
+\left(\frac1z-3\right)\bigl[g_0(t)+g_1(t)\bigr]
+g_0(t)+2g_1(t)+g_2(t)\\
&=
\frac{(z-1)^3}{z^2}\hat u(z,t)
+\left(z+\frac1z-2\right)g_0(t)
+\left(\frac1z-1\right)g_1(t)
+g_2(t).
\end{align*}
Introducing
\[
w(z):=\frac{(z-1)^3}{z^2},
\qquad z\neq0,
\]
we arrive at the global relation
\begin{equation}\label{third-order-global-relation}
\hat u(z,t)
=
{\mathrm e}^{w(z)t}\hat u(z,0)
+
\left(z+\frac1z-2\right)\tilde g_0(z,t)
+
\left(\frac1z-1\right)\tilde g_1(z,t)
+
\tilde g_2(z,t),
\end{equation}
where
\[
\tilde g_j(z,t)
:=
\int_0^t
g_j(s){\mathrm e}^{w(z)(t-s)}\,{\mathrm d}s,
\qquad j=0,1,2.
\]
Formal inversion therefore gives
\begin{align*}
u(x,t)
&=
\frac{1}{2\pi{\mathrm i}}
\int_{|z|=1}
\Bigg[
{\mathrm e}^{w(z)t}\hat u(z,0)
+\left(z+\frac1z-2\right)\tilde g_0(z,t)
+\left(\frac1z-1\right)\tilde g_1(z,t)
+\tilde g_2(z,t)
\Bigg]\frac{{\mathrm d}z}{z^{x+1}}.
\end{align*}

To obtain a formula involving only the prescribed data, we
must eliminate the contributions from $\tilde g_1$ and
$\tilde g_2$.
As in the second-order case, we use transformations that
leave the phase function $w$ unchanged. 

The equation $w(\phi(z))=w(z)$ is equivalent to the cubic equation for $\phi(z)$,
\[
z^2(\phi(z)-1)^3-(z-1)^3\phi(z)^2=0.
\]
Since $\phi(z)=z$ is one solution, we can factor this equation as
\[
(\phi(z)-z)
\bigl[z^2\phi(z)^2+(1-3z)\phi(z)+z\bigr]=0.
\]
The two remaining solutions are
\[
\phi^\pm(z)
=
\frac{3z-1\pm(1-z)\sqrt{1-4z}}{2z^2}.
\]
Both take the value $1$ at $z=1$.
The square root makes $\phi^\pm$ local branches of an
algebraic function.


Since $w(\phi^\pm(z))=w(z)$, we have
\[
\tilde g_j(\phi^\pm(z),t)=\tilde g_j(z,t),
\qquad j=0,1,2.
\]
Evaluating \eqref{third-order-global-relation} at these points
gives
\begin{align*}
\hat u(\phi^\pm(z),t)
&=
{\mathrm e}^{w(z)t}\hat u(\phi^\pm(z),0)+
\left(\phi^\pm(z)+\frac1{\phi^\pm(z)}-2\right)
\tilde g_0(z,t)+
\left(\frac1{\phi^\pm(z)}-1\right)\tilde g_1(z,t)
+\tilde g_2(z,t),
\end{align*}
where
\[
\hat u(\phi^\pm(z),0)
=
\sum_{x=1}^{\infty}u_x\bigl(\phi^\pm(z)\bigr)^x.
\]
We have obtained a linear system for the unknown
boundary transforms.
To display its solution, put
\begin{align*}
R_\pm(z,t)
&:=
\hat u(\phi^\pm(z),t)
-
{\mathrm e}^{w(z)t}\hat u(\phi^\pm(z),0)-
\left(\phi^\pm(z)+\frac1{\phi^\pm(z)}-2\right)
\tilde g_0(z,t).
\end{align*}
Then
\[
R_\pm(z,t)
=
\left(\frac1{\phi^\pm(z)}-1\right)\tilde g_1(z,t)
+\tilde g_2(z,t).
\]
Whenever $\phi^+(z)\neq\phi^-(z)$, it follows that
\begin{align*}
\tilde g_1(z,t)
&=
\frac{R_+(z,t)-R_-(z,t)}
{1/\phi^+(z)-1/\phi^-(z)},\\
\tilde g_2(z,t)
&=
R_+(z,t)
-
\left(\frac1{\phi^+(z)}-1\right)\tilde g_1(z,t).
\end{align*}
To be able to substitute these formulas
into the global relation \eqref{third-order-global-relation}, we first calculate
\begin{align*}
\left(\frac1z-1\right)\tilde g_1(z,t)+\tilde g_2(z,t)
&=
R_+(z,t)
+
\left(\frac1z-\frac1{\phi^+(z)}\right)
\frac{R_+(z,t)-R_-(z,t)}
{1/\phi^+(z)-1/\phi^-(z)}\\
&=
A_+(z)R_+(z,t)+A_-(z)R_-(z,t),
\end{align*}
where
\[
A_\pm(z):=
\frac{1/z-1/\phi^\mp(z)}
{1/\phi^\pm(z)-1/\phi^\mp(z)}.
\]
To simplify the result of substitution, observe that
\begin{gather*}
A_++A_-=1,
\qquad
\frac{A_+}{\phi^+}+\frac{A_-}{\phi^-}=\frac1z,\\
A_+\phi^++A_-\phi^-
=
\phi^++\phi^--\frac{\phi^+\phi^-}{z}
=
\frac3z-\frac2{z^2},
\end{gather*}
where the last equality follows from Vieta's formulas.
Using these identities and the definitions of $R_\pm$
in the global relation, we obtain
\begin{align}
\hat u(z,t)
&=
{\mathrm e}^{w(z)t}
\left[
\hat u(z,0)
-A_+(z)\hat u(\phi^+(z),0)
-A_-(z)\hat u(\phi^-(z),0)
\right]
\notag\\
&\quad+
\left(z-\frac3z+\frac2{z^2}\right)\tilde g_0(z,t)
+
A_+(z)\hat u(\phi^+(z),t)
+A_-(z)\hat u(\phi^-(z),t).
\label{third-order-eliminated-relation}
\end{align}

The first two terms on the right-hand side depend only on
the prescribed initial and boundary data.
The last two terms still involve the unknown solution.
We shall prove that their combined contribution to the
inversion integral vanishes:
\begin{align}
J
:={}&
\int_{|z|=1}
A_+(z)\hat u(\phi^+(z),t)
\frac{{\mathrm d}z}{z^{x+1}}
+
\int_{|z|=1}
A_-(z)\hat u(\phi^-(z),t)
\frac{{\mathrm d}z}{z^{x+1}}
=0,
\qquad x\in\mathbb N.
\label{third-order-cancellation}
\end{align}
Once this identity is established, inversion of
\eqref{third-order-eliminated-relation} gives the Fokas-type formula
\begin{align}
u(x,t)
&=
\frac1{2\pi{\mathrm i}}
\int_{|z|=1}
{\mathrm e}^{w(z)t}
\left[
\hat u(z,0)
-A_+(z)\hat u(\phi^+(z),0)
-A_-(z)\hat u(\phi^-(z),0)
\right]
\frac{{\mathrm d}z}{z^{x+1}}
\notag\\
&\quad+
\frac1{2\pi{\mathrm i}}
\int_{|z|=1}
\left(z-\frac3z+\frac2{z^2}\right)
\left[
\int_0^t g(s){\mathrm e}^{w(z)(t-s)}\,{\mathrm d}s
\right]
\frac{{\mathrm d}z}{z^{x+1}},
\qquad x\in\mathbb N.
\label{third-order-solution}
\end{align}
Here $\hat u(z,0)=\sum_{y=1}^{\infty}u_yz^y$ is the
prescribed initial transform. The combinations involving $\phi^+$ and $\phi^-$ are
understood by continuous extension at $z=1$, where the
two values coincide; the apparent singularities there
are removable.

We now fix $t\geq0$ and prove
\eqref{third-order-cancellation} under the analyticity
assumption on $\hat u(z,t)$ stated above.
The idea is to expand the unknown transform in its Taylor
series and show that each term makes zero contribution
to the contour integral.
For $n\in\mathbb N_0$, put
\[
T_n(z):=
A_+(z)\phi^+(z)^n+A_-(z)\phi^-(z)^n.
\]
The identities established above give
\[
T_0(z)=1,
\qquad
T_1(z)=\frac3z-\frac2{z^2}.
\]
Moreover, both $\phi^+(z)$ and $\phi^-(z)$ satisfy
\[
z^2\phi(z)^2+(1-3z)\phi(z)+z=0,
\]
or equivalently
\[
\phi(z)^2
=
\left(\frac3z-\frac1{z^2}\right)\phi(z)-\frac1z.
\]
Multiplying this identity by $\phi(z)^{n-2}$, applying it
to each of $\phi^\pm(z)$, and taking the linear combination
with coefficients $A_\pm(z)$, we obtain
\[
T_n(z)
=
\left(\frac3z-\frac1{z^2}\right)T_{n-1}(z)
-\frac1zT_{n-2}(z),
\qquad n\geq2.
\]
It follows by induction that every $T_n(z)$ is a polynomial
in $z^{-1}$.
In particular, for $x\in\mathbb N$,
\[
\int_{|z|=1}T_n(z)\frac{{\mathrm d}z}{z^{x+1}}=0,
\qquad n\in\mathbb N_0,
\]
because the integrand has no $z^{-1}$ term.

Since $\hat u(\cdot,t)$ is entire, its Taylor expansion gives
\[
\hat u(\phi^\pm(z),t)
=
\sum_{n=1}^{\infty}u(n,t)\phi^\pm(z)^n.
\]
Consequently,
\[
A_+(z)\hat u(\phi^+(z),t)
+
A_-(z)\hat u(\phi^-(z),t)
=
\sum_{n=1}^{\infty}u(n,t)T_n(z).
\]
This series converges uniformly on $|z|=1$.
Indeed, the values $\phi^\pm(z)$ remain in a bounded set,
and the Taylor series of an entire function converges
absolutely and uniformly on every closed disk.
The coefficients $A_\pm(z)$ are also bounded on the circle,
with their removable limits understood at $z=1$.
Interchanging the two branches merely interchanges the two
terms and does not affect their sum.

Termwise integration is therefore justified, and
\[
J
=
\sum_{n=1}^{\infty}u(n,t)
\int_{|z|=1}T_n(z)\frac{{\mathrm d}z}{z^{x+1}}
=
0.
\]
This proves \eqref{third-order-cancellation} and completes
the derivation of \eqref{third-order-solution}.

\subsection{A fourth-order example}

We consider the lattice equation
\begin{equation}\label{fourth-order-equation}
\frac{\partial u}{\partial t}(x,t)
=
a u(x+2,t)+b u(x+1,t)+c u(x-1,t)+d u(x-2,t),
\qquad x\in\mathbb N,\quad t\geq0,
\end{equation}
where $a,b,c,d$ are nonzero real constants.
Because the equation involves $u(x-2,t)$, we prescribe two
boundary values:
\begin{equation}\label{fourth-order-data}
u(-1,t)=g_{-1}(t),\qquad
u(0,t)=g_0(t),\qquad t\geq0,\qquad
u(x,0)=u_x,\qquad x\in\mathbb N.
\end{equation}
Here $g_{-1},g_0\colon[0,\infty)\to\mathbb R$ are continuous.

In this subsection, we use the notation
\[
g_1(t):=u(1,t),
\qquad
g_2(t):=u(2,t).
\]
Thus, unlike in the third-order example, $g_1$ and $g_2$
denote the values themselves rather than their differences.
These two functions are unknown, whereas $g_{-1}$ and $g_0$
are prescribed.

As before, we proceed formally, assuming the convergence and
regularity required for the transform calculations. We use the discrete transform
\[
\hat u(z,t):=\sum_{x=1}^{\infty}u(x,t)z^x
\]
and the inversion formula
\[
u(x,t)
=
\frac{1}{2\pi{\mathrm i}}
\int_{|z|=1}
\hat u(z,t)\frac{{\mathrm d}z}{z^{x+1}},
\qquad x\in\mathbb N,
\]
where the unit circle is oriented counterclockwise.
For $z\neq0$, the shifted sums can be expressed in terms of $\hat u$:
\begin{align*}
\sum_{x=1}^{\infty}u(x+1,t)z^x
&=
z^{-1}\hat u(z,t)-u(1,t),\\
\sum_{x=1}^{\infty}u(x-1,t)z^x
&=
z\hat u(z,t)+zu(0,t),\\
\sum_{x=1}^{\infty}u(x+2,t)z^x
&=
z^{-2}\hat u(z,t)-z^{-1}u(1,t)-u(2,t),\\
\sum_{x=1}^{\infty}u(x-2,t)z^x
&=
z^2\hat u(z,t)+z^2u(0,t)+zu(-1,t).
\end{align*}
Multiplying \eqref{fourth-order-equation} by $z^x$ and summing
over $x\geq1$, we obtain
\begin{align*}
\frac{\partial\hat u}{\partial t}(z,t)
&=
a\bigl[z^{-2}\hat u(z,t)-z^{-1}g_1(t)-g_2(t)\bigr]\\
&\quad+
b\bigl[z^{-1}\hat u(z,t)-g_1(t)\bigr]
+c\bigl[z\hat u(z,t)+zg_0(t)\bigr]+
d\bigl[z^2\hat u(z,t)+z^2g_0(t)+zg_{-1}(t)\bigr]\\
&=
w(z)\hat u(z,t)
-a g_2(t)-\left(\frac az+b\right)g_1(t)
+(cz+dz^2)g_0(t)+dzg_{-1}(t),
\end{align*}
where
\[
w(z):=\frac a{z^2}+\frac bz+cz+dz^2.
\]
Solving this first-order linear equation  gives
\begin{align}
\hat u(z,t)
&=
{\mathrm e}^{w(z)t}\hat u(z,0)
-a\tilde g_2(z,t)
-\left(\frac az+b\right)\tilde g_1(z,t)
+
(cz+dz^2)\tilde g_0(z,t)
+dz\tilde g_{-1}(z,t),
\label{fourth-order-global-relation}
\end{align}
where
\[
\tilde g_j(z,t)
:=
\int_0^t
g_j(s){\mathrm e}^{w(z)(t-s)}\,{\mathrm d}s,
\qquad j=-1,0,1,2.
\]
To eliminate the unknown functions $\tilde g_1$
and $\tilde g_2$, we look for symmetries $\phi$ such that $w(\phi(z))=w(z)$, i.e.,
\[
\frac a{\phi(z)^2}+\frac b{\phi(z)}
+c\phi(z)+d\phi(z)^2
=
\frac a{z^2}+\frac bz+cz+dz^2.
\]
After multiplication by $z^2\phi(z)^2$, the equation factors as
\[
(\phi(z)-z)
\left[
dz^2\phi(z)^3
+(dz^3+cz^2)\phi(z)^2
-(bz+a)\phi(z)-az
\right]
=0.
\]
Thus, besides the trivial solution $\phi(z)=z$, the
remaining three values satisfy a cubic equation.
Away from points where roots coincide, they can be viewed
as locally analytic branches of an algebraic function.

The symmetries become particularly transparent when
$a=d$ and $b=c$. In this case,
\[
w(z)
=
a\left[\left(z+\frac1z\right)^2-2\right]
+b\left(z+\frac1z\right),
\]
and hence
\begin{align*}
w(\phi(z))-w(z)
&=
\left(
\phi(z)+\frac1{\phi(z)}-z-\frac1z
\right)
\left[
a\left(
\phi(z)+\frac1{\phi(z)}+z+\frac1z
\right)+b
\right].
\end{align*}
The first factor gives the two symmetries
$\phi(z)=z$ and $\phi(z)=z^{-1}$.
The remaining two values satisfy
\[
\phi(z)^2+
\left(z+\frac1z+\frac ba\right)\phi(z)+1=0.
\]
Their product is $1$, so they are reciprocal to one another.
The geometry of their images of the unit circle depends
on the coefficients.




For every $\phi$ satisfying $w(\phi(z))=w(z)$, we have
\[
\tilde g_j(\phi(z),t)=\tilde g_j(z,t),
\qquad j=-1,0,1,2.
\]
Where $\hat u(\phi(z),t)$ is defined, evaluation of
\eqref{fourth-order-global-relation} at $\phi(z)$ gives
\begin{align*}
\hat u(\phi(z),t)
&=
{\mathrm e}^{w(z)t}\hat u(\phi(z),0)
-a\tilde g_2(z,t)
-\left(\frac a{\phi(z)}+b\right)\tilde g_1(z,t)+
\bigl[c\phi(z)+d\phi(z)^2\bigr]\tilde g_0(z,t)
+d\phi(z)\tilde g_{-1}(z,t),
\end{align*}
where
\[
\hat u(\phi(z),0)
=
\sum_{x=1}^{\infty}u_x\phi(z)^x.
\]
Two distinct roots provide two linear equations for
$\tilde g_1$ and $\tilde g_2$.
The coefficient matrix is non-singular because $a\neq0$
and the two roots are distinct and nonzero.
Solving this system expresses $\tilde g_1$ and $\tilde g_2$
in terms of the prescribed data and the transformed values
$\hat u(\phi(z),t)$.
Substituting these expressions into the inversion formula
is the next step of the unified transform method.

To obtain a representation involving only the prescribed
data, one must still show that the remaining terms involving
the unknown transformed solution make no contribution to
the inversion integral.
In the third-order example, this cancellation followed from
an algebraic recurrence for combinations of powers of the
phase symmetries.
For the present equation, a corresponding cancellation
argument remains to be established; its form may depend
on the coefficients and on the analytic properties of the
transformed solution.
We do not pursue these details here.

\subsection{The general case}

The same algebraic procedure applies to a general
lattice equation
\[
\frac{\partial u}{\partial t}(x,t)
=
\sum_{j=-p}^{q}a_j u(x+j,t),
\qquad x\in\mathbb N,\quad t\geq0,
\]
where $p,q\geq1$ and $a_{-p}a_q\neq0$.
The $p$ boundary values
$u(1-p,t)$, \dots, $u(0,t)$,
are prescribed, while the transform calculation
introduces the $q$ unknown functions
$u(1,t)$, \dots, $u(q,t)$.
The phase function is
\[
w(z)=\sum_{j=-p}^{q}a_jz^{-j},
\]
which has a pole of order $q$ at zero.
After multiplication by $\phi(z)^q$, the equation
$w(\phi(z))=w(z)$ becomes a polynomial equation of degree
$p+q$ in $\phi(z)$, with the known root $\phi(z)=z$.

Evaluating the global relation at suitable additional roots
provides equations for the unknown boundary transforms.
As in the preceding examples, completing the method requires
justifying the elimination of the remaining terms involving
the unknown transformed solution from the inversion integral.

\section{Conclusion}\label{sec-9}

We have studied initial-boundary value problems for linear
lattice differential equations using reflection, the Laplace
transform, and the unified transform method. The resulting
contour integral and Bessel-function representations provide
both explicit solution formulas and tools for qualitative
analysis. For the semidiscrete diffusion equation, we have justified
the formulas for initial data of polynomial growth and
continuous boundary data, treated spatially bounded forcing,
and established spatial estimates and long-time properties.
The extensions to Robin boundary conditions, nonsymmetric
coefficients, and a class of time-dependent coefficients
illustrate the scope of these approaches.

Several directions remain open for further investigation.
One is a systematic treatment of initial-boundary value
problems involving higher-order spatial differences.
The uniqueness result in Theorem~\ref{uniquenessForward}
and the fundamental solutions obtained in
Section~\ref{section-symmetric-FS} provide a starting point.
The higher-order examples presented in Section \ref{sec-8} outline how the
unified transform method can be used to derive solution
representations. A fuller treatment would include precise
conditions for the contour deformations and elimination
of unknown boundary transforms, rigorous justification
of the resulting formulas, and analysis of the solutions.

Our long-time results concern fixed lattice points.
It would also be useful to investigate asymptotic behavior
along sequences for which $x/t\to c>0$ as $t\to\infty$.
Such results would describe propagation at nonzero velocities
and complement the fixed-position limits studied here.

Another direction is to replace continuous time by discrete
time or, more generally, by a time scale containing both
continuous and discrete parts. An explicit solution formula
for a purely discrete initial-boundary value problem was
obtained in \cite{Ander}. It would be interesting to develop
corresponding qualitative results and determine how the
choice of the time domain affects the long-time behavior.
Related questions arise for equations involving Stieltjes
derivatives; a continuous-space setting was recently
investigated in~\cite{ST}.

Finally, one may consider fractional derivatives in time
or fractional powers of spatial difference operators.
For nonlocal spatial operators, the formulation of the
boundary conditions itself requires additional care.
Extending the solution representations and asymptotic
analysis to these problems would provide a further
connection between lattice equations and their
continuous counterparts.

\bigskip

\noindent{\bf Acknowledgement.} Andreas Chatziafratis is supported by JSPS KAKENHI Grant Number 24H00024 and JST [Moonshot R\&D][Grant Number JPMJMS25A5].

\end{document}